\documentclass[12pt]{amsbook}   	
\usepackage{geometry}                		
\usepackage{graphicx}				
\usepackage{caption}
\usepackage{subcaption}
\usepackage{amssymb}
\usepackage{amsthm}
\usepackage{csquotes}
\usepackage{amsmath}
\usepackage{listings}
\usepackage{mathtools}
\usepackage{float}
\usepackage{physics}
\MakeOuterQuote{"}
\newtheorem{theorem}{Theorem}
\newtheorem*{theorem*}{Theorem}
\newtheorem{definition}{Definition}
\newtheorem{proposition}{Proposition}
\newtheorem{corollary}{Corollary}
\newtheorem{lemma}{Lemma}
\newtheorem*{lemma*}{Lemma}
\newtheorem{remark}{Remark}
\newtheorem{conjecture}{Conjecture}
\newtheorem*{conjecture*}{Conjecture}
\newcommand{\cal}{\mathcal}

\title{Some New Results in Geometric Analysis}
\author{Matei P. Coiculescu}
\begin{document}
\maketitle
\frontmatter
\tableofcontents
\mainmatter
\chapter{Prelude}

This thesis presents three main results in geometric analysis. Chapter 2 is concerned with the evolution of figure-eight curves by curve-shortening flow. Chapter 3 presents our work on the geometric flow on space curves that is point-wise curvature preserving. We find all stationary solutions of this flow and show that the stationary solutions corresponding to helices are $L^2$ linearly stable. Most of this chapter has been published as an article in Archiv der Mathematik \cite{COI2}, but we have added a section here on numerical solutions of the key partial differential equation. Chapter 4 presents an abridgment of our work on a family of Lie groups that, when equipped with canonical left-invariant metrics, interpolates between the Sol geometry and a model of Hyperbolic space. The complete version of our work will appear in Experimental Mathematics, and a preprint can be found on the ArXiv as well \cite{COI}.

The common thread running through our work is the geometric analysis of differential equations. That is to say: the obtainment of information about geometric objects by the analysis of (usually nonlinear) differential equations. Deep and difficult work on the Ricci flow and the behavior of minimal submanifolds are some of the most famous parts of geometric analysis, but we will deal with simpler and more concrete problems in this thesis.

\begin{center}
\textbf{Curve-Shortening Flow on Figure-Eight Curves}
\end{center}

Let $C(t): S^1\cross [0,T) \to \mathbb{R}^2$ be a family of immersed curves in the plane. We say that the family is evolving according to curve shortening flow (CSF) if and only if at any $(u,t) \in S^1\cross [0,T)$ we have
$$\frac{\partial C}{\partial t} = k N$$
where $k$ is the curvature and $N$ is the unit normal vector of the immersed curve $C(\cdot,t)$. For any given smoothly immersed initial curve $C(0)$, the existence of a corresponding family of curves $C(t)$ undergoing CSF is guaranteed, and the proof of short-term existence can be found in \cite{GH}. Due to the work of Gage, Hamilton, and Grayson in \cite{GH} and \cite{G}, we know that CSF shrinks any smoothly embedded initial curve to a point. We also know that, after the appropriate rescaling, any embedded curve becomes a circle in the limit. We will start this chapter by presenting some basic lemmas about the CSF, proven first in \cite{ANG}, \cite{GH}, or \cite{G}.

Later in the chapter, we focus on the asymptotic results about balanced figure-eight curves (having equal-area lobes) obtained by Grayson in \cite{G3}. Namely, we will present Grayson's proof that the isoperimetric ratio of a balanced figure-eight blows up in the singularity. Then, we will present Angenent's main result from \cite{ANG} that characterizes the blowup set of collapsing lobes. We finish this chapter with a discussion of a related problem communicated from Richard Schwartz concerning the evolution of figure-eight curves that have only one inflection point, are four-fold symmetric, and have additional monotonicity assumptions on curvature. These curves, which we term \textbf{Concinnous}, will be the main object of our study in the latter portion of Chapter 2. We will offer a method towards proving the following result, which describes the limiting shape of a concinnous figure-eight curve after evolution by CSF and an affine rescaling. Qualitatively, if we rescale the curve so that it has aspect ratio equal to $1$, the shape will limit to that of a bowtie. In a subsequent paper, Schwartz and I prove the following theorem:
\begin{theorem*}
Concinnous figure-eight curves converge to a point under CSF, and, after affine-rescaling, concinnous curves converge to a bow-tie shape.
\end{theorem*}
In this thesis, I will sketch some ideas in the proof of this result.
\begin{center}
\textbf{The Curvature Preserving Flow on Space Curves}
\end{center}

Substantial work has been done towards understanding geometric flows on curves immersed in Riemannian Manifolds. For example, the author of \cite{H} found that the vortex filament flow is equivalent to the non-linear Schr\"odinger equation, which enabled the discovery of explicit soliton solutions. Recently, geometric evolutions that are integrable, in the sense of admitting a Hamiltonian structure, have been of interest. The authors of \cite{I} and \cite{BSW} analyze the integrability of flows in Euclidean space and Riemannian Manifolds, respectively. In this chapter, we study the following geometric flow for curves in $\mathbb{R}^3$ with strictly positive torsion that preserves arc-length and curvature:
$$X_t=\frac{1}{\sqrt{\tau}}\textbf{B}.$$ 
Hydrodynamic and magnetodynamic motions related to this geometric evolution equation have been considered, as mentioned in \cite{SR}. The case when curvature is constant demonstrates a great deal of structure. After rescaling so that the curvature is identically $1$, the evolution equation for torsion is given by:
$$\tau_t = D_s\big(\tau^{-1/2}-\tau^{3/2}+D_s^2(\tau^{-1/2})\big).$$
This flow has been studied since at least the publication of \cite{SR}. The authors of \cite{SR} show that the above equation, which they term the \textit{extended Dym equation}, is equivalent to the m$^2$KDV equation. In addition, they present auto-B\"acklund transformations and compute explicit soliton solutions. In Chapter 3, we continue the investigation of the curvature-preserving geometric flow.

We will characterize the flow $X_t=\frac{1}{\sqrt{\tau}}\textbf{B}$ as the unique flow on space curves that is both curvature and arc-length preserving, and we will provide another proof that this flow is equivalent to the m$^2$KDV equation. Afterwards, we will prove the global existence and uniqueness of solutions of the flow (for periodic $C^\infty$ initial data).

Then, we will concern ourselves with the stationary solutions to the geometric flow in the case of constant curvature. Namely, we derive the linearization of the evolution equation for torsion around explicit stationary solutions, and prove $L^2(\mathbb{R})$ stability of the linearization in the case of constant torsion (or for helices).

In addition to our work in \cite{COI2}, we present some numerical work done with Mathematica related to the evolution equation for torsion.

\begin{center}
\textbf{An Interpolation from Sol to Hyperbolic Space}
\end{center}

As mentioned earlier, this chapter presents an abbreviation of our work in \cite{COI}, whose complete version will appear in Experimental Mathematics. We study a one-parameter family of homogeneous Riemannian 3-manifolds that interpolates between three Thurston geometries: Sol, $\mathbb{H}^2\cross\mathbb{R}$, and $\mathbb{H}^3$. Sol is quite strange from a geometric point of view. For example, it is neither rotationally symmetric nor isotropic, and, since Sol has sectional curvature of both signs, the Riemannian exponential map is singular. In this article, we attempt to show that Sol's peculiarity can be slowly untangled by an interpolation of geometries until we reach $\mathbb{H}^2\cross\mathbb{R}$, which has a qualitatively "better" behavior than Sol. The interpolation continues on to $\mathbb{H}^3$, but we will not spend much time with that part of the family. 

Our family of Riemannian 3-manifolds arises from a one-parameter family of solvable Lie groups equipped with canonical left-invariant metrics. We denote the groups by $G_\alpha$ with $-1\leq\alpha\leq 1$. Each $G_\alpha$ is the semi-direct product of $\mathbb{R}$ with $\mathbb{R}^2$, with the following operation on $\mathbb{R}^3$:
$$(x,y,z)\ast(x',y',z')=(x'e^z+x, y'e^{-\alpha z}+y, z'+z).$$
Then $\mathbb{R}^3$, which is only the underlying set, can be equipped with the following left-invariant metric:
$$ds^2=e^{-2z}dx^2+e^{2\alpha z}dy^2+dz^2.$$
These $G_\alpha$ groups, when endowed with the canonical metric, perform our desired interpolation: linking the familiar and the unfamiliar. It is a natural question to analyze what happens in between.

For positive $\alpha$, it happens that typical geodesics starting at the origin in $G_\alpha$ spiral around certain cylinders. For each such geodesic, there is an associated period that determines how long it takes for it to spiral exactly once around. We denote the function determining the period by $P$, and we will show that it is a function of the initial tangent vector. We call a geodesic segment $\gamma$ of length $T$ $\textit{small, perfect,}$ or $large$ whenever $T<P_\gamma,$  $T=P_\gamma,$ or $T>P_\gamma,$ respectively. Our primary aim is the following conjecture:

\begin{conjecture*}
  A geodesic segment in $G_\alpha$ is a length minimizer
  if and only if it is small or perfect.
\end{conjecture*}

The above conjecture is already known for $G_1$, or Sol, and was proven in \cite{MS}. In this article, we will reduce the general conjecture to obtaining bounds on the derivative of the period function by proving the \textbf{Bounding Box Theorem}. In particular, we will define a certain curve, and a key step in our would-be proof of the main conjecture is that a portion of this curve should be the graph of a monotonically decreasing function. The main obstacle to proving the whole conjecture is this monotonicity result. We are able to show the desired monotonicity  for the group $G_{1/2}$ with our \textbf{Monotonicity Theorem} because we have found an explicit formula for the period function in this case. The group $G_{1/2}$ maximizes scalar curvature in our family, which offers more credence that it is truly a "special case" along with Sol. Thus, our main theorem is a proof of the conjecture for $\alpha=1/2$:
\begin{theorem*}
A geodesic segment in $G_{1/2}$ is a length minimizer
  if and only if it is small or perfect.
\end{theorem*}

We have sufficient information about the period function for the group $G_{1/2}$ to extend the main result from \cite{MS}, obtaining a characterization of the cut locus and its consequences for geodesic spheres. 
\begin{center}
\textbf{Acknowledgments}
\end{center}

I would first like to thank my advisor, Professor Richard Schwartz, for his encouragement and guidance during the last four years. He has taught me mathematics and a great way of seeing the beauty in mathematics.

I would like to thank all of my professors at Brown University for their openness, friendliness, and erudition in mathematics. I would especially like to express gratitude to Professors Benoit Pausader, Georgios Daskalopoulos, and Jeremy Kahn.

Lastly, I would like to thank my parents for their love and resolute support, ever since I was born.

\chapter{Curve-Shortening Flow on Figure-Eight Curves}
\section{Preliminaries}
Let $C(t): S^1\cross [0,T) \to \mathbb{R}^2$ be a family of immersed curves in the plane. We say that the family is evolving according to curve shortening flow (CSF) if and only if at any $(u,t) \in S^1\cross [0,T)$ we have
$$\frac{\partial C}{\partial t} = k N$$
where $k$ is the curvature and $N$ is the unit normal vector of the immersed curve $C(\cdot,t)$. A curvature function $k$ determines a plane curve up to Euclidean isometries. Thus, it suffices to understand the evolution of the curvature in order to understand the evolution of the curve.  

The partial differential equation governing the evolution of curvature $k$ is given by:
\begin{equation}\frac{\partial k}{\partial t} = \frac{\partial^2 k}{\partial s^2} + k^3\end{equation}
where the derivatives in equation $(1)$ are with respect to arc-length. Remembering that the curve is shortening in length, we see that equation $(1)$ is a non-linear PDE. In many applications, however, it is more useful to use a modification of CSF that preserves some geometric quantity. Examples include the area preserving flow, the flow fixing $x$-coordinates, and the flow fixing tangent angles. First let us review the flow fixing $x$-coordinates. 

We choose cartesian coordinates in the plane and rotate the evolution to fix $x$-coordinates. This yields a different flow that has the same point-set curves as solutions, but a different evolution equation for curvature. Equations describing this flow are found in the following lemma, proven in \cite{G} inter alia.
\newpage
\begin{lemma}[\cite{G}]
Choose coordinates so that $C(t)$ is locally a graph: $(x,y(x,t))$. If $'$ denotes differentiation w.r.t. $x$ and if $\theta$ is the angle that tangents to the curve make with the $x$ axis (i.e. $\theta(x,t)=\arctan{y'(x,t)}$), then
$$\frac{\partial y}{\partial t} = \frac{y''}{1+y'^2} \quad \textrm{ and }\quad \frac{\partial \theta}{\partial t} =\theta''\cdot \cos^2{\theta}$$
\end{lemma}
\begin{proof}
Since 
$$\frac{\partial C}{\partial t} = k N$$
the speed of the evolution of the curve in its normal direction is $k$. Adjusting for the speed in the vertical direction (after having chosen coordinates) gets us
$$\frac{\partial y}{\partial t} = k \sec{\theta}.$$
Away from vertical tangents, we know that
$$k = \frac{y''}{(1+y'^2)^{3/2}} \quad\textrm{ and } \quad \sec{\theta} = \sqrt{1+y'^2}$$
so we have derived the evolution equation for $y$. The equation for $\theta$ is also derived easily and is strictly parabolic away from vertical tangents. 
\end{proof}
\begin{figure}[h!]
\centering
\includegraphics[width=1\textwidth]{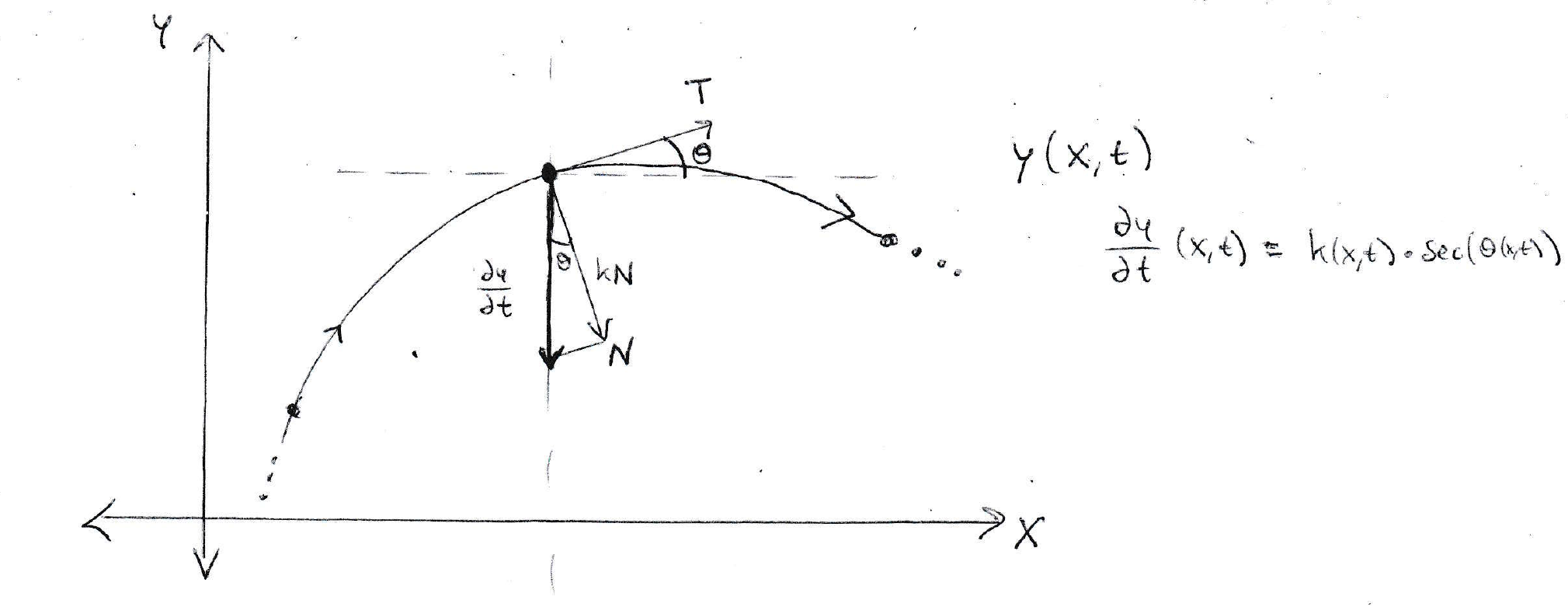}
\caption{The local geometry of the $x$-coordinate preserving modification.}
\end{figure}
We also need to understand how the area enclosed by a curve decays with CSF.
\begin{lemma}[Also found in \cite{G}]
Let $A$ be the area bounded by a curve, the $x$-axis, and two vertical lines that intersect the curve at $x=a$ and $x=b$. Then 
$$\frac{dA}{dt}= \int_a^b k\hspace{3pt} ds$$
\end{lemma}
\begin{proof}
This follows by differentiation of an integral:
$$\frac{dA}{dt} = \frac{d}{dt} \int_a^b y\hspace{3pt}dx = \int_a^b \frac{y''}{1+y'^2} dx = \int_a^b k \hspace{3pt}ds.$$
\end{proof}
Another modification to CSF is obtained by rotating to fix tangents to the curve (i.e. fixing $\theta$), which can be done away from inflection points. Using the chain rule, one can show that determining the evolution by CSF is equivalent to solving the following PDE:
\begin{lemma}[See \cite{ANG} or \cite{GH}]
For the CSF modification that fixes tangents, the evolution equation for curvature is 
$$\frac{\partial k}{\partial t} = k^2 \frac{\partial^2 k}{\partial \theta^2}+k^3$$
\end{lemma}
We will be particularly interested in the evolution by CSF of the simplest non-embedded plane curves: figure-eights.
\begin{definition}
A smooth curve $C$ immersed in the plane is a \textbf{Figure-Eight} if and only if
\begin{itemize}
\item It has exactly one double point.
\item It has total rotation number $\int_C k ds =0$.
\end{itemize}
\end{definition}
\begin{definition}
A figure-eight curve $C$ is called \textbf{Balanced} if and only if its lobes bound regions of equal area.
\end{definition}
\begin{definition}
A balanced figure-eight curve is termed \textbf{Concinnous} if and only if
\begin{itemize}
\item Its double point is its only inflection point.
\item It is reflection symmetric about two perpendicular axes intersecting at its double point. 
\item Suppose we parametrize so that $k>0$ in a lobe. Then, $k_\theta>0$ in the interior of that lobe and $k_{\theta\theta}<0$ in the interior of the appropriate half-lobe.
\end{itemize}
The axis of symmetry that only intersects the curve at its double point is called the \textbf{Major Axis}, and the other is called the \textbf{Minor Axis}.
\end{definition}
\begin{figure}
\centering
\includegraphics[width=0.5\textwidth]{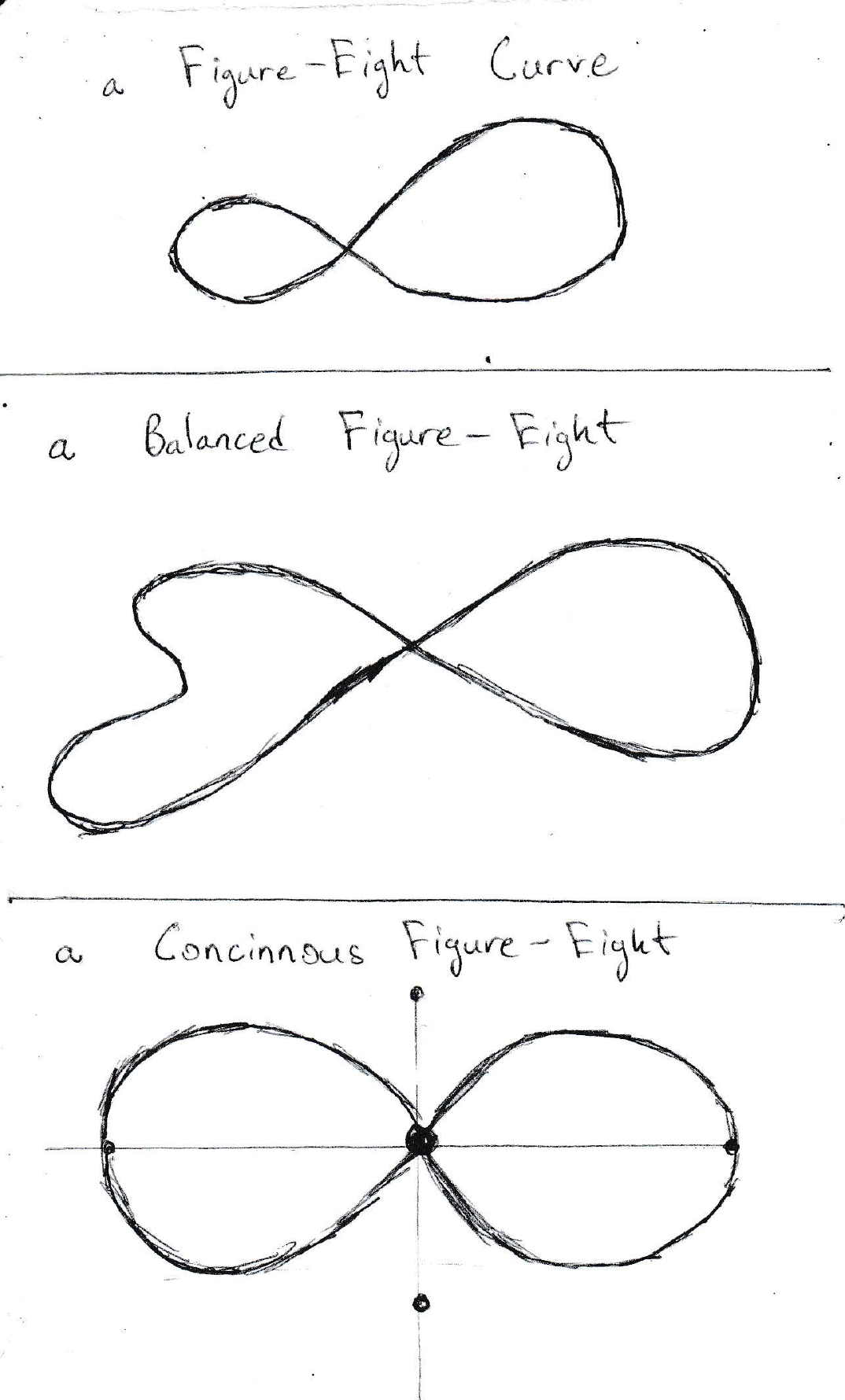}
\caption{Some typical examples for each class of figure-eight curve}
\end{figure}
\newpage
\section{Balanced Figure-Eights}
Figure-eights that are not balanced evolve until one lobe collapses to a point. On the other hand, balanced figure-eights will evolve until the enclosed area converges to zero \cite{G3}. There are two possibilities here: either the balanced figure-eight converges to a point or to some line segment. It is still unknown whether all balanced figure-eights behave in the same manner.
\begin{conjecture}[Grayson]
Every balanced figure-eight converges to a point under the curve-shortening flow.
\end{conjecture}

For a figure-eight, the function $\theta$ is single-valued, and the image of the figure-eight in the $x\theta$-plane is a closed curve with well defined extrema. Moreover, because the evolution of $\theta$ is strictly parabolic (cf. Lemma 1), $\theta_{max}$ is strictly decreasing and $\theta_{min}$ is strictly increasing. By comparing the rate at which the extrema of $\theta$ approach each other with the diminution of the enclosed area, we can get a quantitative statement about the shape of a balanced figure-eight under CSF. The key to the proof of the main theorem in this section is the oft-employed maximum principle for strictly parabolic evolution equations.
\begin{figure}[H]
\centering
\includegraphics[width=0.55\textwidth]{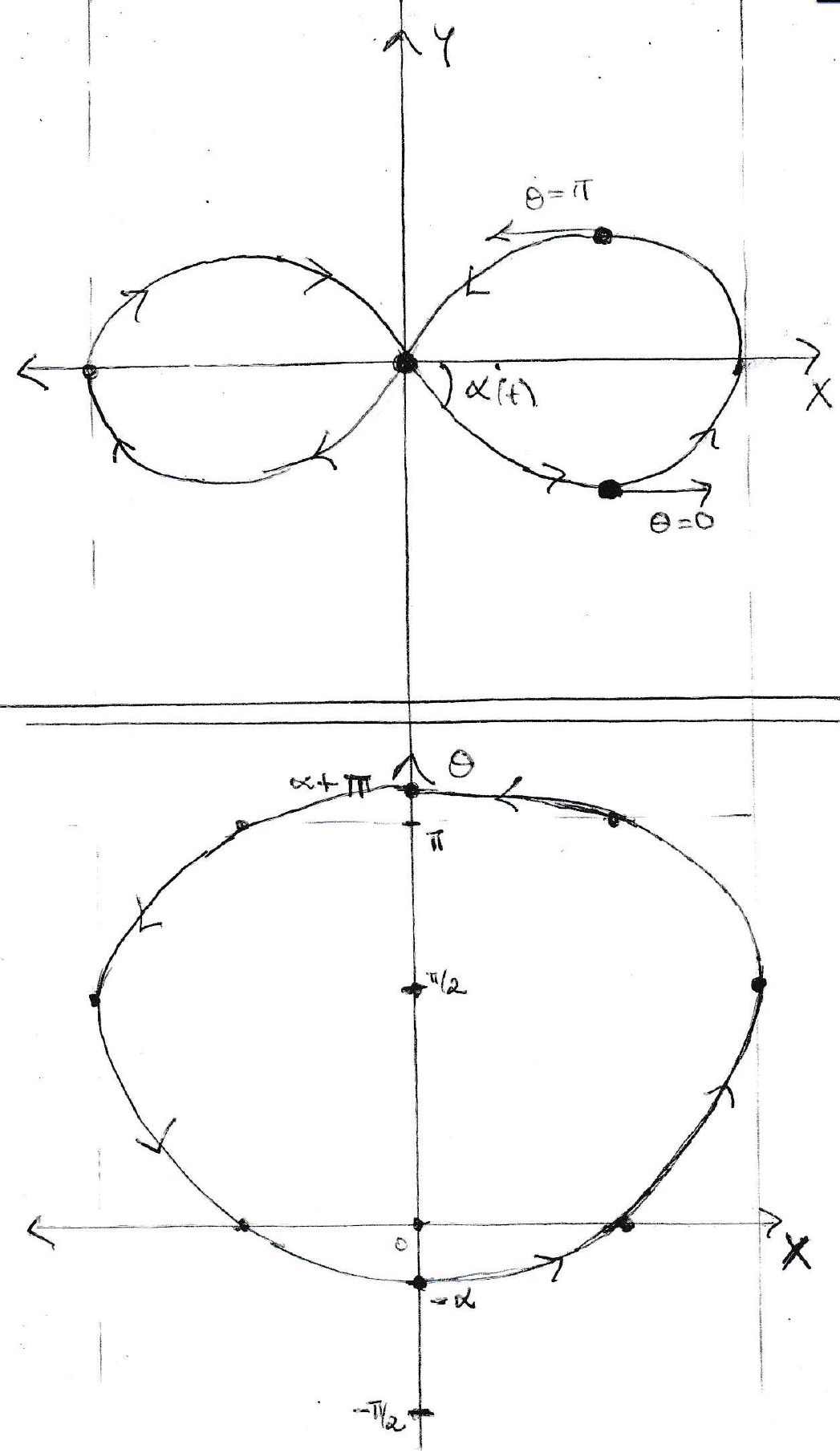}
\caption{The function $\theta(x,t)$ gives us an evolution of closed curves, evolving according to the equation in Lemma 1.}
\end{figure}
\begin{lemma}[The Maximum Principle as used in \cite{G}]
Suppose $F(x,t): [0,\epsilon]\cross [t_0, t_0 +\epsilon] \to\mathbb{R}$ satisfies a strictly parabolic evolution w.r.t time. Assume also that $F$ is not identically zero at $t_0$, $F(x,t_0)\geq 0$ for all $x$, and that $F(0,t), F(\epsilon, t) \geq 0$ for all $t$. Then
$$F(x,t)>0 \quad \textrm{ for all } \quad (x,t) \in (0,\epsilon) \cross (t_0, t_0 + \epsilon]$$
\end{lemma}
\begin{theorem}[Grayson, \cite{G3}]
When $C(0)$ is a balanced figure-eight, the isoperimetric ratio converges to $\infty$ in the singularity.
\end{theorem}
\begin{proof}
We shall use the modification of CSF that preserves $x$-coordinates. We also choose cartesian coordinates so that $\theta_{min}(t=0)$ is zero and the double point is at the origin. Without loss of generality, the initial curve encloses unit area. Now, the interior angle at the double point of our figure-eight is bounded between $0$ and $\pi$, so by Lemma 2, we know that the total area of our curve decays like:
\begin{equation}-4\pi \leq \frac{dA}{dt} \leq -2\pi.\end{equation}
Thus, after $1/(8\pi)$ seconds of evolution, our curve will still have a total area greater than $1/2$. 

Suppose that the isoperimetric ratio of our curve at $t=0$ satisfies\\ $L^2/A=L^2 < M$. Then $|x| < \sqrt{M}/2$, and, since the curve is shortening in length, this remains true for all time. In our effort to apply the maximum principle to the evolution equation for $\theta$, we will compare $\theta$ with a function whose range is a subset of $[0,\pi/2)$. For the purposes of this theorem, we consider the solution $f(x,t)$ to the following (linear) initial value problem:
$$\frac{\partial f}{\partial t} = \frac{1}{2}\cdot \frac{\partial^2 f}{\partial x^2} \quad\textrm{and} \quad f(x,0) = \begin{cases} 0, & |x|< \sqrt{M} \\ \pi/4, & |x| \geq M\end{cases}$$
Explicitly, we have
$$f(x,t)=\frac{1}{8} \pi  \left(\text{erfc}\left(\frac{\sqrt{M}-x}{\sqrt{2}
   \sqrt{t}}\right)+\text{erfc}\left(\frac{\sqrt{M}+x}{\sqrt{2} \sqrt{t}}\right)\right)$$
where $\text{erfc}$ is the complementary error function. The range of $f(x,t)$ is contained in the interval $[0,\pi/4]$ and begins below the graph of $\theta(x,t=0)$. Thus, since the maximum principle applies for $\theta \in [0,\pi/4]$, we get that 
$$\theta_{min}(t) > f(0,t)=\frac{\pi}{4} \text{erfc}\left(\frac{\sqrt{M}}{\sqrt{2} \sqrt{t}}\right)$$
for all further time. In particular, after $1/(8\pi)$ seconds, $\theta_{min}$ has increased by at least $\delta=\frac{\pi}{4}\text{erfc}\big(2\sqrt{M\pi}\big)$.

Given our starting, unit area curve, we also suppose that the isoperimetric ratio stays bounded during the evolution, i.e. $\frac{L^2}{A}(t) <M$ for all time. Since the area is converging to zero, the area halves infinitely many times, but we have already shown that every time the area halves, $\theta_{min}$ increases by at least a fixed positive amount $\delta$, which would mean that $\theta_{min}$ converges to $\infty$, an obvious contradiction.
\end{proof}
Before we continue, we remark that our particular choice of $f(x,t)$ for the comparison function was only a matter of convenience. We could just have easily chosen some other step function's evolution, so long as its range were contained in $[0,\pi/2]$.
\section{Concinnous Figure-Eights}
The following is an interesting and useful proposition of Grayson and Angenent, and more can be learned about its proof and consequences in \cite{DHW}.
\begin{proposition}
The limit of a balanced figure-eight curve in the singularity is contained in the closure of the set of (pre-singular) double points.
\end{proposition}
This is of course in the same vein as Grayson's conjecture, for the proposition ensures that the limiting set is either a point or line segment. For us, this also means that balanced figure-eight curves with sufficient symmetry (like our concinnous curves) definitely converge to a point.
\begin{corollary}
Concinnous figure-eight curves converge to a point under the curve-shortening flow.
\end{corollary}
The natural question to ask now: what is the asymptotic shape of the curve? The work done in \cite{HA} guarantees that (any) figure-eight curve cannot evolve in a self-similar manner. Theorem 1, on the other hand, tells us that the isoperimetric ratio blows up in the singularity. Thus, the area preserving flow will have the shape limit to two increasingly long and narrow slits. 

Here we will mainly employ the modification of CSF that fixes tangent directions, previously mentioned in Lemma 3. The key is that, away from the only inflection point, we can parametrize our (strictly convex) concinnous curves by their tangent angle. Also, by symmetry, we need only examine the behavior of one lobe to get information about the other.

Choose cartesian coordinates so that the double point of our concinnous figure-eight curve is at the origin and such that the minor axis of symmetry coincides with the $x$-axis. Now, if $2\alpha(t)$ is the interior angle at the double point, then the domain for $\theta$, our tangent angle parametrization, is
$$D_1(t)= \big(-\alpha(t), \pi +\alpha(t)\big)$$
and the length of the interval $D_1(t)=2\alpha(t)+\pi$ is bounded as:
$$\pi<|D_1(t)| <2\pi.$$
Before we continue, we make precise what we mean by "blow-up set":
\begin{definition}
The \textbf{Blowup Set} is
$$\Sigma = \{ \theta\in D_1 | \lim_{t\to T} k(\theta, t) = \infty\} $$
\end{definition}
Angenent's Theorem is a very precise description of the blowup set whenever we have a collapsing lobe, which occurs, for example, in the case of balanced figure-eights and symmetric cardioids. Angenent colorfully calls the limiting shape a grim-reaper noose, because after rescaling so that $k_{max}(t)\equiv 1$, the blow-up set converges uniformly to a grim-reaper curve (i.e. $k(\theta)= \cos{\theta}$). Quantitatively, we have
\begin{theorem}[Angenent, \cite{ANG}]
If the blowup set $\Sigma$ has length less than $2\pi$, then for some $\beta\in D_1(t)$, we have
$$\Sigma = [\beta-\pi/2, \beta + \pi/2],$$
and
$$\lim_{t\to T} \frac{k(\beta+\phi, t)}{k_{max}(t)}= \cos{\phi}$$
uniformly for $\phi \in [-\pi/2, \pi/2]$.
\end{theorem}
For our concinnous figure-eights, we have an even more precise description of $\Sigma$. Indeed, after choosing cartesian coordinates so that the double point is at the origin and the $x$-axis is the minor symmetry axis of the curve, we know that $k_{max}(t)$ must be attained at $x_{max}(t)$ (i.e. at the end of the curve). We are able to get the following result in a future paper, written with Richard Schwartz:
\begin{corollary}
The blowup set of a concinnous figure-eight curve is $\Sigma = [0, \pi]$ and 
$$\lim_{t\to T} \frac{k(\phi,t)}{k_{max}(t)}= \sin{\phi}$$
uniformly for $\phi\in[0,\pi]$.
\end{corollary}
Qualitatively, the far end of a figure-eight becomes a grim-reaper curve (turned sideways) after the appropriate rescaling. 
\newpage
\begin{figure}[H]
\centering
\includegraphics[width=\textwidth]{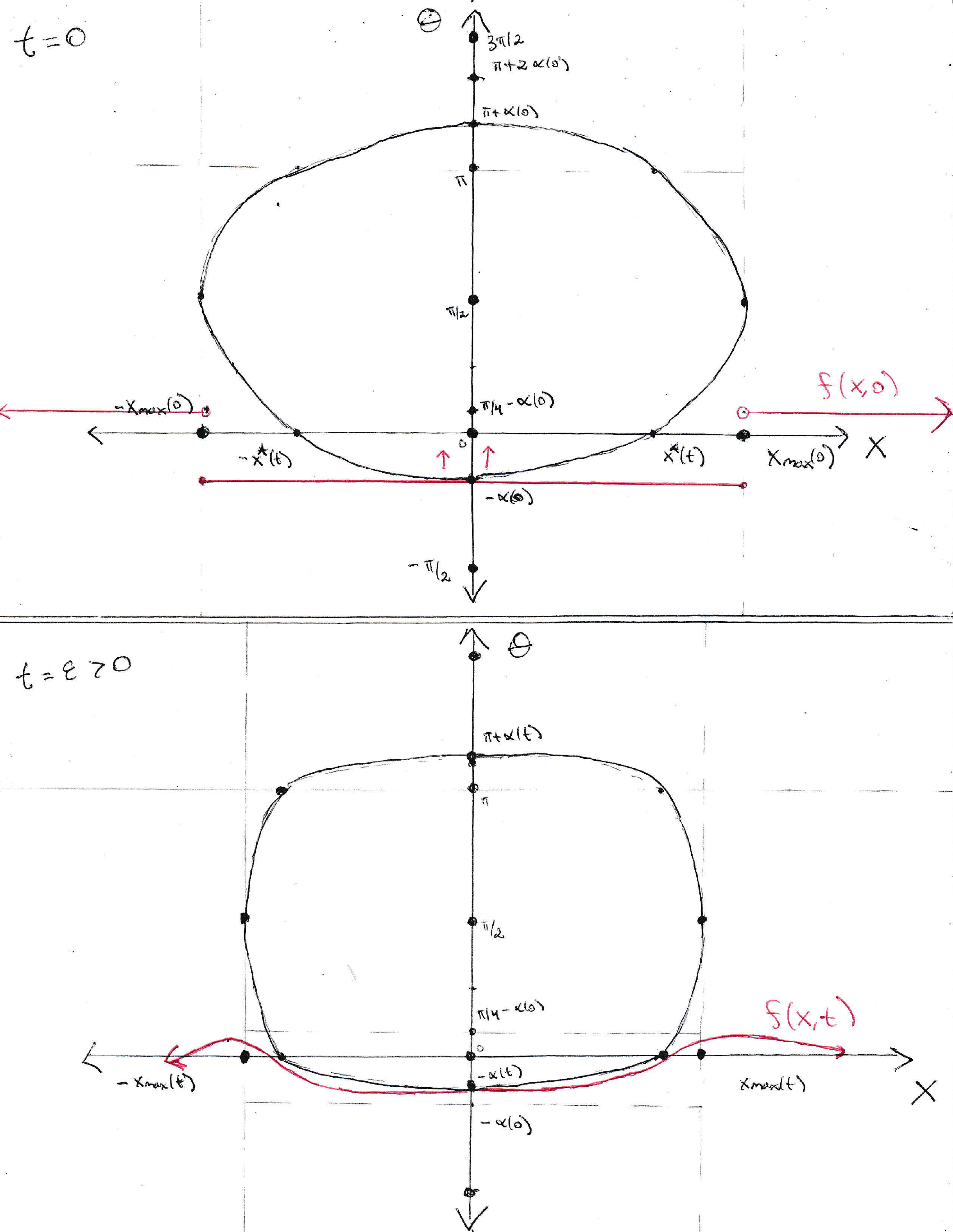}
\caption{We are comparing the flow of a portion of the $\theta$ curve with the red step-function in the proof of Theorem 1. }
\end{figure}
\newpage
\section{The Affine Rescaling}
Professor Richard Schwartz has suggested to us another rescaling of curve-shortening flow that is especially suited for concinnous figure-eights: an "affine" rescaling that keeps the curve in a "bounding box" of side-length $1$. Typical affine transformations do \textit{not} commute with the CSF, unlike isometries of the plane, so we are performing the affine transformation at the end of the evolution and seeing what the result should be. We \textit{cannot} affinely rescale, afterwards continue the CSF, and still get the same result. 

In this section, we will use both modifications of CSF. As usual, we choose cartesian coordinates so that the double point is at the origin and the minor axis of symmetry is the $x$-axis. Also, by the symmetry of a concinnous figure eight, we need only study the evolution of the portion of the curve that lies in the positive quadrant (i.e. $1/4$ of the curve). In particular, if we are using the $x-$coordinate preserving flow, we are studying the evolution of a multi-variable function $y(x,t)$ that satisfies
$$\frac{\partial y}{\partial t} = \frac{y''}{1+y'^2}$$
as in Lemma 1, with the initial function $y(x,0)$ corresponding to a concave curve having a vertical tangent at $x_{max}(0)$. The domain of this function is $D(t)=[0, x_{max}(t)]$, the range is $R(t)=[0,y_{max}(t)]$, and we choose the orientation of the curve so that $y''<0$ on $D(t)$. Both $x_{max}(t)$ and $y_{max}(t)$ are monotonically decreasing; we also define $x^*(t)$ to be the function of time that satisfies 
$$y(x^*(t), t) =y_{max}(t).$$
Now we can make our affine rescaling explicit. We are dilating the $x$ direction by $1/x_{max}(t)$ and the $y$ direction by $1/y_{max}(t)$, which evidently leads to a curve contained within a unit square bounding box. 

In the previous section, we saw Angenent's theorem on collapsing nooses and a related result for concinnous figure-eights. The far end converges to a grim-reaper curve. When we perform the affine rescaling on the grim reaper portion of the curve, we get a vertical line-segment because of its position on the far end and because the diameter is bounded in the $y$ direction and unbounded in the $x$ direction. In the framework of this section, this means that 
\begin{proposition}
$$\lim_{t\to T} \frac{x^*(t)}{x_{max}(t)}=1.$$
\end{proposition}
What about the rest of the curve (after affine rescaling)? Numerical evidence, communicated to us from Richard Schwartz, suggests that the resulting shape should be shaped like a bow-tie. In one quadrant, this means that our curve limits to two line segments intersecting in the corner of the bounding box. This is our main claim in Chapter 2:
\begin{theorem}[Main Claim]
After affine-rescaling, concinnous figure-eights converge to a bow-tie shape.
\end{theorem}
\section{Towards our Main Claim}
The "proof" in this chapter is just a sketch. For example, we need to also show that the results in \cite{ANG} (Angenent only studied strictly convex immersed curves) still apply to the case of the figure-eight. However, we are preparing another paper which will address these issues and provide a complete proof of our main claim.

This next lemma gives us a goal to aim for.
\begin{lemma} 
The concinnous figure-eight converges to a bow-tie if 
$$\lim_{t\to T}\bigg( y_{max}(t)\cdot\frac{dx_{max}}{dt}(t)+x_{max}(t)\cdot\frac{dy_{max}}{dt}(t)\bigg)=-\pi.$$
\end{lemma}
\begin{proof}
After applying the affine rescaling and before the blow-up time, our concinnous figure-eight remains convex except at the double point. The limiting curve, after affine rescaling, will still have absolute curvature that is everywhere greater than or equal to zero. By Proposition 2, the far end of the figure-eight becomes a line segment after the affine rescaling (alternatively, it limits to the vertical edge of the bounding box). The only curve from the origin to the corner of the bounding box that encloses an area equal to $1/2$ and has $|k|\geq0$ is the obvious line segment. Thus, to show that our concinnous figure-eight converges to a bow-tie, it suffices to show that
$$\lim_{t\to T} \frac{x_{max}(t)y_{max}(t)}{2A(t)} = 1$$
where $A(t)$ is the area enclosed by one quarter of the curve and the axes of symmetry. Using Lemma 2, we see that
$$\frac{dA}{dt}(t)= -(\alpha(t)+\pi/2)$$
but since the blow-up set $\Sigma= [0,\pi]$, we know $\lim_{t\to T} \alpha(t)=0$.
Finally, by L'Hospital's rule, we have
$$\lim_{t\to T} \frac{x_{max}(t)y_{max}(t)}{2A(t)} = 1$$
if we can show that 
$$\lim_{t\to T} \frac{y_{max}(t)\cdot\frac{dx_{max}}{dt}(t)+x_{max}(t)\cdot\frac{dy_{max}}{dt}(t)}{-2(\alpha(t)+\pi/2)} = 1$$
which is equivalent to showing
$$\lim_{t\to T}\bigg( y_{max}(t)\cdot\frac{dx_{max}}{dt}(t)+x_{max}(t)\cdot\frac{dy_{max}}{dt}(t)\bigg)=-\pi.$$
\end{proof}
We continue by proving some technical lemmas (we are using the $x$-coordinate preserving flow in the next few).
\begin{lemma}
$$\frac{d x_{max}}{dt}(t)=k(x_{max}(t),t) \quad \textrm{ and } \quad \frac{d y_{max}}{dt}(t)=k(x^*(t),t)$$
\end{lemma}
\begin{proof}
By the chain rule:
$$\frac{d y_{max}}{dt}(t)=\frac{d}{dt}y(x^*(t),t)=\frac{\partial y}{\partial x}(x,t) \frac{d x^*}{dt}(t)+\frac{\partial y}{\partial t}(x,t)= y''(x^*(t),t)=k(x^*(t),t)$$
The proof is identical for the derivative of $x_{max}(t)$.
\end{proof}
\begin{lemma}
$$\lim_{t\to T}\frac{\frac{dy_{max}}{dt}(t)}{\frac{dx_{max}}{dt}(t)}=0$$
\end{lemma}
\begin{proof}
By Corollary 2, we know that 
$$\lim_{t\to T} \frac{k(x^*(t),t)}{k(x_{max}(t),t)}=0$$
so the result follows immediately from the previous lemma.
\end{proof}
\begin{lemma}
$$\lim_{t\to T} \frac{y_{max}(t)}{x_{max}(t)}=0$$
\end{lemma}
\begin{proof}
Use L'Hospital's rule and the previous lemma.
\end{proof}
\begin{lemma}
$$\lim_{t\to T} \bigg(-y_{max}(t)\cdot \frac{dx_{max}}{dt}(t)\bigg)= \frac{\pi}{2}$$
\end{lemma}
\begin{proof}
Corollary 2 states that the far end of our concinnous figure-eight converges in $C^2$ to a Grim Reaper curve, after rescaling so that the maximum curvature is identically equal to $1$. However, from the same result, we also know 
$$k_{max}(t)= \sup_{x\in D(t)} \abs{k(x,t)} = \abs{k(x_{max}(t),t)}.$$
Since the Grim Reaper is turned on its side, has finite vertical diameter equal to $\pi$, and is reflection-symmetric across the minor axis of our figure eight, the desired result follows immediately. 
\end{proof}
The rest of the proof will be split into two main parts. In the first part, we will prove that
$$\limsup_{t\to T} \bigg(-x_{max}(t)\cdot\frac{dy_{max}}{dt}(t)\bigg)\leq\frac{\pi}{2}$$
by employing some integral estimates and our characterization of the Blowup Set in Corollary 2. In the second part we will use the Maximum Principle to show that
$$\liminf_{t\to T} \bigg(-x_{max}(t)\cdot\frac{dy_{max}}{dt}(t)\bigg)\geq\frac{\pi}{2}.$$
Together with Lemma 5, these results prove our main theorem.
\begin{center}
\textbf{First Part}
\end{center}
To use the full strength of Corollary 2, we need to use the CSF modification that preserves tangent angles. We recall that the evolution equation for curvature is then given by
$$\frac{\partial k}{\partial t} = k^2 \frac{\partial^2 k}{\partial \theta^2}+k^3$$
and the line element is $ds=d\theta/k$. We are now in a position to get integral formulae for $x_{max}(t), y_{max}(t)$, and their derivatives.
\begin{lemma}
$$x_{max}(t) = -\int_{-\alpha(t)}^{\pi/2} \frac{\cos\theta}{k(\theta,t)}d\theta \quad \textrm{ and } \quad y_{max}(t) = -\int_{0}^{\pi/2} \frac{\sin\theta}{k(\theta,t)}d\theta.$$
\end{lemma}
\begin{proof} This is obvious as long as we keep track of how our curve is parametrized by $\theta$ and if we remember that $ds=d\theta/k$.
\end{proof}
\begin{lemma}
$$\frac{dy_{max}}{dt}(t)=k(\theta=0,t)$$
\end{lemma}
\begin{proof}
Directly differentiating the integral equation in Lemma 10 gets us
$$\frac{dy_{max}}{dt}(t)=\int_0^{\pi/2}(k_{\theta \theta} + k)\sin(\theta) d\theta.$$
Since
$$(k_{\theta \theta} + k)\sin\theta=\frac{\partial}{\partial \theta}\big(k_\theta \sin\theta -k\cos \theta\big)$$
we have
$$\frac{dy_{max}}{dt}(t)=\big(k_\theta \sin\theta -k\cos \theta\big)\bigg|_{\theta=0}^{\theta=\pi/2}=k(\theta=0,t)$$
because curvature is maximized at the far end of our figure eight (i.e. at $\theta=\pi/2$).
\end{proof}
We are now equipped to accomplish our goal for Part 1.
\begin{lemma}
$$\limsup_{t\to T} \bigg(-x_{max}(t)\cdot \frac{dy_{max}}{dt}(t)\bigg) \leq\frac{\pi}{2}$$
\end{lemma}
\begin{proof}
Using the two preceding lemmas, we have
$$-x_{max}(t)\cdot \frac{dy_{max}}{dt}(t)=\int_{-\alpha(t)}^{\pi/2} \frac{k(0,t)\cos\theta}{k(\theta,t)}d\theta$$
so
$$\abs{-x_{max}(t)\cdot \frac{dy_{max}}{dt}(t)}\leq (\pi/2+\alpha(t))\cdot\sup_{\theta\in[-\alpha(t),\pi/2]}\abs{\frac{k(0,t)}{k(\theta,t)}}.$$
From Corollary 2, which describes the blow-up set $[0,\pi]$, we get
$$\lim_{t\to T} \sup_{\theta\in[-\alpha(t),\pi/2]}\abs{\frac{k(0,t)}{k(\theta,t)}}=1$$
therefore, since $\lim_{t\to T} \alpha(t)=0$ as well, we have
$$\limsup_{t\to T} \abs{-x_{max}(t)\cdot \frac{dy_{max}}{dt}(t)}\leq \frac{\pi}{2}$$
as desired. 
\end{proof}
\begin{center}
\textbf{Second Part}
\end{center}
Recall that our goal for this section is to prove
$$\liminf_{t\to T} \bigg(-x_{max}(t)\cdot \frac{dy_{max}}{dt}(t)\bigg) \geq \frac{\pi}{2}.$$
In what follows, we use the modification of CSF that preserves $x-$coordinates. As usual, this means that we are looking at the evolution of one quarter of the concinnous figure-eight, viewed as the graph of the concave function $y(x,t)$. This function evolves with time according to
$$\frac{\partial y}{\partial t} = \frac{y''}{1+y'^2}.$$
Consider the function
$$Y(x,t):= \frac{y(x,t)}{x_{max}(t)}.$$
Using the quotient rule, we see that $Y(x,t)$ satisfies the following evolution equation:
$$\frac{\partial Y}{\partial t}(x,t) = \frac{Y''(x,t)}{1+y'(x,t)^2}-\frac{Y(x,t)\cdot \frac{dx_{max}}{dt}(t)}{x_{max}(t)}.$$
Away from vertical tangents, the evolution of $Y$ is strictly parabolic. In particular, by applying the Maximum Principle, we see that local maxima of $Y$ are strictly decreasing. Thus:
$$\frac{d}{dt}\bigg(\frac{y_{max}(t)}{x_{max}(t)}\bigg)<0$$
so we get, for all $t$,
$$-x_{max}(t)\frac{dy_{max}}{dt}(t)> -y_{max}(t)\frac{dx_{max}}{dt}(t)$$
and taking the limit:
$$\liminf_{t\to T} \bigg(-x_{max}(t)\cdot \frac{dy_{max}}{dt}(t)\bigg) \geq \lim_{t\to T} \bigg(-y_{max}(t)\frac{dx_{max}}{dt}(t)\bigg)=\frac{\pi}{2}.$$
As a consequence, and using the results from Part 1, we have
$$\lim_{t\to T} \bigg(-x_{max}(t)\cdot \frac{dy_{max}}{dt}(t)\bigg) = \lim_{t\to T} \bigg(-y_{max}(t)\frac{dx_{max}}{dt}(t)\bigg)=\frac{\pi}{2}$$
therefore, by Lemma 5, our Main Theorem is proven. Q.E.D.

\chapter{The Curvature Preserving Flow on Space Curves}
\section{Preliminaries}
We recall the following standard computation:
\begin{lemma}[cf. \cite{BSW}, \cite{I}]
Let $\gamma(t)$ be a family of smooth curves immersed in $\mathbb{R}^3$ and let $X(s,t)$ be a parametrization of $\gamma(t)$ by arc-length. Consider the following geometric evolution equation:
\begin{equation}X_t = h_1 \textbf{T} +h_2 \textbf{N}+h_3 \textbf{B}\end{equation}
where $\{\textbf{T},\textbf{N},\textbf{B}\}$ is the Frenet-Serret Frame and where we denote the curvature and torsion by $\kappa$ and $\tau$, respectively. Let $h_1, h_2, h_3$ be arbitrary smooth functions of $\kappa$ and $\tau$ on $\gamma(t)$. If the evolution is also arc-length preserving, then the evolution equations of $\kappa$ and $\tau$ are
$$\begin{pmatrix} \kappa_t \\ \tau_t \end{pmatrix}=P\begin{pmatrix} h_3 \\ h_1 \end{pmatrix}$$
where $P$ is
$$\begin{pmatrix} -\tau D_s - D_s\tau & D_s^2\frac{1}{\kappa}D_s-\frac{\tau^2}{\kappa}D_s+D_s\kappa\\
D_s\frac{1}{\kappa}D_s^2-D_s\frac{\tau^2}{\kappa}+\kappa D_s& D_s(\frac{\tau}{\kappa^2}D_s+D_s\frac{\tau}{\kappa^2})D_s+\tau D_s+D_s\tau\end{pmatrix}.$$
\end{lemma}
The next theorem follows without difficulty.
\begin{theorem}
Up to a rescaling, a geometric evolution of curves immersed in $\mathbb{R}^3$, as in equation $(3)$, is both curvature and arc-length preserving if and only if its evolution evolution is equivalent to 
\begin{equation}
X_t = \frac{1}{\sqrt{\tau}}\textbf{B}.
\end{equation}
\end{theorem}
\begin{proof}
Let $X_t$ be a curvature and arc-length preserving geometric flow. The tangential component of $X_t$ in equation $(3)$ provides no interesting geometric information; it amounts to a re-parametrization of the curve. Thus, we may assume that $h_1=0$, so, since $X_t$ is arc-length preserving, $h_2=0$ as well. Lemma 13 gives us that 
$$\kappa_t = -\tau D_s (h_3)-D_s(\tau h_3)=-2\tau D_s (h_3)-h_3D_s(\tau).$$
Since $X_t$ is curvature preserving, we must have $\kappa_t=0$, or 
$$D_s(\log{h_3})=D_s(\log{\tau^{-1/2}}).$$
Integrating, we see that 
$$h_3=\frac{c}{\sqrt{\tau}}$$
where $c$ is a constant. Therefore, up to a rescaling, the evolution $X_t$ must be precisely as in $(4)$.
\end{proof}
Unfortunately, the evolution in $(4)$ only makes sense if $\gamma(t)$ has strictly positive torsion (or strictly negative torsion, with the flow $X_t=\frac{1}{\sqrt{-\tau}}\textbf{B}$). This motivates the following definition.
\begin{definition}
We call a smooth curve immersed in $\mathbb{R}^3$ a positive curve and only if it has strictly positive torsion.
\end{definition}
Fortunately, there are many interesting positive curves. For example, some knots admit a parametrization with constant curvature and strictly positive torsion, and there exist closed curves with constant (positive) torsion. Henceforth, we will only consider positive curves. 

The partial differential equation governing the evolution of $\tau$ follows below:
\begin{lemma}
For the geometric flow given in equation $(4)$, the evolution equations of curvature and torsion are $\kappa_t =0$ and
\begin{equation}
\tau_t = \kappa D_s(\tau^{-1/2})+D_s\bigg(\frac{D_s^2(\tau^{-1/2})-\tau^{3/2}}{\kappa}\bigg)
\end{equation}
\end{lemma}
The equation for torsion is reminiscent of the Rosenau-Hyman family of equations, which are studied, inter alia, in \cite{HEER} and \cite{LW}. The authors of \cite{SR} call the evolution of the torsion, in the case of constant curvature, the \textit{extended Dym equation} for its relationship with the Dym equation (a rescaling and limiting process converts the evolution for torsion to the Dym equation). 
As discussed in \cite{BSW}, the condition for the flow $X_t$ from Lemma 13 to be the gradient of a functional is for the Frechet derivative of $(h_3,h_1)$ to be self-adjoint. In general, this does not occur for the flow under our consideration in equation $(4)$, so it cannot be integrable in the sense of admitting a Hamiltonian structure (Indeed, when $\kappa$ is not constant, it is not even formally integrable in the sense of \cite{MSS}). Nevertheless, the case of constant curvature exhibits a great deal of structure, which makes the study of the evolution equation in $(5)$ worthwhile.
\section{Constant Curvature}
When $\kappa$ is constant, we may rescale the curves $\gamma(t)$ so that $\kappa \equiv 1$. In this way, the evolution of torsion becomes
\begin{equation}
\tau_t = D_s\big(\tau^{-1/2}-\tau^{3/2}+D_s^2(\tau^{-1/2})\big)
\end{equation}
\subsection{Equivalence with the m$^2$KDV Equation}
We recall the notion of "equivalence" of two partial differential equations from \cite{CFA}:
\begin{definition}
Two partial differential equations are equivalent if one can be obtained from the other by a transformation involving the dependent variables or the introduction of a potential variable.
\end{definition}
The authors in \cite{CFA} discuss a general method of transforming quasilinear partial differential equation, such as the evolution of $\tau$ in $(6)$, to semi-linear equations. By applying their algorithm, we obtain another proof of the following theorem, first demonstrated in \cite{SR}.
\begin{theorem}[\cite{SR}]
The evolution equation for torsion, in the case of constant curvature, which is given by
$$\tau_t = D_s\big(\tau^{-1/2}-\tau^{3/2}+D_s^2(\tau^{-1/2})\big),$$
is equivalent to the m$^2$KDV equation. Thus, it is a completely integrable evolution equation.
\end{theorem}
\begin{proof}
First, let $\tau=v^2$, so that $(6)$ becomes
$$
2vv_t=D_s\bigg(\frac{1}{v}-v^3+D_s^2\big(\frac{1}{v}\big)\bigg)=-\frac{v_s}{v^2}-3v^2v_s-\frac{v_{sss}}{v^2}-\frac{6v_s^3}{v^4}+\frac{6v_sv_{ss}}{v^3}$$
or, in a simpler form:
\begin{equation}
v_t=D_s\bigg(\frac{1}{4v^2}-\frac{3v^2}{4}+\frac{3v_s^2}{4v^4}-\frac{v_{ss}}{2v^3}\bigg)
\end{equation}
A potentiation, $v=w_s$ followed by a simple change of variables $(t\rightarrow -t/2)$ yields
\begin{equation}
w_t=-\frac{1}{2w_s^2}+\frac{3w_s^2}{2}-\frac{3w_{ss}^2}{2w_s^4}+\frac{w_{sss}}{w_s^3}.
\end{equation}
Equation $(8)$ is fecund territory for a pure hodograph transformation, as used, for example, in \cite{CFA}. Let $\tilde{t}=t$, $\xi=w(s,t)$, and $s=\eta(\xi,\tilde{t})$. The resulting equation, after a simple computation, is 
\begin{equation}
\eta_{\tilde{t}}=\eta_{\xi\xi\xi}-\frac{3\eta_{\xi\xi}^2}{2\eta_\xi}+\frac{\eta_\xi^3}{2}-\frac{3}{2\eta_\xi}.
\end{equation}
We rename the variables to the usual variables of space and time: $s$ and $t$; in addition, we anti-potentiate the equation by letting $\eta_s=z$. This makes equation $(9)$ equivalent to
\begin{equation}
z_t = z_{sss}-\frac{3}{2}\bigg(\frac{z_s^2}{z}\bigg)_s+\frac{3z^2z_s}{2}+\frac{3z_s}{2z^2}.
\end{equation}
Lastly, if we let $q=\sinh(z/2)$ in equation (10), a transformation also discussed in \cite{CD}, we get the m$^2$KDV equation
\begin{equation}q_t=q_{sss}-\frac{3}{2}\bigg(\frac{qq_s^2}{1+q^2}\bigg)_s+6q^2q_s\end{equation}
which finishes the proof of the theorem. 
\end{proof}
Equations $(6)$ and $(7)$ give us the first two integrals of motion of this flow:
$$\int \sqrt{\tau} ds \textrm{ and } \int \tau ds.$$
The rest can be found by pulling back the m$^2$KDV invariants; we note that these invariants were obtained in a different way by the authors of \cite{SR}. Our next step is to prove the long term existence of our geometric flow using Theorem 5.
\begin{corollary}
Let $\tau_0\in C_{per}^\infty([0,2\pi])$ be a strictly positive and periodic function. We can solve the evolution equation for torsion with initial data $\tau_0$ and get the unique solution $\tau(t)\in C_{per}^\infty([0,2\pi])$ that is strictly positive for all time $t\geq 0$.
\end{corollary}
\begin{proof}
In order to prove this corollary, it suffices to show that we can reconstruct the solution of $(6)$ by solving the m$^2$KDV equation $(11)$ instead. The long term existence for solutions to the m$^2$KDV equation will imply the same for $(6)$, so all that remains to prove is that every differential transformation we used in the proof of Theorem 5 is truly "invertible" in the sense that no singularities arise. We use the same notation for our differential transformations as before.

For the proof of this corollary, it is more convenient to work with an equivalent form of the m$^2$KDV equation called the Calogero-Degasperis-Fokas (CDF) equation \cite{CFA}, obtained by letting $z=e^u$ in equation $(10)$:
$$
u_t=u_{sss}-\frac{u_s^3}{2}+\frac{3u_s}{2}\big(e^{2u}+e^{-2u}\big).
$$
Beginning with our initial torsion, $\tau_0$, we let $v_0=\sqrt{\tau_0}$ and then
$$w_0(s)=\int_0^s v_0(\tilde{s})d\tilde{s}.$$
By construction $w_0\in C^\infty([0,2\pi])$, is not periodic, but satisfies $w_0'(s)>0$ for all $s\in[0,2\pi]$. Thus, $w_0$ has a global inverse $\eta_0(\xi)$ and we also have $\eta_0'(\xi) := z_0(\xi)>0$ for all $\xi$ in the domain of $\eta_0$, which is some interval of the form $[0,M]$. Finally, let $u_0=\log(z_0)$, which is well defined and in $C^\infty([0,M])$ because $z_0>0$ and $z_0\in C^\infty([0,M])$. The CDF equation is globally well-posed for smooth functions on a compact interval with periodic boundary conditions (see \cite{CD}, \cite{HEER}), so we get a solution $u(\xi,t)$ defined for all time with $u(\xi,0)=u_0(\xi)$ and $u(0,t)=u(M,t)$ for all $t$. Our next step is to reconstruct the solution for $(6)$ using $u(t)$.

First, let $z(t)=e^{u(t)}$, which will be a strictly positive, periodic, $C^\infty$ function for all time. Then, let
$$\eta(\xi,t)=\int_0^\xi z(\zeta, t) d\zeta \quad \textrm{ for all } \xi\in[0,M]$$
which, since $\int z(\zeta,t) d\zeta$ is a time-independent quantity, satisfies
$$\eta(0,t)=0 \quad \textrm{and}\quad \eta(M,t)=2\pi$$
for all time. Obviously $\eta_\xi(\xi, t)>0$ for all $(\xi,t)$, so the hodograph transformation $\tilde{t}=t$, $\xi=w(s,t)$, and $s=\eta(\xi,\tilde{t})$ makes sense. We recover a $C^\infty$ function $w(s,t)$ with the property that
$$w(0,t)=0 \quad\textrm{and} \quad w(2\pi,t)=M$$
for all time, and up to a trivial change of coordinates, we have
$$0<v(s,t)=\frac{\partial w}{\partial s}(s,t_0)\in C^\infty_{per}([0,2\pi])$$
for any fixed time $t_0\geq 0$. Lastly, $\tau(s,t)=v(s,t)^2$ solves equation $(6)$, is in $C^\infty_{per}([0,2\pi])$, and satisfies $\tau(s,0)=\tau_0(s)$. Thus, we have reconstructed our desired solution of $(6)$ using the completely integrable CDF equation instead.
\end{proof}

\subsection{Stationary Solutions}
Helices, with constant curvature and constant torsion, are the obvious stationary solutions. In what follows, we find the rest. 
\begin{theorem}
The stationary solutions of 
$$\tau_t = D_s\big(\tau^{-1/2}-\tau^{3/2}+D_s^2(\tau^{-1/2})\big)$$
are given by the following integral formula
$$\int\frac{du}{\sqrt{C+2Au-u^2-u^{-2}}}=s+k$$
where $\tau(s)=u(s)^{-2}$ and where $A,k$ and $C$ are appropriate real constants. When $A=0$ we get an explicit formula for the solutions:
$$\tau(s)=\frac{2}{C\pm\sqrt{(-4+C^2)}\cdot\sin(2(s+k))},$$
with $k$ and  $C\geq 2$ real constants.
\end{theorem}
\begin{proof}
Let $\tau=u^{-2}$, then, after integrating once, we must examine the following ordinary differential equation (where $A$ is a constant):
\begin{equation}
A=u-\frac{1}{u^3}+D_s^2(u)
\end{equation}
Since equation $(12)$ is autonomous, we may proceed with a reduction of order argument. Let $w(u)=D_s(u)$ so that $D_s^2(u)=wD_u(w)$ by the chain rule. This substitution gives us the first order equation:
\begin{equation}
wD_u(w)=A-u+\frac{1}{u^3}
\end{equation}
or
$$D_u(w^2)=2A-2u+\frac{2}{u^3}.$$
Integrating, we get
$$D_s(u)=w(u)=\sqrt{C+2Au-u^2-u^{-2}}$$
which is a separable differential equation. So, the stationary solutions of equation $(6)$ are given by the following integral formula:
\begin{equation}
\int\frac{du}{\sqrt{C+2Au-u^2-u^{-2}}}=s+k
\end{equation}
for appropriate constants $C, A, \textrm{ and } k$.
It would be pleasant to have explicit solutions, and this occurs in the case when $A=0$, which is more easily handled. Equation $(13)$ above becomes
$$D_s(u)=w(u)=\sqrt{C-u^2-u^{-2}}$$
which is a differential equation that can be solved with the aid of Mathematica or another computer algebra system. The result is
\begin{equation}
u(s)=\sqrt{\frac{C\pm\sqrt{(-4+C^2)}\cdot\sin(2(s+k))}{2}}
\end{equation}
Where $C,k$ are real constants and $C\geq 2$. The corresponding torsion is:
$$\tau(s)=\frac{2}{C\pm\sqrt{(-4+C^2)}\cdot\sin(2(s+k))}$$
\end{proof}
Integrating $\tau$ and $\kappa$ as above using the Frenet-Serret equations will yield the corresponding stationary curves, up to a choice of the initial Frenet-Serret frame and isometries of $\mathbb{R}^3$.  
\subsection{$L^2(\mathbb{R})$ Linear Stability of Helices}

First, we derive the linearization of the evolution for torsion around the stationary solutions corresponding to helices. The linearization of equation $(6)$ at any stationary solution $\tau_0$ is obtained by letting $\tau(s,t) = \tau_0(s,t)+\epsilon w(s,t)$, substituting into equation $(6)$, dividing by $\epsilon$ and then taking the limit as $\epsilon\rightarrow 0$. This is nothing more than the Gateaux derivative of our differential operator at $\tau_0$ in the direction of $w$. Alternatively, one may think of $w$ as the first-order approximation for solutions of $(6)$ near the stationary solution. We can perform this operation when $\tau_0$ is given by an explicit formula, but for brevity's sake, we only mention here the linearization around helices when $\tau_0$ is constant. A short calculation yields
\begin{proposition}
The linearization of the evolution equation $(6)$ around the stationary solutions of constant torsion is
\begin{equation}w_t+2w_s+\frac{1}{2}w_{sss}=0.\end{equation}
\end{proposition}
In what follows, we show the $L^2(\mathbb{R})$ linear stability of the the constant torsion stationary solution. First we recall the definition of linear stability:
\begin{definition}
A stationary solution $\phi$ of a nonlinear PDE is called $L^2(\mathbb{R})$ linearly stable when $v = 0$ is a stable solution of the corresponding linearized PDE with respect to the $L^2(\mathbb{R})$ norm and whenever $v_{t=0}$ is in $L^2(\mathbb{R})$.
\end{definition}
To work towards this, we need to use test functions from the Schwartz Space $\mathcal{S}(\mathbb{R})$, so we first recall that 
$$\mathcal{S}(\mathbb{R}):=\{ f\in C^{\infty}(\mathbb{R}) \textrm{ s.t. } \sup_{x\in\mathbb{R}} (1+\abs{x})^N\abs{\partial^{\alpha}f(x)}<\infty, \textrm{ for all } N, \alpha \}.$$
Intuitively, functions in $\mathcal{S}(\mathbb{R})$ are smooth and rapidly decreasing. For more on distributions and the Schwartz Space, review \cite{F}.
The rest of this section is devoted to proving:
\begin{theorem}
Helices correspond to $L^2(\mathbb{R})$ linearly stable stationary solutions of $$\tau_t = D_s\big(\tau^{-1/2}-\tau^{3/2}+D_s^2(\tau^{-1/2})\big).$$
\end{theorem}
\begin{proof}
Helices correspond to the constant torsion stationary solutions, so we analyze the linearized PDE in (16):
$$w_t+2w_s+\frac{1}{2}w_{sss}=0.$$
Henceforth, $w_0(s)$ denotes the initial data and $w(s,t)$ denotes the respective solution to the above, linear PDE. Thus, to prove the theorem, it suffices to show: for every $\epsilon$, there exists a $\delta$ such that if $w_0\in L^2(\mathbb{R})$ and $\|w_0\|_2<\delta$, then $\|w(t)\|_2<\epsilon$ for all $t\geq0$. 

We follow the standard process of finding weak solutions to linear PDEs via the Fourier transform. Moreover, we know by the Plancherel Theorem that we can extend the Fourier transform by density and continuity from $\mathcal{S}(\mathbb{R})$ to an isomorphism on $L^2(\mathbb{R})$ with the same properties. Hence, it suffices to prove the desired stability result for initial data in $\mathcal{S}(\mathbb{R})$. 

Let $F_t(\xi)=e^{4 i \pi ^3 \xi^3 t-4 i \pi  \xi t}$. We notice that since $F_t$ is a bounded continuous function for all $t\geq 0$, it can be considered a tempered distribution (or a member of $\mathcal{S}'(\mathbb{R})$, the continuous linear functionals on $\mathcal{S}(\mathbb{R})$), so its inverse Fourier transform makes sense. 

Indeed, we can let $B_t(s)=\mathcal{F}^{-1}(F_t(\xi))\in \mathcal{S}'(\mathbb{R})$ and, again, we denote $w_0\in \mathcal{S}(\mathbb{R})$ to be our initial data. Let
$$w(t)=B_t \ast w_0$$
so that $w(t)$ is a $C^{\infty}$ function with at most polynomial growth for all of its derivatives (see \cite{F}). Moreover, $w(t)$ satisfies equation $(16)$ in the distributional sense, as can be checked by taking the Fourier Transform. Lastly, since the Fourier Transform is a unitary isomorphism, it follows that
$$\lim_{t\rightarrow 0} w(t) = w_0$$
in the distribution topology of $\mathcal{S}'(\mathbb{R})$. Hence, $w(t)$ is the weak solution to equation $(16)$ with initial data $w_0\in\mathcal{S}(\mathbb{R})$. In our final step, we use the Plancherel Theorem and the fact that $\mathcal{F}(B_t)=F_t$ is a continuous function with $\|F_t\|_\infty =1$ for all $t\geq0$ to get:
$$\|w(t)\|_2=\|B_t \ast w_0 \|_2 =\|F_t \cdot \mathcal{F}(w_0)\|_2\leq \|F_t\|_\infty \|\mathcal{F}(w_0)\|_2= \|\mathcal{F}(w_0)\|_2=\|w_0\|_2.$$
From the inequality above, the desired $L^2(\mathbb{R})$ stability for initial data in $\mathcal{S}(\mathbb{R})$ follows forthrightly.
\end{proof}
\subsection{Numerical Rendering of a Stationary Curve}
We provide here the figures obtained from a numerical integration of the Frenet-Serret equations on Mathematica for the following choice of torsion:
$$\tau_1(s)=\frac{2}{3+\sqrt{5}\sin(2s)}$$ 
\begin{figure}[h]
\centering
\includegraphics[width=0.5\textwidth]{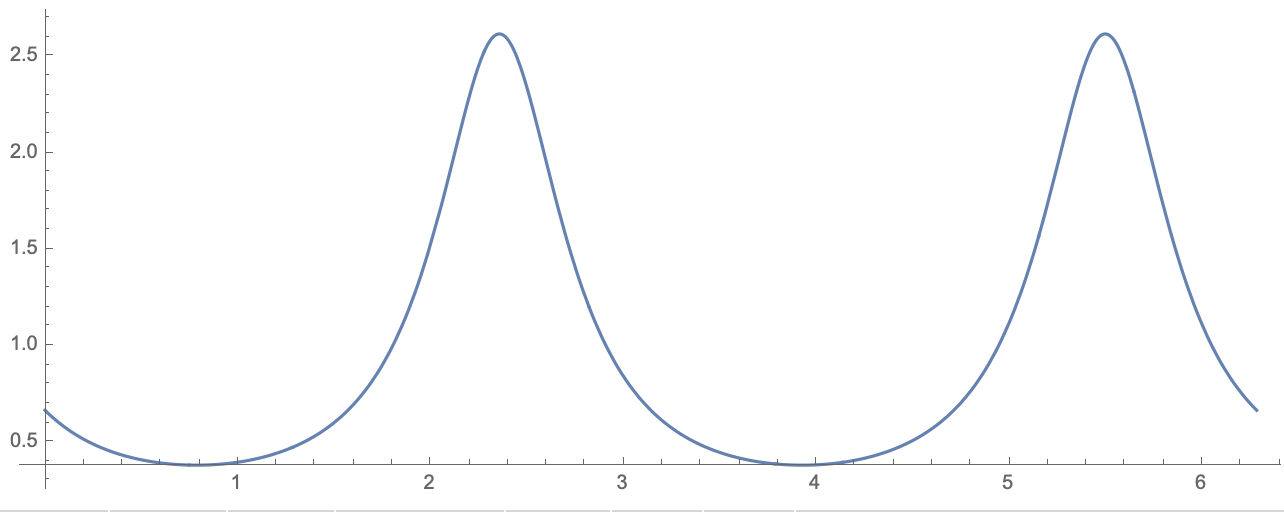}
\caption{This is the graph of the torsion $\tau_1$ over two periods.}
\end{figure}
\begin{figure}[h]
\centering
\includegraphics[width=0.5\textwidth]{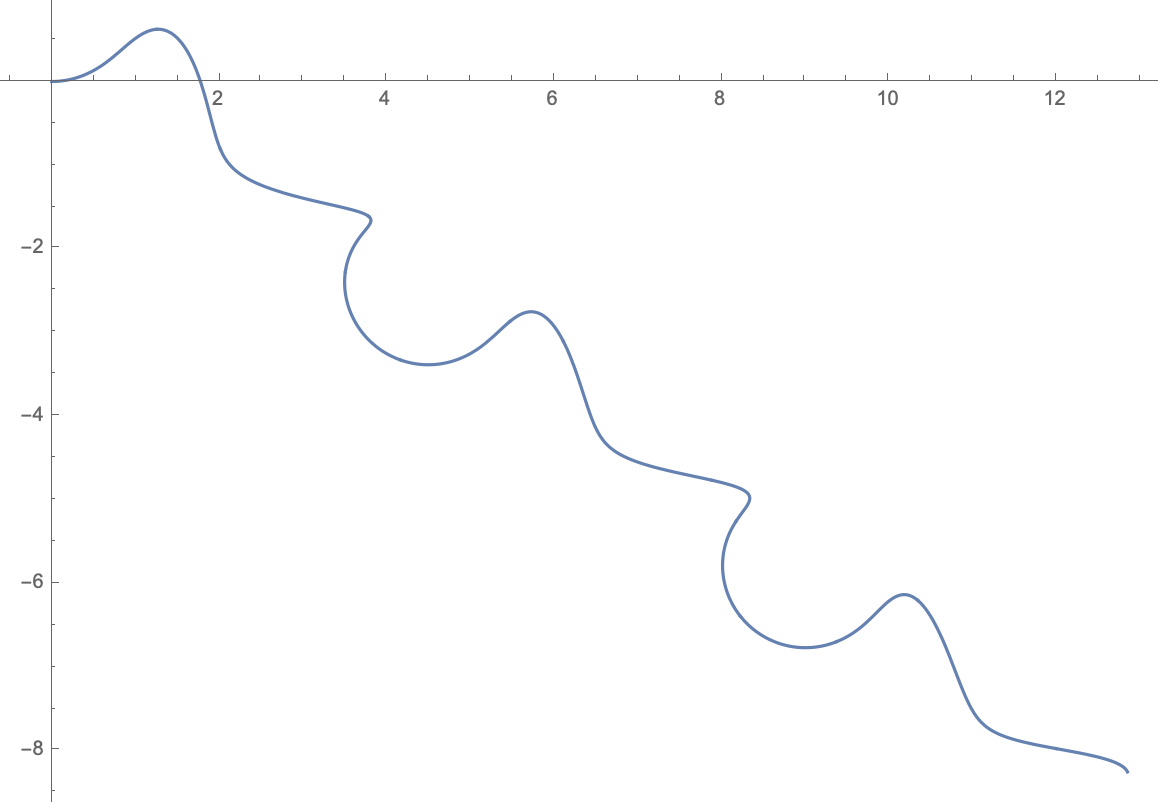}
\caption{This is the projection of the curve corresponding to $\tau_1$ into the $xy$ plane. The projections into the other planes look very similar.}
\end{figure}
\begin{figure}[H]
\centering
\includegraphics[width=0.35\textwidth]{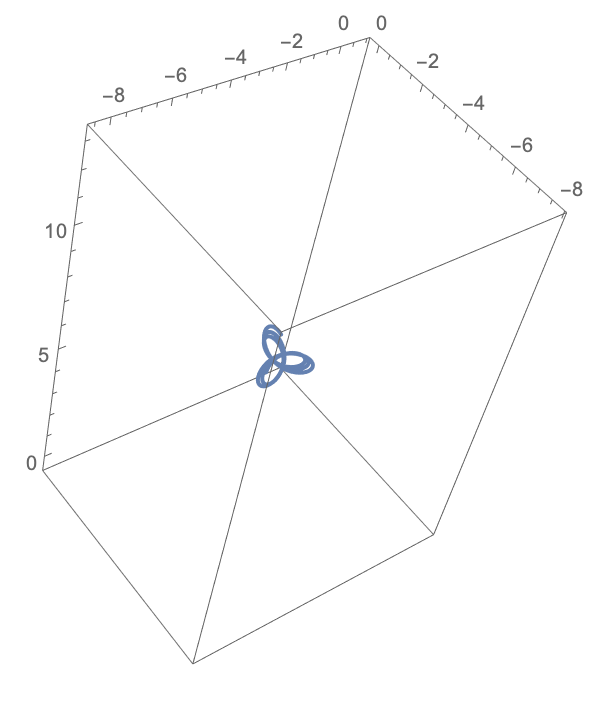}
\caption{This is a top-down view of the same curve, now exhibiting an almost trefoil shape.}
\end{figure}

\section{Numerical Solutions to the Torsion Evolution Equation}
Mathematica offers a built-in function, \textbf{NDSolve}, which can numerically solve both ordinary and partial differential equations. In our case, the PDE in question is a highly non-linear, time-dependent evolution, and Mathematica will usually employ the "method of lines" to get a numerical solution. The method of lines involves a discretization of the spatial variable that reduces the problem to one of integrating a system of ODE's, and its employment in numerical integration of PDE's is documented for at least the past five decades \cite{SC}.

The benefits of examining numerical solutions are manifold. First, we are able to try different initial data for torsion and see if the evolution behaves nicely; second, we can observe whether our stationary solutions demonstrate stability by slightly perturbing the initial data; third, we can experiment with the case of non-constant curvature, which is difficult to investigate with analytic methods. In this section, we pursue all three directions with several examples in each case. 
\subsection{Constant Curvature with Different Choices for Initial Torsion}
In this subsection, the curvature of our positive curves will always be identically equal to $1$. In our experiments, we have observed that initial torsion that is "too close" to zero leads to numerical problems (as might be expected), so we tend to choose our torsion to be at least $\gg1$. Our first example is the evolution of $2\pi$ periodic, smooth initial data that is quite far from vanishing anywhere:
$$\tau(s,t=0)=10+\frac{\sin(s)}{2}$$
We numerically solve equation $(6)$ for $\tau(s,t)$ and plot the graph of the result.
\begin{figure}[h]
\centering
\includegraphics[width=0.7\textwidth]{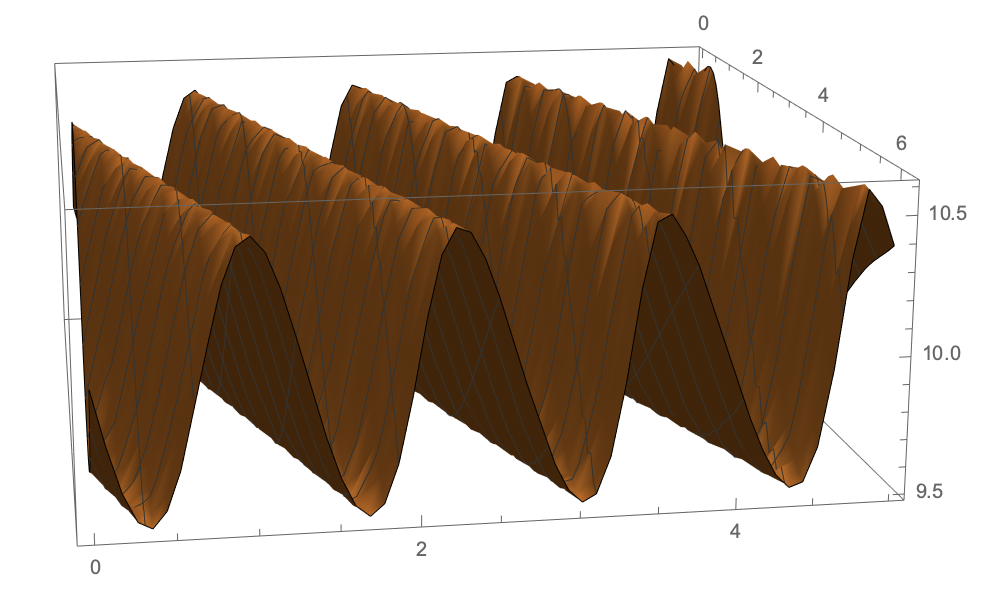}
\caption{The is the graph for $(s,t)\in [0,2\pi] \cross [0,5]$}
\end{figure}
\newpage
Already, we see a wave structure traveling along some direction (not parallel to the axes) while preserving its structure, for the most part. Quasi-periodicity is behavior typical of integrable partial differential equations, and this is evident for our flow as well. We direct the reader to the following figure, which depicts how our initial sine-wave is (almost) seen again at the times $t= 1.32, 2.64, 3.96$.
\begin{figure}[h]
\centering
\includegraphics[width=0.7\textwidth]{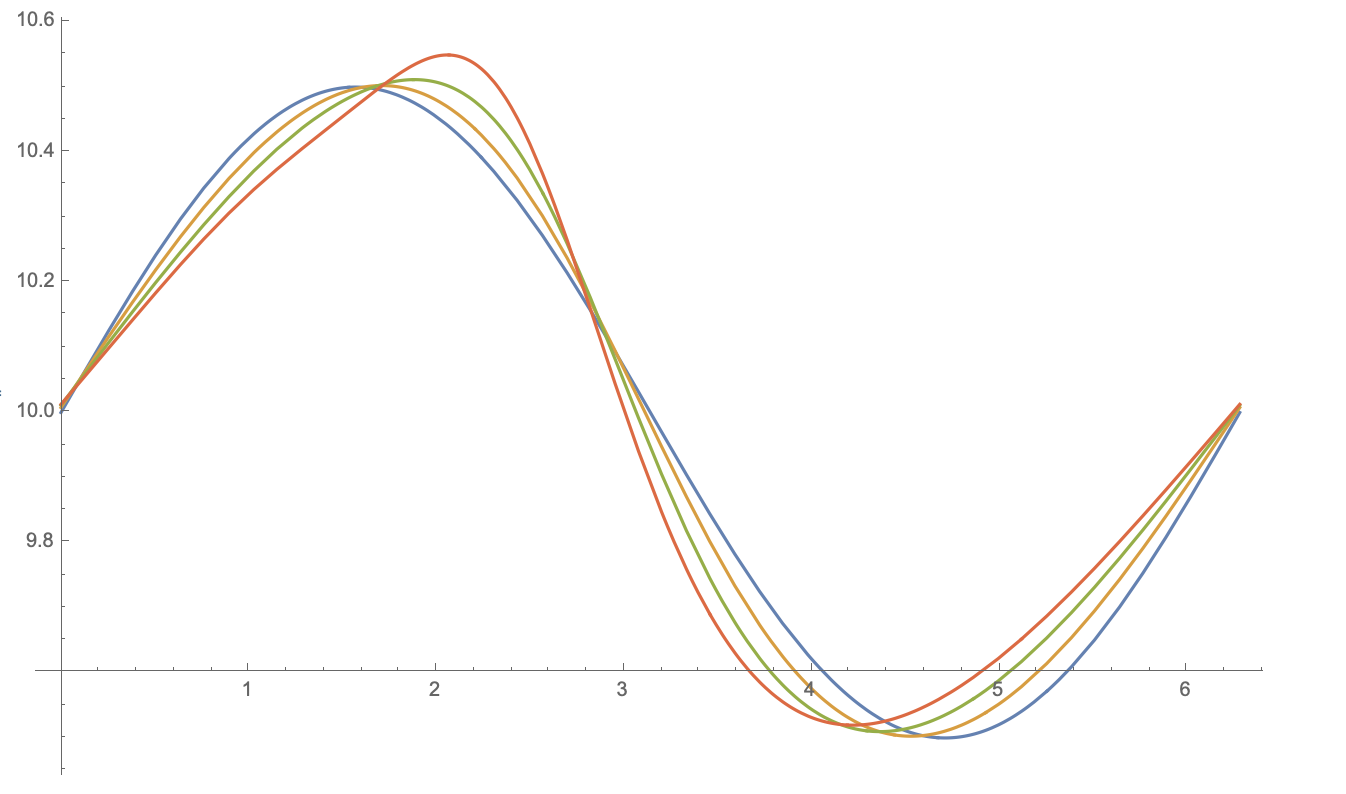}
\caption{The is the graph of $\tau$ at times $t=0, 1.32, 2.64,$ and $3.96$.}
\end{figure}

Our next example is the evolution of the following initial torsion data:
$$\tau(s,t=0) = 10 +\sin{s}+\cos{s}.$$
Once again, we numerically solve and plot the graph of the result.
\begin{figure}[h]
\centering
\includegraphics[width=0.7\textwidth]{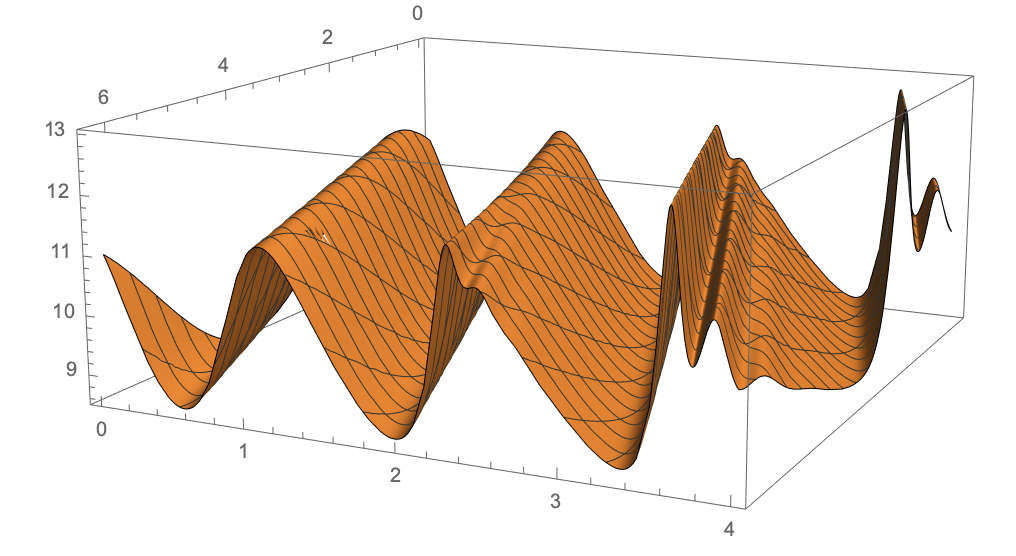}
\caption{The is the graph for $(s,t)\in [0,2\pi] \cross [0,4]$}
\end{figure}

Although we can once again see the propagation of a wavefront, in this case the solution at fixed times seems to split into two different waves with different velocities.
If we plot $\tau$ at the quasi periods $t=0, 1.26, 2.53,$ and $3.685$, we get the subsequent graphs:
\begin{figure}[h]
\centering
\includegraphics[width=0.7\textwidth]{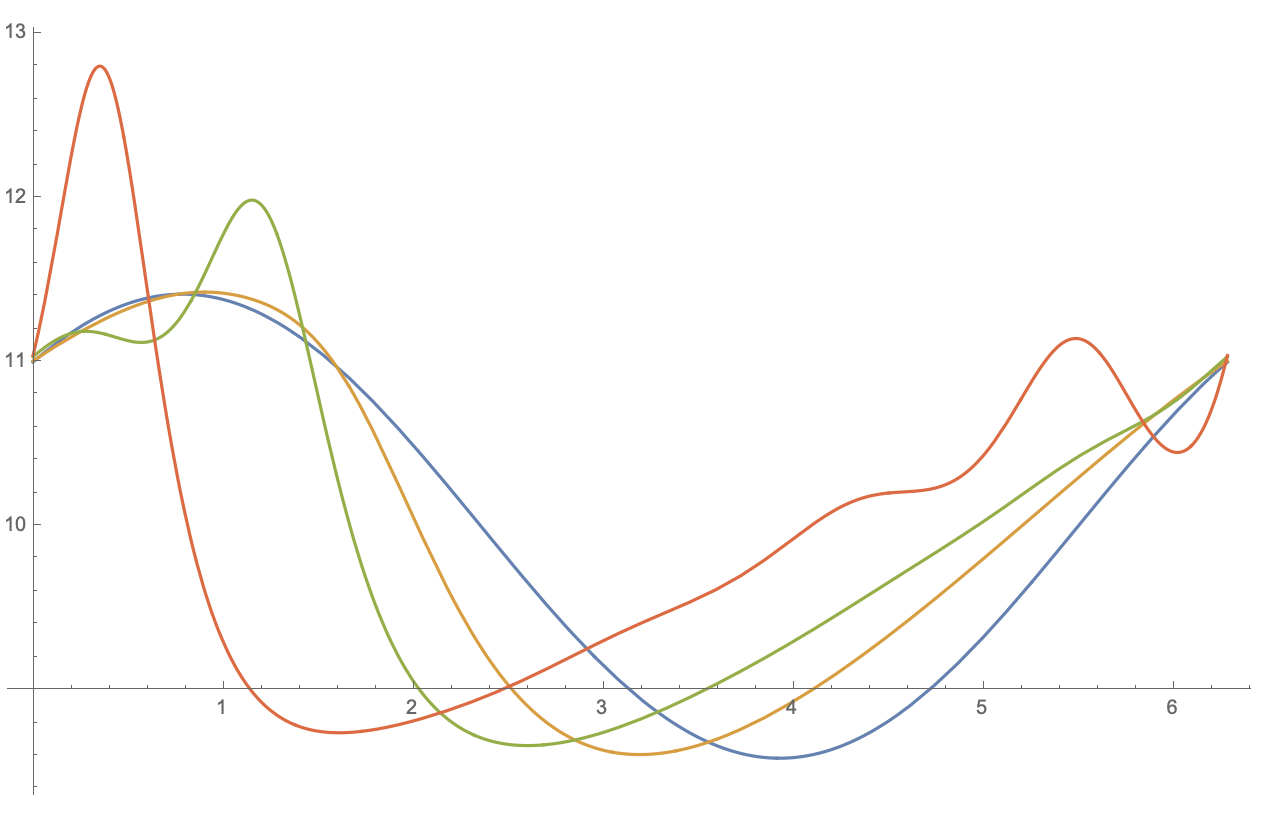}
\caption{The is the graph of $\tau$ at times $t=0, 1.26, 2.53,$ and $3.685$.}
\end{figure}

We may conclude that the evolution equation $(6)$ demonstrates some typical behavior for non-linear wave equations.

\subsection{Numerical Evidence for the Stability of Helices}
In Section 2.3 of this chapter, we have proven that helix stationary solutions of the equation $(6)$ are $L^2(\mathbb{R})$ linearly stable. In this section, we will examine small perturbations of the helix solutions and observe that there is some evidence that helices have non-linear stability as well.

To that end, let 
$$\tau_0(s)=1+\frac{\sin(s)}{100}$$
and let $\tau(s,t)$ be the solution to the initial value problem with $\tau(s,0)=\tau_0(s)$. Let $\| \cdot \|_2$ be the usual norm on $L^2(\mathbb{T})=L^2([0,2\pi]/_{per})$. The function $\tau_0$ adds only a small initial perturbation to the helix solution $\tau \equiv 1$. Indeed, we can compute
$$\| \tau(s,0)-1\|_2 \approx 0.0177245$$
Now we define the function
$$S(t) := \| \tau(\cdot,t)-1\|_2$$
which measures how much our solution for $\tau$ will deviate from the the constant torsion of a helix. Numerical integration of our (numerical) solution for $\tau(s,t)$ allows us to graph $S(t)$ and see whether the solution is non-linearly stable in $L^2(\mathbb{T})$ norm.
Below is an approximation of $S(t)$ for $t\in [0,50]$:
\begin{figure}[h]
\centering
\includegraphics[width=0.7\textwidth]{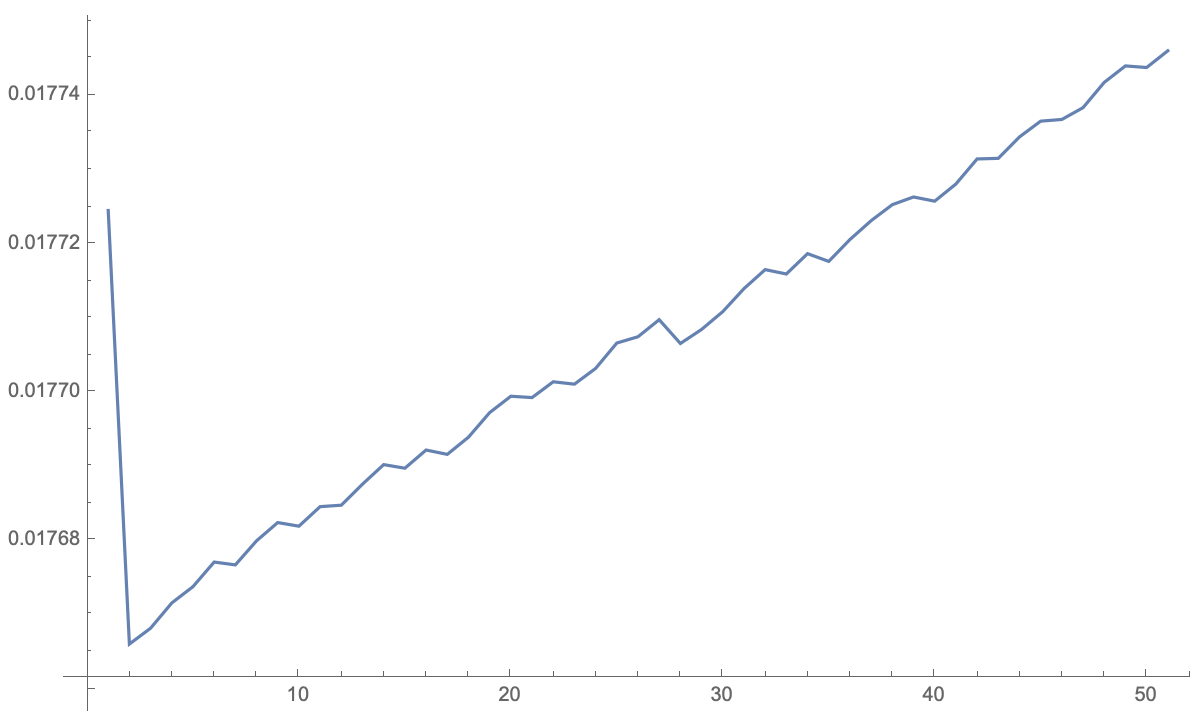}
\end{figure}

Barring numerical errors that occur when integrating differential equations over long time intervals, it appears that the solution $\tau(s,t)$ stays relatively close to $1$ in $L^2(\mathbb{T})$ norm, since $S(t)\leq S(0)$ for most of the time interval. We can also glean that asymptotic stability is unlikely, since it appears that the difference, $S(t)$, does not decrease to zero in time. We might conjecture that helix solutions are non-linearly stable, but the evidence for this is certainly not definitive.
\subsection{Solutions with Non-Constant Curvature}
We recall equation $(5)$, which is the evolution equation for the torsion with arbitrary curvature:
$$\tau_t = \kappa D_s(\tau^{-1/2})+D_s\bigg(\frac{D_s^2(\tau^{-1/2})-\tau^{3/2}}{\kappa}\bigg)$$
Our goal in this subsection is to numerically integrate the above inhomogeneous PDE for a certain choice of curvature and initial torsion. We will choose our curvature to be periodic, $\kappa(s) = 2+\cos{s}$ and our initial torsion will be $\tau(s,0)=1+\frac{\sin{s}}{10}$. To observe how the inhomogeneous PDE differs from the PDE with constant coefficients, we also compare the solution found with the same initial torsion but with constant curvature equal to $2$ (the average value of our choice of $\kappa(s)$ over the interval $[0,2\pi]$). The results of numerically solving the PDE are depicted on the next page.

Obtaining analytic results on the inhomogeneous PDE is obstructed by the inconvenient algebraic nature of the PDE (e.g. it is not even formally integrable). It would be interesting to tackle the inhomogeneous case, especially if one could obtain such results regardless of the choice of curvature. 
\begin{figure}[h]
\centering
\includegraphics[width=0.7\textwidth]{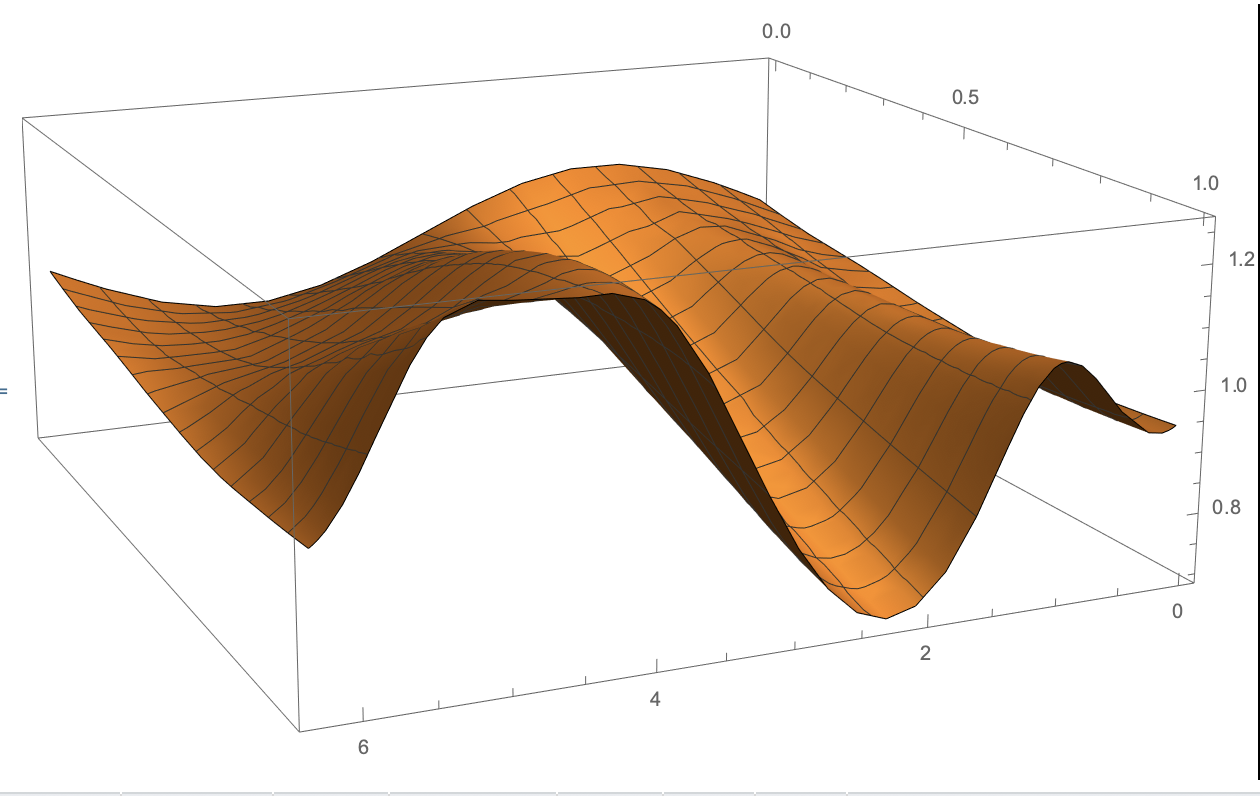}
\caption{The graph of the numerical solution for $\tau(s,t)$ with $(s,t)\in[0,2\pi]\cross[0,1]$.}
\end{figure}
\begin{figure}[h]
\centering
\includegraphics[width=0.7\textwidth]{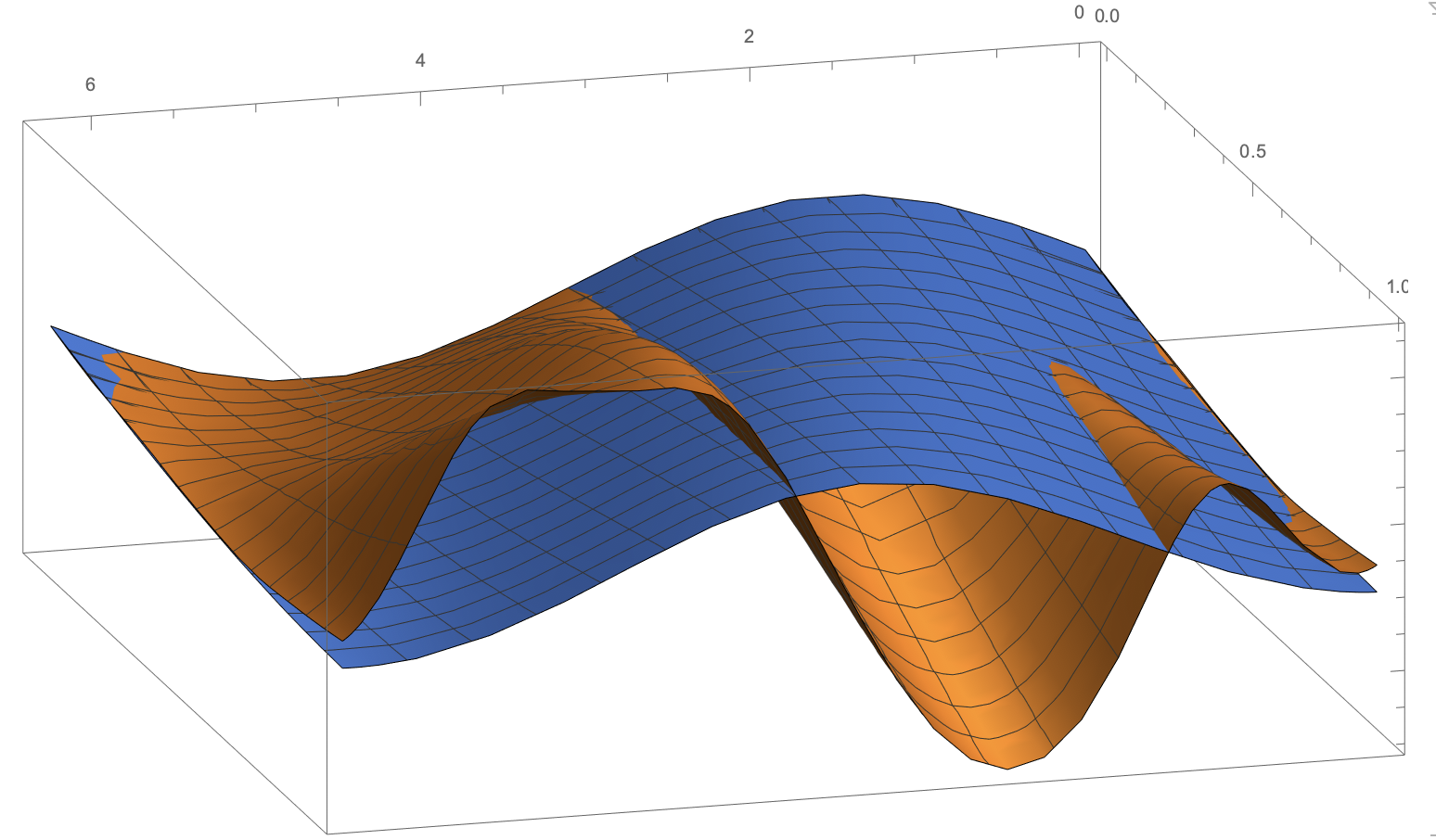}
\caption{The graphs of the numerical solutions for $\tau(s,t)$ with $(s,t)\in[0,2\pi]\cross[0,1]$ and with the blue graph corresponding to the solution with constant curvature equal to $2$. Observe that the solution to the inhomogeneous PDE more rapidly differs from the initial torsion.}
\end{figure}
\section{Computer Code}
The graphs in this chapter were rendered in Mathematica, which we also used to numerically solve the PDE. The code we typically used to perform the latter is similar to the following:
\begin{lstlisting}[language=Mathematica]
NDSolve[{D[ta[s, ti], ti] == (Cos[s] + 2)*D[ta[s, ti]^(-1/2), s] + 
    D[(D[ta[s, ti]^(-1/2), {s, 2}] - ta[s, ti]^(3/2))/(Cos[s] + 2), 
     s], ta[s, 0] == 1 + Sin[s]/10, 
  ta[0, ti] == ta[2 Pi, ti]}, {ta}, {s, 0, 2 Pi}, {ti, 0, 1}, 
 MaxStepSize -> 0.01]
\end{lstlisting}
Afterwards, we can recover the associated curve (after choosing initial conditions) by integrating the Frenet-Serret equations like in:
\begin{lstlisting}[language=Mathematica]
Manipulate[
 eqns = {t'[s] == \[Kappa][s] n[s], 
   n'[s] == -\[Kappa][s] t[s] - \[Tau][s] b[s], 
   b'[s] == \[Tau][s] n[s], r'[s] == t[s], t[0] == t0, n[0] == n0, 
   b[0] == b0, r[0] == r0};
 
 \[Kappa][s_] := Sin[s] + 5;
 \[Tau][s_] := ta[s, ti] /. s1;
 {t0, n0, b0} = Orthogonalize[{{1, 0, 0}, {0, 1, 0}, {0, 0, 1}}];
 r0 = {0, 0, 0};
 sol = First@NDSolve[eqns, {r, t, n, b}, {s, 0, 2 Pi}];, {ti, 0, 4}]
\end{lstlisting}

\chapter{An Interpolation from Sol to Hyperbolic Space}
\section{The Basic Structure}

In this section, we collect some basic facts about all of the $G_\alpha$ groups. The main idea is that a fruitful way of analyzing the geometry of these Lie groups is to first understand the geodesic flow, and this is the setting in which we will continue our analysis for the remainder of this paper.

The principal object of our study will be a one-parameter family of three-dimensional Lie groups, whose Lie algebras are of Type VI in Bianchi's classification, as elaborated in \cite{LB}. For ease of computation, we can also construct our groups as certain subgroups of $GL_3(\mathbb{R})$, and to that end we let
$$G_\alpha=\bigg\{ \begin{pmatrix} e^z&0&x \\ 0&e^{-\alpha z}&y\\0&0&1\end{pmatrix}\bigg| x,y,z\in\mathbb{R}\bigg\} \textrm{ for all } -1\leq\alpha\leq1$$
We can consider each $G_\alpha$ as a matrix group or, equivalently, as $\mathbb{R}^3$ with the following group law:
$$(x,y,z)\ast(x',y',z')=(x'e^z+x, y'e^{-\alpha z}+y, z'+z).$$
Then, $\mathbb{R}^3$ with this group law has the following left-invariant metric
$$ds^2=e^{-2z}dx^2+e^{2\alpha z}dy^2+dz^2.$$
The Lie algebra of $G_\alpha,$ which we denote $\frak{g}_\alpha$, has the following orthonormal basis:
$$\bigg\{ X=\begin{pmatrix} 0&0&1 \\ 0&0&0\\0&0&0\end{pmatrix}\quad  Y=\begin{pmatrix} 0&0&0 \\ 0&0&1\\0&0&0\end{pmatrix}\quad  Z=\begin{pmatrix} 1&0&0 \\ 0&-\alpha&0\\0&0&0\end{pmatrix}\bigg\}$$
or
\begin{equation}X=e^z \frac{\partial}{\partial x}, Y=e^{-\alpha z}\frac{\partial}{\partial y}, Z=\frac{\partial}{\partial z}.\end{equation}
So, the structure equations are:
\begin{equation}[X,Y]=0\quad [Y,Z]=\alpha Y\quad [X,Z]=-X\end{equation}
The behavior of $\frak{g}_\alpha$ coincides with the Lie algebras of Type VI (in Bianchi's Classification) when $0<|\alpha|<1$ and we have three limiting cases: $\alpha=1,$ which is a Bianchi group of type VI$_{0}$, $\alpha=0,$ which is a Bianchi group of type III, and  $\alpha=-1$, which is a Bianchi group of type V. The intermediate cases are not unimodular, unlike the limiting cases, and that may also be of interest. As a consequence of the classification done in \cite{LB}, no two of the Lie algebras $\frak{g}_\alpha$ are isomorphic, hence
\begin{proposition}[Bianchi, \cite{LB}]
No two of the $G_\alpha$ are Lie group isomorphic.
\end{proposition}
One might ask what occurs if we let the parameter $|\alpha|>1$, and the answer is simple. In this case, the Lie algebra will be isomorphic to that of one of our $G_\alpha$ groups with $|\alpha|<1$. Since we restrict ourselves to simply connected Lie groups (indeed, Lie groups that are  diffeomorphic to $\mathbb{R}^3$), this means that the corresponding Lie groups are isomorphic as well. 

Now we recall an essential fact from Riemannian geometry. Given a smooth manifold with a Riemannian metric, there exists a unique torsion-free connection which is compatible with the metric, which is called the Levi-Civita connection. This can be easily proven with the following lemma, the proof of which may be found in \cite{K}.
\begin{lemma}[Koszul Formula]
Let $\nabla$ be a torsion-free, metric connection on a Riemannian manifold $(M,g)$. Then, for any vector fields $X,Y,$ and $Z,$ we have:
$$2g(\nabla_X Y, Z)= X(g(Y,Z))+Y(g(X,Z))-Z(g(X,Y))+$$
$$+g([X,Y],Z)-g([X,Z],Y)-g([Y,Z],X)$$
\end{lemma}
Using the Koszul Formula, we can easily compute the Levi-Civita connection $\nabla$ for each $G_\alpha$ from equation $(18)$. We get:
\begin{proposition}
The Levi-Civita connection of $G_\alpha$, with its left-invariant metric, is completely determined by
$$\begin{pmatrix} \nabla_X X && \nabla_X Y&&\nabla_X Z\\ \nabla_Y X && \nabla_Y Y&&\nabla_Y Z\\ \nabla_Z X && \nabla_Z Y&&\nabla_Z Z\end{pmatrix}=\begin{pmatrix} Z && 0&&-X\\ 0 && -\alpha Z&&\alpha Y\\ 0 && 0&&0\end{pmatrix}$$
where $\{X,Y,Z\}$ is the orthonormal basis of the Lie algebra, as in $(1)$.
\end{proposition}
The coordinate planes play a special role in the geometry of $G_\alpha.$ Each group (which is diffeomorphic to $\mathbb{R}^3$) has three foliations by the $XZ,$ $YZ,$ and $XY$ planes. It will be worthwhile to compute the curvatures of these surfaces in $G_\alpha$. The sectional curvature can be computed easily from the Levi-Civita Connection, while the extrinsic (Gaussian) curvature and mean curvature are computed using the Weingarten equation. Lastly, for surfaces in a Riemannian 3-manifold we have the relation 
\begin{equation} Intrinsic = Extrinsic + Sectional \end{equation}
from Gauss' Theorema Egregium. For details, see the first chapter in \cite{N}. Straightforward computations yield the following proposition and its immediate consequences.
\begin{proposition} The relevant curvatures of the coordinate planes:
\begin{center}
 \begin{tabular}{|c c c c c|} 
 \hline
Plane& Sectional & Intrinsic & Extrinsic (Gaussian) & Mean  \\ 
 \hline\hline
XY& $\alpha$ & 0 & $-\alpha$ & $(1-\alpha)/2$ \\ 
 \hline
XZ& -1 & -1 & 0 &0 \\
 \hline
YZ& $-\alpha^2$ & $-\alpha^2$ & 0 &0\\ 
 \hline
\end{tabular}
\end{center}
\end{proposition}
\begin{corollary}
The $XY$ plane is a minimal surface (having vanishing mean curvature) in $G_\alpha$ if and only if $\alpha=1$, and the $XZ$ and $YZ$ planes are minimal for all $\alpha$. Also, the $XY$ plane is always a constant-mean-curvature surface.
\end{corollary}
We will later strengthen this corollary and get that the $XZ$ and $YZ$ planes are geodesically embedded.

It can be easily seen that $G_1$ is a model for the Sol geometry, $G_0$ is a model of $\mathbb{H}^2\cross\mathbb{R}$, and $G_{-1}$ is a model of $\mathbb{H}^3,$ or hyperbolic space. We can also compute the Ricci and scalar curvatures as well and we get that the scalar curvature of $G_\alpha$ is $S_\alpha=2\alpha-2-2\alpha^2$. $S_\alpha$ attains its maximum over the family at $S_{1/2}=-3/2$, and the minimum is attained at $S_{-1}=-6$ (as might be expected). Therefore, the group $G_{1/2}$ may also be considered a special case: the member of the interpolation that maximizes scalar curvature. Moreover, $S_\alpha$ is symmetric in the positive side of the family, in the sense that for all non-negative $\alpha$, $S_{\alpha}=S_{1-\alpha}$. 

Other self-evident properties of the coordinate planes could be stated, but we let the reader find these himself.
Now, we turn our attention to the geodesic flow of $G_\alpha$. Rather than attempting to derive analytic formulas for the geodesics from the geodesic equation, as done for Sol in \cite{T}, we restrict the geodesic flow to the unit-tangent bundle and consider the resulting vector field. The idea of restricting the geodesic flow to $S(G_1)$ for Sol was first explored by Grayson in his thesis \cite{G2} and then used by Richard Schwartz and the current author to characterize the cut locus of the origin of Sol in \cite{MS}. We recall that the cut locus of a point $p$ in a Riemannian manifold $(M,g)$ is the locus of points on geodesics starting at $p$ where the geodesics cease to be length minimizing. 

Now, we extend the previous ideas to the other $G_\alpha$ groups. Indeed, consider $\frak{g}_\alpha$ and let $S(G_\alpha)$ be the unit sphere centered at the origin in $\mathfrak{g}_\alpha.$ Suppose that $\gamma(t)$ is a geodesic parametrized by arc length such that $\gamma(0)$ is the identity of $G$. Then, we can realize the development of $\gamma'(t)$, the tangent vector field along $\gamma$, as a curve on $S(G_\alpha),$ which will be the integral curve of a vector field on $S(G_\alpha)$ denoted by $\Sigma_\alpha.$ We can compute the vector field $\Sigma_\alpha$ explicitly. 
\begin{proposition}
For the group $G_\alpha$ the vector field $\Sigma_\alpha$ is given by
$$\Sigma_{\alpha}(x,y,z)= (xz, -\alpha yz, \alpha y^2-x^2)$$
\end{proposition}
\begin{proof}
Since we are dealing with a homogeneous space (a Lie group) it suffices to examine the infinitesimal change of $V=\gamma'(0)=(x,y,z).$ We remark that parallel translation along $\gamma$ preserves $\gamma'$ because we have a geodesic, but parallel translation does not preserve the constant (w.r.t. the left-invariant orthonormal frame) vector field $V=\gamma'(0)$ along $\gamma$. Indeed, the infinitesimal change in the constant vector field $V$ as we parallel translate along $\gamma$ is precisely the covariant derivative of $V$ with respect to itself, or $\nabla_V V.$ Then, our vector field on $S(G_\alpha)$ is precisely:
$$\Sigma_\alpha=\nabla_V (\gamma' - V)=\nabla_V \gamma' -\nabla_V V=-\nabla_V V.$$
We also remark that since $V$ is a constant vector field, $\Sigma_\alpha$ is determined completely by the Levi-Civita connection that we previously computed, and this computation is elementary.
\end{proof}
We remark where the equilibria points of $\Sigma_\alpha$ are, since these correspond to straight-line geodesics. When $0<\alpha\leq 1,$ the equilibria are
$$\bigg(\pm\sqrt{\frac{\alpha}{1+\alpha}},\pm\sqrt{\frac{1}{1+\alpha}},0\bigg) \textrm{ and } (0,0,\pm1).$$
When $\alpha=0$, the set $\{X=0\}\cap S(G_\alpha)$ is an equator of equilibria, and when $\alpha<0,$ the only equilibria are at the poles.
A glance at $\Sigma_\alpha$ gets us our promised strengthening of Corollary 4:
\begin{corollary}
The $XZ$ and $YZ$ planes are geodesically embedded. The $XY$ is never geodesically embedded, even when it is a minimal surface (i.e. for $\alpha=1$).
\end{corollary}
Consider the complement of the union of the two planes $X=0$ and $Y=0$ in $\frak{g}_\alpha$. This is the union of four connected components, which we call \textit{sectors}. Since the $XZ$ and $YZ$ planes are geodesically embedded, we have
\begin{corollary}
The Riemannian exponential map, which we denote by $E$, preserves each sector of $\frak{g}_\alpha$. In particular, if $(x,y,z)\in \frak{g}_\alpha$ is such that $x,y>0,$ then $E(x,y,z)=(a,b,c)$ with $a,b>0$. 
\end{corollary}
We will use Corollary 6 often and without mentioning it. We make the following key observation about $\Sigma_\alpha$.
\begin{proposition}
The integral curves of $\Sigma_\alpha$ are precisely the level sets of the function $H(x,y,z)=\abs{x}^\alpha y$ on the unit sphere. 
\end{proposition}
\begin{proof}
Without loss of generality, we consider the positive sector. We recall that the symplectic gradient is the analogue in symplectic geometry of the gradient in Riemannian geometry. In the case of the sphere with standard symplectic structure, the symplectic gradient is defined by taking the gradient of the function $H$ (on the sphere) and rotating it 90 degrees counterclockwise. Doing this computation for $H,$ yields $\nabla_{sym}H=x^{\alpha-1}\cdot\Sigma_\alpha(x,y,z).$ Since this vector field is the same as the structure field up to a scalar function, the desired property follows. 
\end{proof}
\begin{remark}
$\Sigma_\alpha$ is a Hamiltonian system in these coordinates if and only if $\alpha=1,$ i.e. for Sol ($G_1$) 
\end{remark}
We finish this section with a conjecture that, if true, provides some connection between the groups in the $G_\alpha$ family, for $\alpha \in [-1,1]$. We recall that the \textit{volume entropy}, $h$, of a homogeneous Riemannian manifold $(M,g)$ is a measure of the volume growth in $M$. We can define
$$h(M,g):=\lim_{R\rightarrow \infty} \frac{\log(\textrm{Vol } B(R))}{R}$$
where $B(R)$ is a geodesic ball of radius $R$ in $M$. Since $G_{-1}$ is a model of Hyperbolic space and $G_{1}$ is Sol, we know that $h(G_{-1})=2$ and $h(G_{1})=1$ (see \cite{S}). Based on this, we conjecture that:
\begin{conjecture}
$h(G_\alpha)$ is a monotonically decreasing function of $\alpha$ for $\alpha \in [-1,1]$.
\end{conjecture}
\section{The Positive Alpha Family}
In this section, we start to explore the positive $\alpha$ side of the family. These geometries exhibit common behaviors such as geodesics always lying on certain cylinders, spiraling around in a "periodic-drift" manner. A natural way to classify vectors in the Lie algebra is by how much the associated geodesic segment under the Riemannian exponential map spirals around its associated cylinder. This classification allows us to discern how the exponential map behaves with great detail.
\subsection{Grayson Cylinders and Period Functions}
More than a few of our theorems in Section 2 may be considered generalizations of results in \cite{G2} and \cite{MS}. To begin our analysis, we study the \textit{Grayson Cylinders} of the $G_\alpha$ groups. 
\begin{definition}
We call the level sets of $H(x,y,z)=|x|^\alpha y$ that are closed curves \textbf{loop level sets}.
\end{definition}
The proof of the following theorem follows a method first used in \cite{G2} for Sol.
\begin{theorem}[The Grayson Cylinder Theorem]
Any geodesic with initial tangent vector on the same loop level set as 
$$\bigg(\beta\sqrt{\frac{\alpha}{1+\alpha}}, \frac{\beta}{\sqrt{1+\alpha}}, \sqrt{1-\beta^2}\bigg), \beta\in[0,1]$$
lies on the cylinder given by 
$$w^2+e^{2z}+\frac{1}{\alpha}e^{-2\alpha z}=\frac{1+\alpha}{\alpha}\cdot\frac{1}{\beta^2}$$
where $w=x-y\sqrt{\alpha}$. We call these cylinders "Grayson Cylinders".
\end{theorem}
If we consider Grayson Cylinders as regular surfaces in $\mathbb{R}^3$ with the ordinary Euclidean metric, then a simple derivation of their first and second fundamental forms reveals that they are surfaces with Gaussian curvature identically equal to zero. Hence, they are locally isometric to ordinary cylinders, and, because they are also diffeomorphic to ordinary cylinders, Grayson Cylinders are in fact isometric to ordinary cylinders, for all choices of $\alpha$ and $\beta$.

It is easier to gauge the shape of a Grayson Cylinder by looking at its projection onto the planes normal to the line $x-y\sqrt{\alpha}$, or, alternatively, as the implicit plot of a function of the two variables $w$ and $z$ defined in the statement of Theorem 8. It appears that as $\alpha$ is fixed, the Grayson Cylinders limit to two "hyperbolic slabs" as $\beta$ goes to zero. Alternatively, as $\beta$ is fixed and $\alpha$ varies, it appears that one side of the Grayson Cylinder is ballooning outwards. In Figures 1 and 2, we have some examples generated with the Mathematica code provided at the end of this chapter.
\begin{figure}[H]
\centering
\begin{subfigure}{0.45\textwidth}
\includegraphics[width=\textwidth]{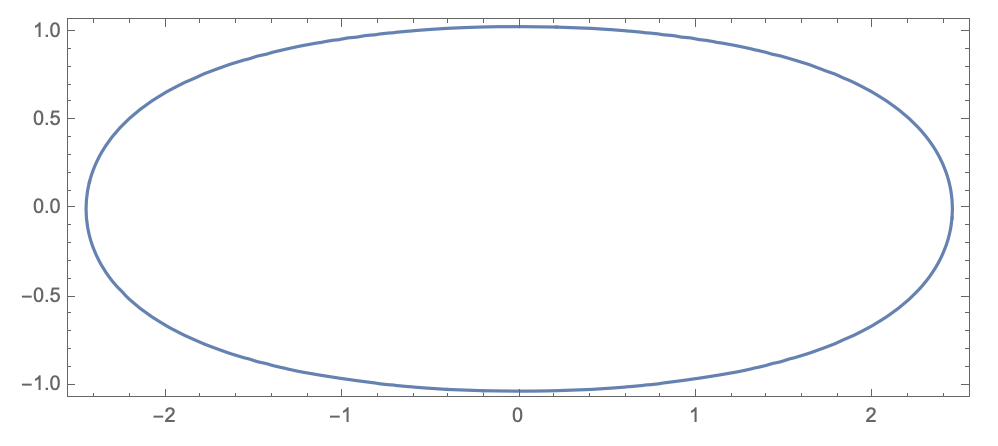}
\caption{$\alpha=1$ and $\beta=1/2$.}
\end{subfigure}
\hfill
\begin{subfigure}{0.45\textwidth}
\includegraphics[width=\textwidth]{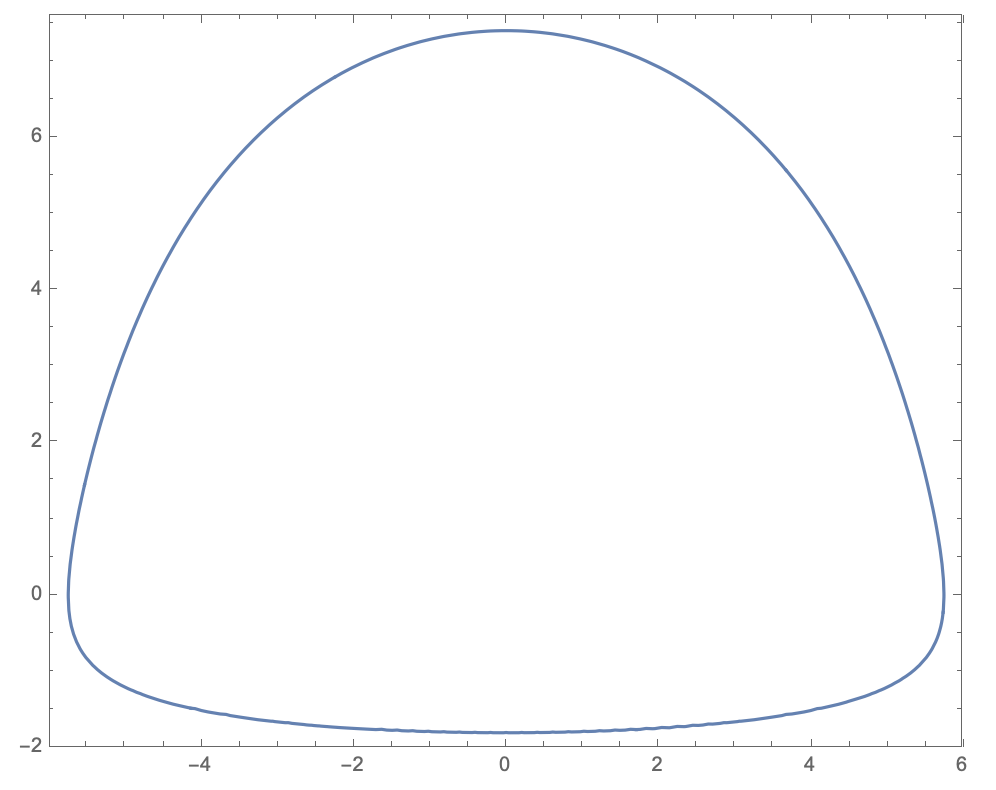}
\caption{$\alpha=1/4$ and $\beta=1/2$.}
\end{subfigure}
\caption{Slices of Grayson Cylinders with varying $\alpha$}
\end{figure}
\begin{figure}[H]
\centering
\begin{subfigure}{0.45\textwidth}
\includegraphics[width=\textwidth]{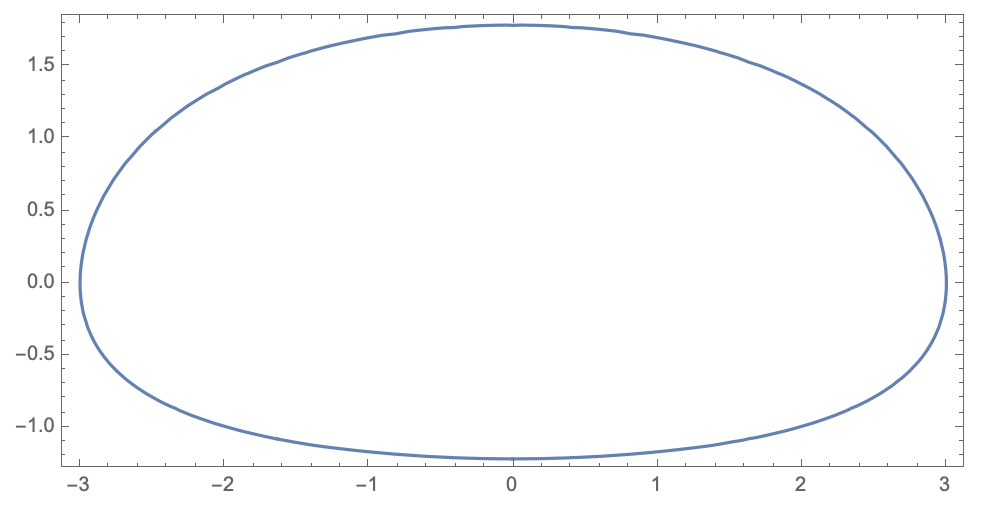}
\caption{$\alpha=1/2$ and $\beta=1/2$.}
\end{subfigure}
\hfill
\begin{subfigure}{0.45\textwidth}
\includegraphics[width=\textwidth]{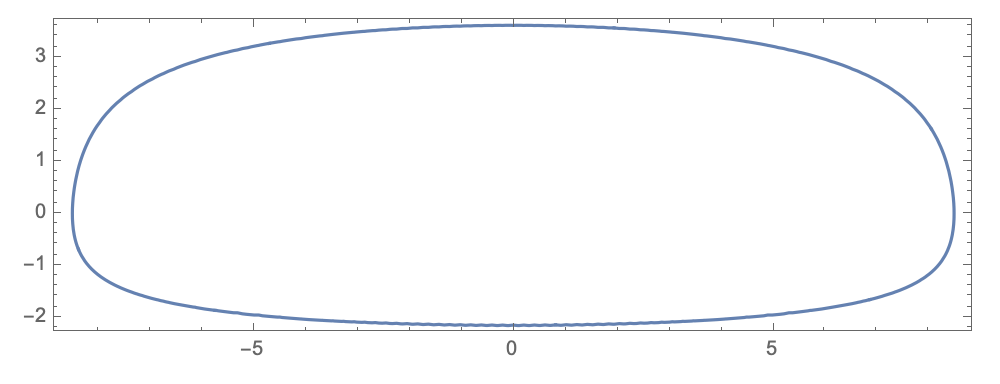}
\caption{$\alpha=1/2$ and $\beta=1/5$.}
\end{subfigure}
\caption{Slices of Grayson Cylinders with varying $\beta$}
\end{figure}
We denote loop level sets by $\lambda$. Each loop level set has an associated period, $P_\lambda,$ which is the time it takes for a flowline to go exactly once around $\lambda$, and it suffices to compute the period at one vector in a loop level set to know $P_\lambda$. We can compare $P_\lambda$ to the length $T$ of a geodesic segment $\gamma$ associated to a flowline that starts at some point of $\lambda$ and flows for time $T.$ We call $\gamma$ $\textit{small, perfect,}$ or $large$ whenever $T<P_\lambda,$  $T=P_\lambda,$ or $T>P_\lambda,$ respectively. It seems that this "classification" of geodesic segments in $G_\alpha$ is ideal. For instance, it was shown in \cite{MS} that a geodesic in Sol is length-minimizing if and only if it is small or perfect. We now derive an integral formula for $P_\lambda$ and simplify the integral in two special cases. 
\begin{proposition}
Let $\lambda$ be the loop level set associated to the vector 
$$V_\beta=\bigg(\beta\sqrt{\frac{\alpha}{1+\alpha}}, \frac{\beta}{\sqrt{1+\alpha}}, \sqrt{1-\beta^2}\bigg),$$
then 
$$P_\lambda(\beta) = \int_{-t_1}^{t_0} \frac{2dt}{\sqrt{1-\frac{\beta^2}{\alpha+1}(\alpha e^{2t}+e^{-2\alpha t})}}$$
where $t_0$ and $t_1$ are the times it takes to flow from $V_\beta$ to the equator of $S(G_\alpha)$ in the direction of, and opposite to the flow of $\lambda$, respectively. 
\end{proposition}
\begin{remark}
The times $t_0$ and $t_1$ are precisely when the flat flow lines hit the unit circle, or when 
\begin{equation}\alpha e^{2t_0}+e^{-2\alpha t_0}=\frac{\alpha+1}{\beta^2} \textrm{ and } \alpha e^{-2t_1}+e^{2\alpha t_1}=\frac{\alpha+1}{\beta^2}\end{equation}
\end{remark}
Stephen Miller helped us to numerically compute the period function for any choice of positive $\alpha$. An explicit formula for the period function in Sol ($G_1$) was derived in \cite{MS} and \cite{T}. It is:
$$P_\lambda(\beta)=\frac{4}{\sqrt{1+\beta^2}}\cdot K\bigg(\frac{1-\beta^2}{1+\beta^2}\bigg)$$
where $K(m)$ is the complete elliptic integral of the first kind, with the parameter as in Mathematica.

A closed-form expression of $P_\lambda$ can also be obtained for $G_{1/2}$. Since elliptic integrals have been studied extensively and many of their properties are well-known, the following expression allows us to analyze $P_\lambda$ more easily. 
\begin{corollary}
When $\alpha=1/2,$ or for the group $G_{1/2}$, the period function is given by 
$$P_\lambda(\beta)=\frac{4\sqrt{3}}{\beta\sqrt{e^{t_0-t_1}+2e^{t_1}}}\cdot K\bigg(\frac{2(e^{t_1}-e^{-t_0})}{e^{t_0-t_1}+2e^{t_1}}\bigg)$$
where
$$t_0=\log\bigg(\frac{1}{\beta}\cdot\bigg(\frac{1}{(-\beta^3 + \sqrt{-1 + \beta^6})^{\frac{1}{3}}} + (-\beta^3 + 
       \sqrt{-1 + \beta^6})^{\frac{1}{3}}\bigg)\bigg)$$
        and
   $$t_1=\log\bigg(\frac{1}{2}\bigg(\frac{1}{\beta^2}+\frac{1}{\beta^4(-2+\frac{1}{\beta^6}+\frac{2\sqrt{-1+\beta^6}}{\beta^3})^{\frac{1}{3}}}+(-2+\frac{1}{\beta^6}+\frac{2\sqrt{-1+\beta^6}}{\beta^3})^{\frac{1}{3}}\bigg)\bigg)$$
\end{corollary}
Here we state an essential fact, which is forthrightly supplied to us by the above expression of $P_\lambda(\beta)$ in terms of an elliptic integral. We have:
\begin{proposition}
$$\frac{d}{d\beta}\big(P_\lambda(\beta)\big)<0$$
when $\alpha=1$ and $\alpha=1/2$. Moreover, for $\alpha=1$,
$$\lim_{\beta\to 1} P_\lambda(\beta)=\pi\sqrt{2}$$
and for $\alpha=1/2$,
$$\lim_{\beta\to 1} P_\lambda(\beta)=2\pi.$$
\end{proposition}
We could not find a similar formula for $P_\lambda$, when $\alpha$ is not $1$ or $1/2$, in terms of elliptic integrals or hypergeometric functions. This is unfortunate, as Proposition 10 is vital to the method here and in \cite{MS} to characterize the cut locus of $G_1$ and $G_{1/2}$. However, we can still numerically compute the period function as presented at the end of this chapter, and this allows us to conjecture:
\begin{conjecture} Let $P(\beta)$ be the period function in $G_\alpha$, then
$$\frac{d}{d\beta}\big(P_\lambda(\beta)\big)<0$$
and
$$\lim_{\beta\to 1} P(\beta)=\frac{\pi\sqrt{2}}{\sqrt{\alpha}}.$$
\end{conjecture}
This conjecture would lend something quantitative to the idea that the bad behavior (or the cut locus) of $G_\alpha$ dissipates at infinity as Sol interpolates to $\mathbb{H}^2\cross\mathbb{R}$. With Stephen Miller's Mathematica code, we get the following numerical evidence for Conjecture 3.
\begin{center}
 \begin{tabular}{|c c c|} 
 \hline
$\alpha$ & Numerical Value of $P(\alpha,\beta=.999)$ & $\pi\sqrt{2}/\sqrt{\alpha}$  \\ 
 \hline\hline
 0.1 & 14.0792 & 14.0496 \\
  \hline
 0.2 & 9.94735 & 9.93459 \\
  \hline
 0.3 & 8.11985 & 8.11156 \\
  \hline
 0.4 & 7.03114 & 7.02481 \\
  \hline
 0.5 & 6.28842 & 6.28319 \\
  \hline
 0.6 & 5.7403 & 5.73574 \\
  \hline
 0.7 & 5.31436 & 5.31026 \\
  \hline
 0.8 & 4.97106 & 4.96729 \\
  \hline
 0.9 & 4.68673 & 4.68321 \\
  \hline
 1. & 4.44622 & 4.44288 \\
  \hline
\end{tabular}
\end{center}

\subsection{Concatenation and Some Other Useful Facts}
An important property which extends from Sol to $G_\alpha$ for $0<\alpha<1$ is that the loop level sets are symmetric with respect to the plane $Z=0.$ This simple observation allows the technique of $\it{concatenation},$ essential to the analysis in \cite{MS}, to be replicated for all of the $G_\alpha$ groups, when $\alpha>0$. For the interested reader, Richard Schwartz's Java program \cite{S2} uses concatenation to generate geodesics and geodesic spheres in Sol, and a modified version of this program can generate the spheres and geodesics in any $G_\alpha$ group as well as in other Lie groups, such as Nil. As an illustration of the power of this technique in numerical simulations, we present in Figure 3 the geodesic spheres of radius around $5$ in four different geometries. The spheres are presented from the same angle, the purple line is the $z$ axis, and the red lines are the horizontal axes. The salient phenomenon is that one "lobe" of the sphere is contracting as $\alpha$ goes to zero. Qualitatively, this corresponds to the dissipation of the "bad" behavior (or the cut locus) of $G_\alpha$ as $\alpha$ tends to zero, since the amount of shear diminishes.
\begin{figure}[H]
\centering
\begin{subfigure}{0.45\textwidth}
\includegraphics[width=0.9\linewidth]{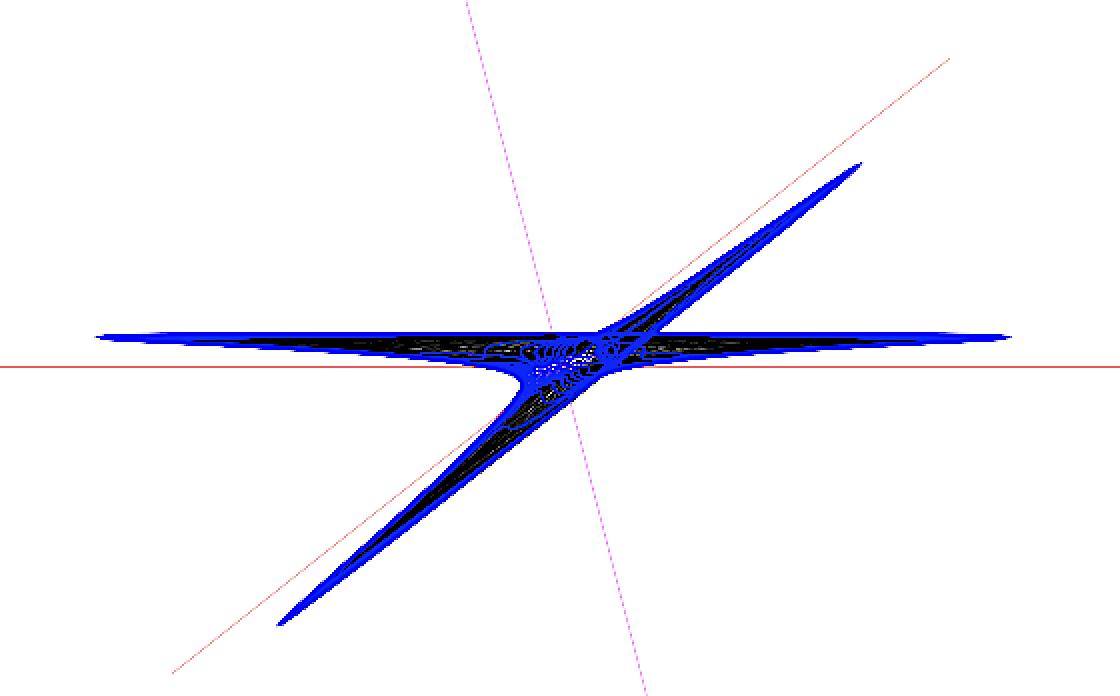}
\caption{In Sol, or the group $G_1$.}
\end{subfigure}
\hfill
\begin{subfigure}{0.45\textwidth}
\includegraphics[width=0.9\linewidth]{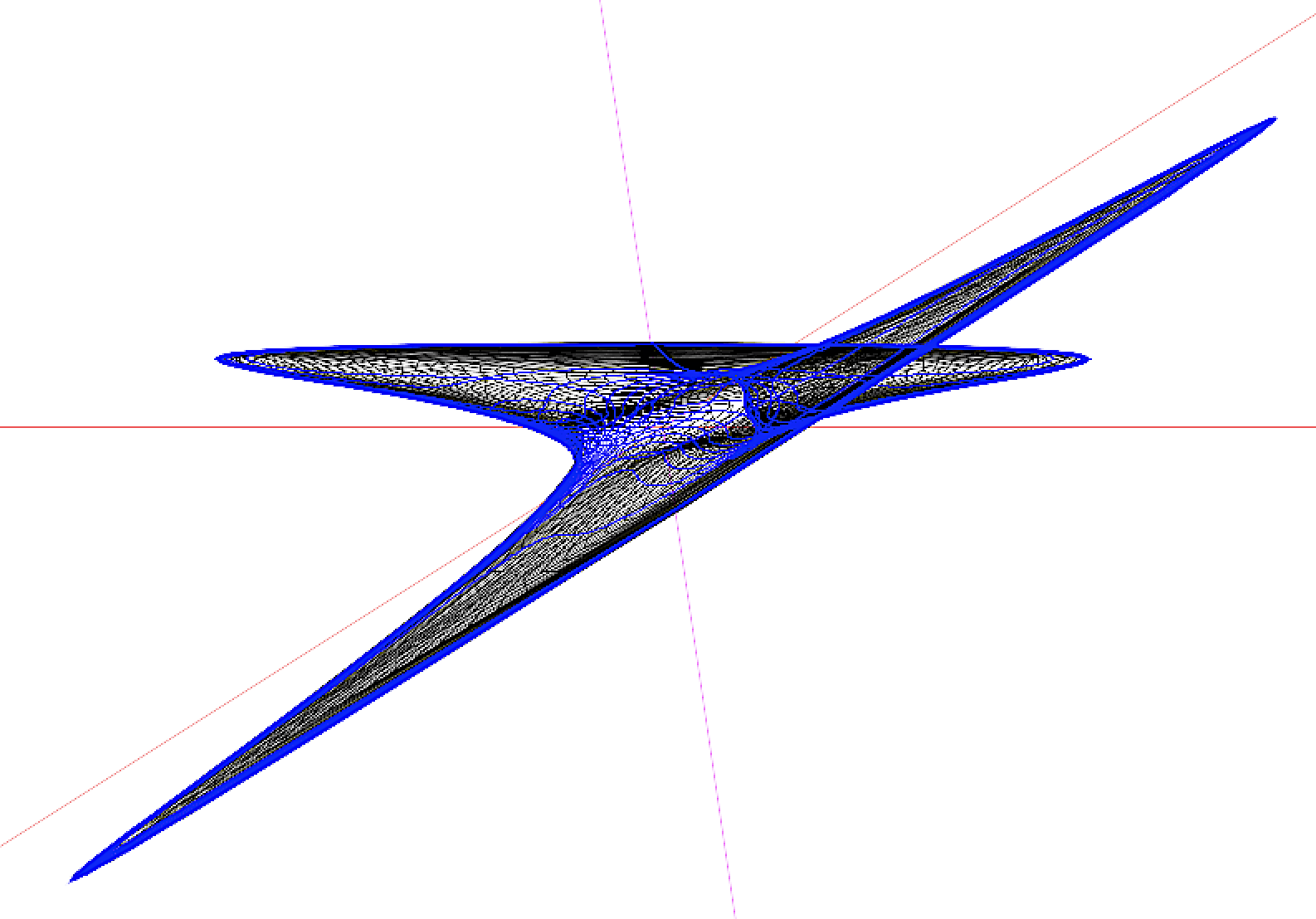}
\caption{In the group $G_{3/4}$.}
\end{subfigure}
\begin{subfigure}{0.45\textwidth}
\includegraphics[width=0.9\linewidth]{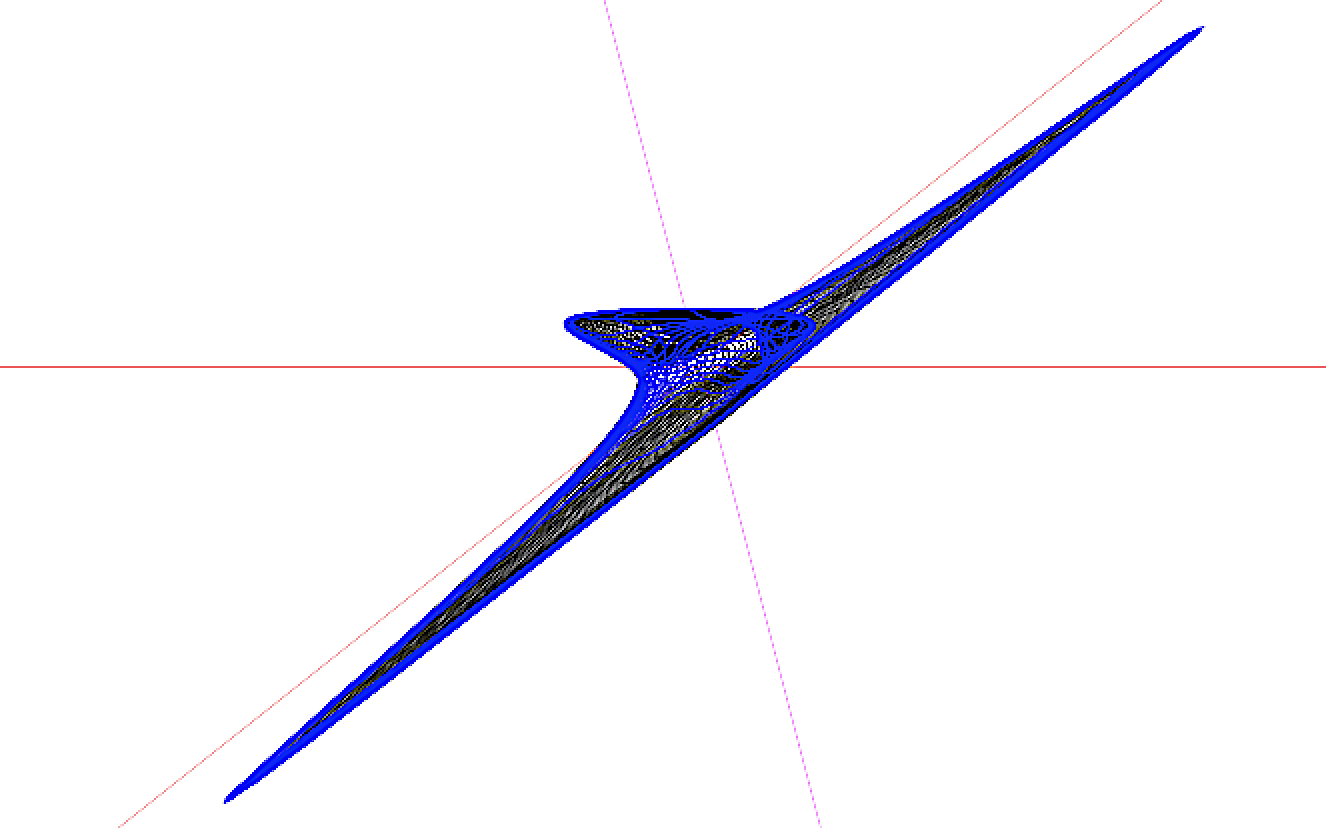}
\caption{In the group $G_{1/2}$}
\end{subfigure}
\begin{subfigure}{0.45\textwidth}
\includegraphics[width=0.9\linewidth]{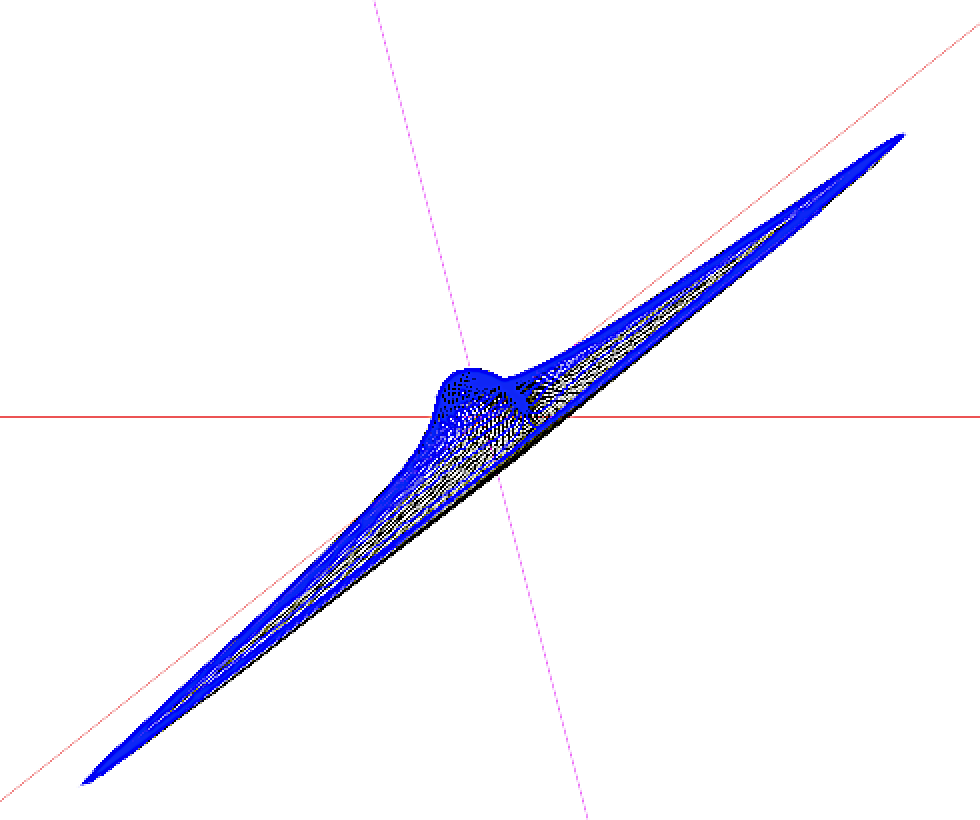}
\caption{In the group $G_0$, or $\mathbb{H}^2\cross \mathbb{R}$.}
\end{subfigure}
\caption{Geodesic spheres of radius around $5$ in four different geometries.}
\end{figure}
 Each flowline $\lambda$ of our vector field corresponds to a segment of a geodesic $\gamma$. Let $T$ be the time it takes to trace out $\lambda,$ then $T$ is exactly the length of $\gamma$ since we always take unit-speed geodesics. Let $L_\lambda$ be the far endpoint of $\gamma$ and consider the equally spaced times
$$0=t_0<t_1<\ldots <t_n=T$$
with corresponding points $\lambda_0, \ldots, \lambda_n$ along $\lambda.$ Then we have
$$L_\lambda=\lim_{n\to\infty} (\epsilon_n \lambda_0)\ast\ldots\ast(\epsilon_n \lambda_n),\quad \epsilon_n=T/(n+1)$$
where $\ast$ is the group law in $G_\alpha.$ The above equation is well-defined because the underlying space of both $G_\alpha$ and its Lie algebra is $\mathbb{R}^3.$ We will also use the notation $\lambda=a|b$ to indicate that we are splitting $\lambda$ into two sub-trajectories, $a$ and $b.$ 

The above equation yields:
$$L_\lambda=L_a \ast L_b$$
when $\lambda=a|b.$
We set $\epsilon_n \lambda_j=(x_{n,j},y_{n,j},z_{n,j})$.
Vertical displacements commute in $G_\alpha,$ therefore the third coordinate of
the far endpoint of $\gamma$ is given by
$$\lim_{n \to \infty} \sum_{j=0}^n z_{n,j}.$$
From this, and the symmetry of the flow lines with respect to $Z=0,$ we get the following lemmas. Since these results are an immediate extension of the results in \cite{MS}, we provide only sketches of their proofs.
\begin{lemma}
If the map $(x,y,z) \to (x,y,-z)$ exchanges the
two endpoints of the flowline $\lambda$ then the endpoints of the geodesic segment $\gamma$
both lie in the plane $Z=0$ and $L_\lambda$ is a
horizontal translation. In this case we call
$\lambda$ {\it symmetric\/}.
\end{lemma}
\begin{proof}
Since $\lambda$ is symmetric, the sum $\lim_{n \to \infty} \sum_{j=0}^n z_{n,j}$ vanishes, so the total vertical displacement is zero.
\end{proof}
\begin{lemma}
If $\lambda$ is not symmetric then we can write
  $\lambda=a|b|c$ where $a,c$ are either symmetric
  or empty, and $b$ lies entirely above or entirely
  below the plane $Z=0$.  Since $L_a$ and
  $L_c$ are horizontal translations -- or just
  the identity in the empty cases -- and
    $L_b$ is not such a translation, the endpoints of $\lambda$ are
  not in the same horizontal plane.
  \end{lemma}
 \begin{lemma} If $\lambda=a|b$, where both $a$ and $b$ are
symmetric, then both endpoints of
$\gamma$ lie in the plane $Z=0$.  We can do this whenever
$\lambda$ is one full period of a loop
level set.  Hence, a perfect
geodesic segment has both endpoints in the same
horizontal plane.
\end{lemma}
\begin{proof}
Since $\lambda=a|b$, we know that $L_\lambda=L_a\ast L_b$. Since $L_a$ and $L_b$ are horizontal translations by Lemma 16, it follows that $L_\lambda$ stays in the same $Z=0$ plane.
\end{proof}
\begin{lemma} If $\lambda_1$ and $\lambda_2$ are full trajectories of the same
  loop level set, then
we can write $\lambda_1=a|b$ and $\lambda_2=b|a$, which leads to
$L_{\lambda_2}=L_a^{-1}L_{\lambda_1}L_a$. Working this out with the group law in $G_\alpha$ gives:
$(a_1 e^{-z}, b_1 e^{\alpha z}, 0)=(a_2, b_2, 0)$, where $(x,y,z)=L_a$ and $(a_i, b_i, 0)=L_{\lambda_i}$.
In particular, we have $a_1^\alpha b_1=a_2^\alpha b_2$ and we call the function (of the flowlines)
$H_\lambda=\sqrt{|a_1^\alpha b_1|}$ the {\it holonomy invariant\/} of
the loop level set $\lambda.$ 
\end{lemma}

Let $E$ be the Riemannian
exponential map.
We call $V_+=(x,y,z)$ and $V_-=(x,y,-z)$, vectors in the Lie algebra, {\it partners\/}.
The symmetric trajectories discussed in Lemma 16 have endpoints
which are partners.
Note that if $V_+$ and $V_-$ are partners, then 
one is perfect if and only if the other one is, because they lie on the same loop level set. The next two facts are generalizations of results from \cite{MS} about Sol, and, in particular, Corollary 8 proves half of our main conjecture.

\begin{theorem}
If $V_+$ and $V_-$ are perfect partners, then
$E(V_+)=E(V_-)$.
\end{theorem}
\begin{proof}
Let $\lambda_{\pm}$ be the trajectory which
makes one circuit around the loop level set starting
at $U_{\pm}$.  As above, we write the flowlines as
$\lambda_+=a|b$ and
$\lambda_-=b|a$.  Since $V_+$ and $V_-$ are partners, we can take
$a$ and $b$ both to be symmetric.
But then the elements
$L_a$, $L_b$, $L_{\lambda_1}$, $L_{\lambda_2}$
all preserve the plane $Z=0$ and hence mutually commute, by Lemma 16.
By Lemma 19, we have
$L_{\lambda_+}=L_{\lambda_-}$.
But $E(V_{\pm})=L_{\lambda_{\pm}}$.
\end{proof}

\begin{corollary}
A large geodesic segment is not a length minimizer.
\end{corollary}
\begin{proof}
If this is false then, by shortening our geodesic, we
can find a perfect geodesic segment $\gamma$, corresponding
to a perfect vector $V=(x,y,z)$, which is a unique
geodesic minimizer without conjugate points.  If $z \not =0$
we immediately contradict Theorem 9. 
If $z=0$, we consider the variation,
$\epsilon \to \gamma(\epsilon)$ through same-length
perfect geodesic segments
$\gamma(\epsilon)$ corresponding to
the vector $V_{\epsilon}=(x_{\epsilon},y_{\epsilon},\epsilon)$.
The vectors $V_{\epsilon}$ and $V_{-\epsilon}$ are partners, so
$\gamma(\epsilon)$ and $\gamma(-\epsilon)$
have the same endpoint. Hence, this variation
corresponds to a conjugate point on $\gamma$
and again we have a contradiction.
\end{proof}
The next step is to analyze what happens for small and perfect geodesics. To begin, we point out another consequence of Theorem 9. Let $M$ be the set of vectors in the Lie algebra of $G_\alpha$ associated to small geodesic segments and let $\partial M$ be the set of vectors associated to perfect geodesic segments. Lastly, let $\partial_0 M$ be the intersection of $\partial M$ with the plane $Z=0.$ Since $E$ identifies perfect partner vectors, we have a vanishing Jacobi field at each point of $\partial_0 M.$ However, we still have:
\begin{proposition}
$dE$ is nonsingular in $\partial M - \partial_0 M$ 
\end{proposition}
Another useful consequence of what we have heretofore shown is that the Holonomy function (defined in Lemma 19) is monotonically increasing. The following result will be useful later, when we analyze the behavior of $E$ on the set of perfect vectors. More precisely, we have:
\begin{proposition}
Let $P$ be the unique period associated to a flowline $\lambda$. We know that the holonomy $H_\lambda$ is an invariant of the flowline, so it is a function of $P$. Moreover, $H_\lambda$ varies monotonically with the flowlines. Explicitly, we have $\frac{dH}{dP}(P)>0$.
\end{proposition}
To begin our analysis of the small geodesic segments, we prove an interesting generalization of the \textit{Reciprocity Lemma} from \cite{MS}. If $V$ is a perfect vector, then $E(V)$ will lie in the $Z=0$ plane by Lemma 18; however, we can get something better.
\begin{theorem}
Let $V=(x,y,z)$ be a perfect vector. There exists a number $\mu\neq 0$ such that $E(V)=\mu(\alpha y,x,0).$
\end{theorem}
\subsection{Symmetric Flowlines}
We now introduce another technique introduced first in \cite{MS}: the emphasis on symmetric flow lines, which is justified by our previous lemmas. We introduce the following sets in $\frak{g}_\alpha$ and $G_\alpha$:
\begin{itemize}
\item Let $M, \partial M \subset \frak{g}_\alpha$ be the set of small and perfect vectors, as previously defined.

\item Let $\Pi$ be the $XY$ plane in $\frak{g}_\alpha$ and $\tilde{\Pi}$ be the $XY$ plane in $G_\alpha$.

\item Let $\partial_0 M=\partial M \cap \Pi$.

\item Let $M^{symm} \subset M$ be those small vectors which correspond to symmetric flowlines.

\item Let $\partial_0 N=E(\partial_0 M)$.

\item Let $\partial N$ be the complement, in $\tilde{\Pi}$, of the
  component of $\tilde{\Pi}-\partial_0 N$ that contains the origin.
  
\item Let $N=G_\alpha-\partial N$.

\item For any set $A$ in either $\frak{g}_\alpha$ or $G_\alpha,$ we denote $A_+$ to be the elements of $A$ in the positive sector, where $x,y>0$.
\end{itemize}
Our underlying goal is to show that $\partial N$ is the cut locus of the origin in $G_\alpha$. This has already been done for Sol ($G_1$) in \cite{MS}, and we shall prove the same for $G_{1/2}$. Thus, although the notation we use here is suggestive of certain topological relationships (e.g. is $\partial N$ the topological boundary of $N$?), we are only able to prove these relationships for $G_{1/2}$ in the present paper. Reflections across the $XZ$ and $YZ$ planes are isometries in every $G_\alpha$, so proving something for the positive sector (where $x,y>0$) proves the same result for every sector. This is useful in simplifying many proofs. 

A first step towards proving that the cut locus is $\partial N$ is to show that $$E(M) \cap \partial N =\emptyset,$$ or, intuitively, that the exponential map "separates" small and perfect vectors. The following lemma is a step towards this.

\begin{lemma}
If $E(M) \cap \partial N \not = \emptyset$, then
$E(M_+^{symm}) \cap \partial N_+ \not = \emptyset$.
\end{lemma}

With this lemma in hand, we should analyze the symmetric flowlines in detail in order to prove that $E(M_+^{symm}) \cap \partial N_+= \emptyset$.  Symmetric flowlines are governed by a certain system of nonlinear ordinary differential equations.  Let $\Theta_P^+$ denote those points in the
(unique in the positive sector) loop level set of period $P$ having
all coordinates positive. Every element
of $M_+^{{\rm symm\/}}$ corresponds to a small symmetric
flowline starting in $\Theta_P^+$.
\newline
\newline
{\bf The Canonical Parametrization:\/}
The set $\Theta_P^+$ is an open arc. 
We fix a period $P$ and we set $\rho=P/2$.
Let $p_0=(x(0),y(0),0) \in \Theta_P \cap \Pi$ be the point with $x(0)>y(0)$. The initial value $x(0)$ varies from $\sqrt{(\alpha+1)/\alpha}$ to $1$.
We then let 
\begin{equation}
p_t=(x(t),y(t),z(t))
\end{equation}
be the point on $\Theta_P^+$ which we reach after time $t \in (0,\rho)$ by 
flowing {\it backwards\/} along the structure field
$\Sigma$.   That is
\begin{equation}
\label{backwards}
\frac{dp}{dt}=
(x',y',z')=-\Sigma(x,y,z)=(-xz,+\alpha yz,x^2-\alpha y^2).
\end{equation}
Henceforth, we use the notation $x'$ to stand for $dx/dt$, etc.
\newline
\newline
{\bf The Associated Flowlines:\/} 
We let $\hat{p}_t$ be the partner of $p_t$, namely
\begin{equation}
\hat{p}_t=(x(t),y(t),-z(t)).
\end{equation}
 We let $\lambda_t$ be the small symmetric
flowline having endpoints $p_t$ and $\hat{p}_t$.
Since the structure field $\Sigma$
points downward at $p_0$, the symmetric flowline $\lambda_t$
starts out small and increases all the way to a perfect
flowline as $t$ increases from $0$ to $\rho$.
We call the limiting perfect flowline $\lambda_{\rho}$.
\newline
\newline
\noindent
{\bf The Associated Plane Curves:\/}
Let $V_t \in M_+^{{\rm symm\/}}$ be the vector
corresponding to $\lambda_t$.  (Recall that $E(V_t)=L_{\lambda_t}$)
Define
\begin{equation}
\Lambda_P(t):=E(V_t)=(a(t),b(t),0) \hskip 30 pt t \in (0,\rho].
\end{equation}
These plane curves are in $\tilde{\Pi}$ because they are endpoints of symmetric flowlines, and they will be among our main objects of interest in what follows. In Figure 4, we present a collection of the plane curves (colored blue) for $\alpha=1/2$ with the choice of $x(0)$ varying from $0.6$ to $0.95$ at intervals of $0.05$. We also include the initial value $x_0=1/\sqrt{3}$, which corresponds to the straight geodesic segment in $G_{1/2}$. The black curve is an approximation of $\partial_0 N_+$, or endpoints of perfect flowlines, which are the right-hand endpoints of each $\Lambda_P$ curve.
\begin{figure}[h!]
\centering
\includegraphics[width=1\textwidth]{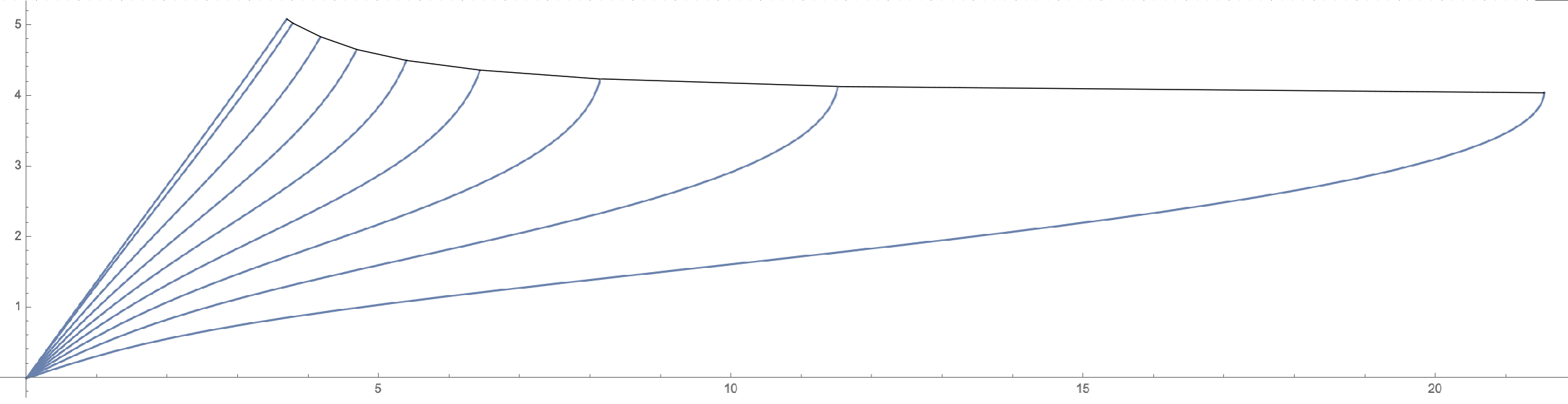}
\caption{The image of $\Lambda_P$ over the interval $(0,\rho]$ for varying $x_0$ and $\partial_0 N_+$. }
\end{figure}
\begin{lemma}
\label{endpoint}
  $\Lambda_P(\rho) \in \partial_0 N_+$, and $0<b(\rho)<a(\rho)$.
\end{lemma}

We have $E(M_+^{{\rm symm\/}}) \cap \partial N_+=\emptyset$ provided that
\begin{equation}
\label{goal}
\Lambda_P(0,\rho) \cap  \partial N_+=\emptyset,
\textrm{ for all periods } P.
\end{equation}
So all we have to do is establish Equation 25. Let $B_P$ be the rectangle in the $XY$ plane with vertices 
$$(0,0,0), (0,b(\rho),0), (a(\rho),0,0), \textrm{ and } (a(\rho),b(\rho),0).$$ 
Our first step in proving Equation 25 is to contain the image of $\Lambda_P$ with the following theorem, which we will prove to be true for each $G_\alpha$ group:
\begin{theorem}[The Bounding Box Theorem]
$\Lambda_P(0,\rho) \subset {\rm interior\/}(B_P)$ for all $P$.
\end{theorem}
The Bounding Triangle Theorem serves a similar role in \cite{MS} for Sol $(G_1)$, but it cannot be generalized to any other $G_{\alpha}$ group. It states that $\Lambda_P(0,\rho)$ is contained inside the  triangle with vertices $(0,0,0), (a(\rho),0,0)$, and $(a(\rho),b(\rho),0)$. In Figure 6, we depict the image of a single plane curve $\Lambda_P$ for $\alpha=1/2$ and $x_0=0.99945$, which illustrates the failure of the Bounding Triangle Theorem in the other Lie groups.

Now, if we could also manage to show ${\rm interior\/}(B_P)\cap \partial N_+ =\emptyset$, we would finish proving Equation 25. Since the Bounding Box Theorem is not as powerful as the Bounding Triangle theorem of \cite{MS}, we need more information about $\partial N_0$ to prove Equation 25 than was needed in \cite{MS}. We succeed in performing this second step for the group $G_{1/2}$ by getting bounds on the derivative of the period function (using its expression in terms of an elliptic integral in that case). The necessary ingredient that we get is
\begin{theorem*}[The Monotonicity Theorem]
For $\alpha=1/2$, $\partial_0 N_+$ is the graph of a non-increasing function (in Cartesian coordinates).
\end{theorem*}
\subsection{Proof of the Main Results for $G_{1/2}$}
For $G_{1/2}$, assuming that the Bounding Box and Monotonicity Theorems are true, we can proceed to characterize the cut locus of the identity.

First, we prove equation $(25)$:
\begin{theorem}
For the group $G_{1/2}$ we have, for all $P$,
$$\Lambda_P(0,\rho) \cap \partial N_+ = \emptyset$$
\end{theorem}
\begin{proof}
By the Bounding Box Theorem we know that $\Lambda_P(0,\rho) \subset {\rm interior\/}(B_P)$ for all $P$. By the Monotonicity Theorem, we know that $\partial_0 N_+$ is the graph of a decreasing function in Cartesian coordinates, so we conclude that $\partial N_+$ is disjoint from ${\rm interior\/}(B_P)$ for all $P$. Our desired result holds. 
\end{proof}
The above theorem, combined with Lemma 20 gets us:
\begin{corollary}
\label{smallperfect} 
$$E(M) \cap \partial N=\emptyset$$
\end{corollary}
The rest of our argument for showing that the cut locus of $G_{1/2}$ is $\partial N$ follows exactly as in \cite{MS}.
Let $E$ be Riemannian exponential map.
Let $\cal M$ be the component
of $\partial M_+-\partial_0M_+$ which contains
vectors with all coordinates positive.
Let ${\cal N\/}=\partial N_+ - \partial_0 N_+$. We first prove a few lemmas. 
\begin{lemma}
The map $E$ is injective on $\mathcal{M}$.
\end{lemma}
\begin{proof}
Let $V_1$ and $V_2$ be two vectors in $\mathcal{M}$ such that $E(V_1)=E(V_2)$. We also let $U_1=E(V_1)$ and $U_2=E(V_2)$ and denote the $j^{th}$ coordinate of $U_i$ as $U_{ij}$ and likewise for $\frac{V_i}{\|V_i\|}$. Since $U_1$ and $U_2$ have the same holonomy invariant and since the holonomy is monotonic with respect to choice of flowline, it follows that $\frac{V_1}{\|V_1\|}$ and $\frac{V_2}{\|V_2\|}$ lie on the same loop level set in $S(G_{1/2})$. Thus, $V_{11}V_{12}^2=V_{21}V_{22}^2$. By the Reciprocity Lemma, and since $U_1=U_2$, we get 
$$\frac{V_{12}}{V_{11}}=\frac{V_{22}}{V_{21}}.$$ 
We can now conclude that $V_{11}=V_{21}$ and $V_{12}=V_{22}$. Since $\|V_1\|=\|V_2\|$, we get $V_1=V_2$, finishing the proof.
\end{proof}
\begin{lemma}
  \label{SP0}
  $E({\cal M\/}) \subset \cal N$.
\end{lemma}
\begin{proof}
The map $E$ is injective
on ${\cal M\/} \cup \partial_0 M_+$, by the previous lemma. At the same
time, $E(\partial_0 M_+)=\partial_0 N_+$.
Hence
\begin{equation}
  \label{alternative}
  E({\cal M\/}) \subset \Pi - \partial_0 N_+.
\end{equation}
By definition, $\cal N$ is one of the components of the
$\Pi-\partial_0 N_+$.  Therefore,
since $\cal M$ is connected, the image
$E({\cal M\/})$ is either contained in $\cal N$ or disjoint from $\cal N$.
Since the sets are evidently not disjoint (large perfect vectors land far away from the identity and near the line $x=y/\sqrt{2}$), we have containment. 
\end{proof}

\begin{corollary}
\label{SP}
$E(\partial M) \cap E(M)=\emptyset$.
\end{corollary}
\begin{proof}
Up to symmetry, every vector in $\partial M$ lies either in
$\cal M$ or in $\partial_0 M_+$.
By definition, $E(\partial_0 M)=\partial_0 N \subset \partial N$.  So,
by the previous result, we have
$E(\partial M) \subset \partial N$.
By Corollary \ref{smallperfect} we have
$E(M) \cap \partial N=\emptyset$.
Combining these two statements gives the result.
\end{proof}

\begin{theorem}
  \label{minimi}
  Perfect geodesic segments are length minimizing.
\end{theorem}
\begin{proof}
Suppose $V_1 \in \partial M$ and
$E(V_1)=E(V_2)$ for some $V_2$ with $\|V_2\|< \|V_1\|$.
By symmetries of $G_{1/2}$ and the flowlines, we can assume that both $V_1$ and $V_2$ are in the positive sector and that their third coordinates are also positive.
By Corollary 8, we have $V_2 \in M \cup \partial M$.
By Corollary \ref{SP} we have $V_2 \in \mathcal{M}$.
But then $V_1=V_2$, by Lemma 22, which contradicts $\|V_2\|< \|V_1\|$.
\end{proof}
The results above identify $\partial N$ as the cut locus of the
identity of $G_{1/2}$ just as obtained in \cite{MS} for Sol. We can summarize by saying 
\begin{theorem}
A geodesic segment in $G_{1/2}$ is a length
minimizer if and only if it is small or perfect.
\end{theorem}
In addition, small geodesic segments are
unique length minimizers and they have no conjugate points.
Hence, using standard results about the cut locus, as in \cite{N}, we get that $E: M \to N$
is an injective, proper, local diffeomorphism. This implies
that $E: M \to N$ is also surjective and hence a diffeomorphism.
Moreover, $E: \partial_0 M_+ \to \partial_0 N_+$ is a diffeomorphism, by similar considerations. Results about the geodesic spheres in $G_{1/2}$ follow immediately, as in \cite{MS} for Sol, by "sewing up" $\partial M$ in a 2-1 fashion with $E$.
In particular, we have:
\begin{corollary}
Geodesic spheres in $G_{1/2}$ are always topological spheres.
\end{corollary} 

The rest of this chapter is devoted to proving our technical results: the Bounding Box Theorem and the Monotonicity Theorem. We will prove the Bounding Box Theorem in full generality, i.e. for all $\alpha\in (0,1]$. However, we only manage to prove the Monotonicity Theorem for $G_{1/2}$, where we have an expression of the period in terms of an elliptic integral. It is our opinion that either an expression for $P$ in terms of hypergeometric functions exists for general $\alpha$ or a thorough analysis of the (novel?) integral function in Proposition 9 can be done to demonstrate the monotonicity results required. Regardless, our Bounding Box Theorem does half of the work necessary to finish the proof of our main conjecture: for all $G_\alpha$ groups, a geodesic segment is length minimizing if and only if it is small or perfect. 

We reiterate that the necessary step to prove our conjecture is to show the Monotonicity Theorem holds for general $G_\alpha$ and that there is encouraging numerical evidence supporting this proposition. We plan to investigate this last step and prove our main conjecture in the future.
\begin{center}
\begin{figure}[H]
\vspace{1.5in}
\includegraphics[width=1\textwidth]{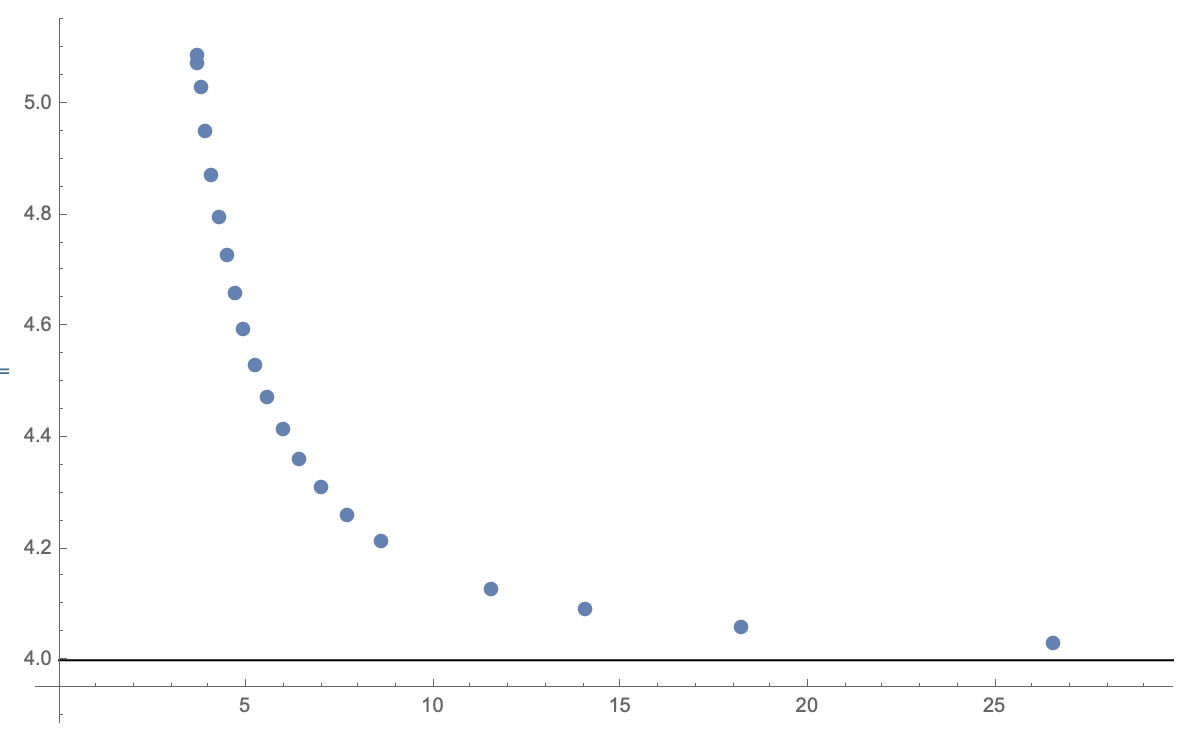}
\caption{Here, for $G_{1/2}$, we have plotted points on $\partial_0 N_+$, as $x_0$ varies from $0.6$ to $0.98$ in increments of $0.02$.}
\end{figure}
\end{center}
\begin{figure}[H]
\centering
\includegraphics[width=\textwidth]{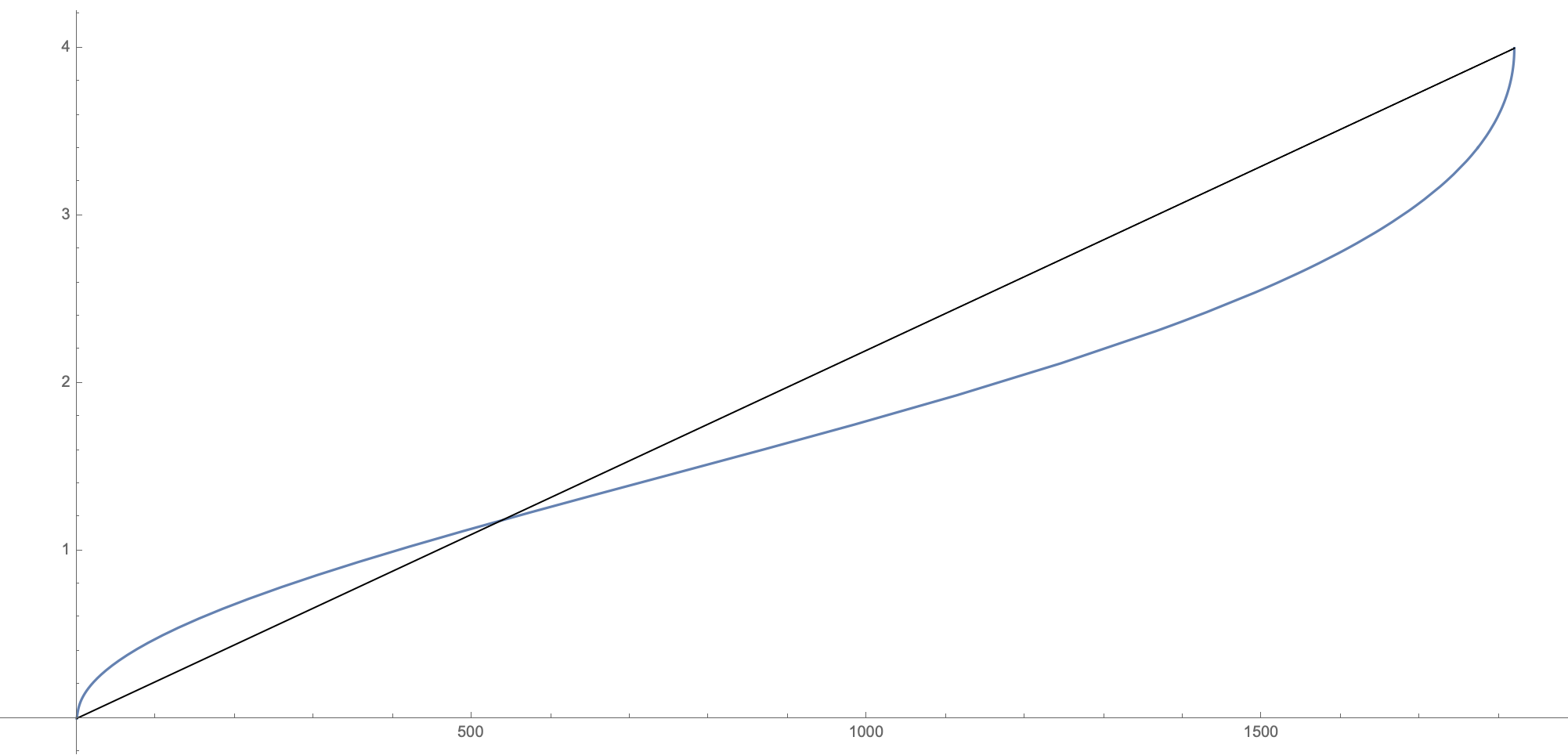}
\caption{This depicts a the image of $\Lambda_P$ for $\alpha=1/2$ and $x_0=0.99945$ over the interval $(0,\rho]$.}
\end{figure} 
\section{Proof of the Bounding Box Theorem}
We now study the system of ODE's that governs the behavior of $x,y,z,a,$ and $b$ as in equation $(22)$.
We write 
$\lambda_{t+\epsilon}=u|\lambda_t|v$,
where $u$ is the flowline
connecting $p_{t+\epsilon}$ to $p_t$ and $v$ is the flowline 
connecting $\hat{p}_t$ to $\hat p_{t+\epsilon}$.
We have
$$
(a',b',0)=\Lambda'_P(t)=\lim_{\epsilon \to 0} \frac{\Lambda_P(t+\epsilon)-\Lambda(t)}{\epsilon},
$$
$$
  \Lambda_P(t+\epsilon) \approx (\epsilon x, \epsilon y,\epsilon z) *
  (a,b,0) * ( \epsilon x, \epsilon y,-\epsilon z).
$$
The approximation is true up to order $\epsilon^2$ and
$(*)$ denotes multiplication in $G_\alpha$.
A direct calculation gives
\begin{equation} a'=2x+az \textrm{ and  } b'=2y-\alpha bz.\end{equation}

Simply from its differential equation, it is evident that $a'>0$ on $(0,\rho)$. This implies that
$\Lambda_P(t)$ is the graph of a function for $t\in (0,\rho)$, hence $\Lambda_P(t)$
avoids the vertical sides of $B_P$. This is the first, easy step in proving the Bounding Box Theorem.

To finish the proof, it would suffice to show that $\Lambda_P(t)$ also avoids the horizontal sides of $B_P$, which amounts to proving that $b'(t)>0$ for all $t\in(0,\rho]$. A priori, it is not evident that $b'>0$ in this interval. For example, the function $b$ may start out concave. Also, after the half-period $\rho$, $b'$ may actually be negative. However, the remarkable fact that $b'>0$ in $(0,\rho]$ for all choices of $\alpha$ and $x_0$ is also true, and we demonstrate this fact in what follows.

Once again, we collect the ODE's of interest to us together:
$$x'=-xz\quad y'=\alpha y z\quad z'=x^2-\alpha y^2\quad b'=2y-\alpha b z$$ from which we compute
\begin{equation}z''=-2z(x^2+\alpha y^2)\quad b''=\alpha b(\alpha z^2-z').\end{equation}
\begin{lemma}
To show that $b'>0$ in $(0,\rho),$ it suffices to show that $b'>0$ whenever $b''<0$ in the interval $(0,\rho).$ 
\end{lemma}
We observe that $z''<0$ everywhere in $(0,\rho)$. Also, whenever $b''<0,$ we have that $z'>\alpha z^2>0$. We first get an inequality regarding the function $z(t):$ 
\begin{lemma}
$$2\alpha \int_0^t z(s) ds\geq t\alpha z(t)$$
\end{lemma}
We now recall the Log-Convex Version of Hermite-Hadamard, proven first in \cite{GPP}:
\begin{proposition}[\cite{GPP}]
If $f$ is log-convex on $[a,b]$ then
$$\frac{1}{b-a}\int_a^b f(s) ds\leq \frac{f(b)-f(a)}{\log f(b)-\log f(a)}.$$
\end{proposition}
Let's return to one of our initial ODE's: $y'=\alpha yz.$ Dividing by $y,$ integrating, and multiplying by 2, we get:
\begin{equation}2\log y(t) -2\log y(0)= 2\alpha \int_0^t z(s) ds. \end{equation}
Since $z'>0$ whenever $b''<0,$ we get that $y^2$ is log-convex whenever $b''<0.$ We can now get:
\begin{lemma} For all $t$ where $b''(t)<0,$ we have
$$\int_0^t y(s)^2 ds \leq t\cdot\frac{y(t)^2}{2\log y(t)/y(0)}$$
\end{lemma}
We are now in a position to prove the Bounding Box Theorem:
\begin{proof}
By Lemma 24, it suffices to show that $b'>0$ whenever $b''<0.$ If we integrate the ODE for $b$, we get
\begin{equation} b(t)=\frac{2}{y(t)}\int_0^t y(s)^2 ds.\end{equation}
Differentiating, we want to show $y(t)^3-y'(t)\int_0^t y(s)^2 ds\geq 0$ whenever $b''<0.$ Equivalently (we can divide by $y,y'$ since they are always strictly greater than $0$):
$$ \int_0^t y(s)^2 ds\leq \frac{y(t)^3}{y'(t)}.$$
By Lemma 26 it suffices to show
$$t\cdot\frac{y(t)^2}{2\log y(t)/y(0)}\leq\frac{y(t)^3}{y'(t)}$$
or, by cancelling some terms and taking the reciprocal,
$$2\log y(t)/y(0) \geq t\cdot \frac{y'(t)}{y(t)}.$$
Since $y'=\alpha yz$, this is equivalent to
$$2\alpha \int_0^t z(s) ds\geq t\alpha z(t)$$
which is nothing but the inequality of Lemma 25.
\end{proof}
Finally, since $b'>0, \Lambda_P(t)$ is an increasing function, so $\Lambda_P(t)$ avoids the vertical sides of $B_P$ when $t\in(0,\rho)$. This, along with the previously stated fact that $\Lambda_P(t)$ avoids the horizontal sides of $B_P$, finishes the proof of the Bounding Box Theorem. 
\section{Proof of the Monotonicity Theorem}
\subsection{Endpoints of Symmetric Flowlines}
Henceforth, we view $\partial_0 N_+$ as the following parametrized curve in the $XY$ plane. Denote $x(0)=x_0$, then
$$\partial_0 N_+= \{ (a_{x_0}(P(x_0)/2),b_{x_0}(P(x_0)/2),0)\}, \textrm{ as } x_0 \textrm{ varies in } \bigg(\sqrt{\frac{\alpha}{1+\alpha}},1\bigg).$$
To prove Equation $25$ in general it would suffice, by using the Bounding Box Theorem, to prove $B_P\cap \partial_+ N=\emptyset$, and, to prove the latter statement, it suffices to show that $\partial_0 N_+$ is the graph of a decreasing function in Cartesian coordinates. This involves differentiating our ODE's with respect to the initial value $x_0$. Let $\bar{x}$ denote $dx(t,x_0)/dx_0$, etc. Then we get
$$\bar{x}'=-x\bar{z}-z\bar{x},\quad \bar{y}'=\alpha y\bar{z}+\alpha z\bar{y},\quad \bar{z}'=2x\bar{x}-2\alpha y \bar{y},$$
$$\bar{a}'=2\bar{x}+a\bar{x}+x\bar{a}, \quad \bar{b}'=2\bar{y}-\alpha\bar{y}b-\alpha y\bar{b}.$$
Since $x^2+y^2+z^2=1$ for all $t, x_0,$ it follows that 
\begin{equation}x\bar{x}+y\bar{y}+z\bar{z}=0\end{equation}
always, and we can get a similar equation for the time derivative. Now, we prove some very useful propositions:
\begin{proposition}
$ax-\alpha by = 2z$ for all $t$ and $x_0$.
\end{proposition}
Since the above equality is true for all $t$ and $x_0$ we can differentiate with respect to $x_0$ and get
\begin{corollary}
$a\bar{x}+x\bar{a}-\alpha b\bar{y}-\alpha y \bar{b}=2\bar{z}$ for all $t$ and $x_0$.
\end{corollary}
Now we prove:
\begin{proposition}
$x\bar{a}+y\bar{b}=0$ for all $t$ and $x_0$.
\end{proposition}
Corollary 12 and Proposition 14 combine to get us the following useful expressions for $\bar{a}$ and $\bar{b}$. 
\begin{corollary}
We have
$$\bar{a}=\frac{1}{x(1+\alpha)}\bigg(2\bar{z}+\alpha b\bar{y}-a\bar{x}\bigg)$$
and
$$\bar{b}=-\frac{1}{y(1+\alpha)}\bigg(2\bar{z}+\alpha b\bar{y}-a\bar{x}\bigg).$$
\end{corollary}
\subsection{The Monotonicity Theorem for $G_{1/2}$}
To get our desired results about $\partial_0 N_+$, we must look at a particular case of our one-parameter family, where we have more information about the derivative of $P(x_0)$, courtesy of the expression of $P$ in terms of an elliptic integral. Everything we have heretofore shown is, however, applicable to every $G_\alpha$ with $0<\alpha\leq 1$.
Although we restrict ourselves to $G_{1/2}$, the methods presented here could just as well be applied to $G_1$, where we also have an explicit formula for the period function. 

In this section, we show that $\partial_0 N_+$ is the graph of a non-increasing function in Cartesian coordinates (for $G_{1/2}$) by using properties of the period function. This finishes the proof of Equation 25, which in turn allows us to prove our main theorem. For encouragement, we refer the reader back to Figure 5, where we can see that $\partial_0 N_+$ is indeed the graph of a non-increasing function in Cartesian coordinates. It will be relatively easy to show that $\partial_0N_+$ is the graph of a Cartesian function, but the proof that $\partial_0N_+$ is a non-increasing function will be more involved. For example, we will first need to show that $\partial_0N_+$ limits to the line $b=4$ as $x_0\to 1$. 

 Recall that $P(\beta)$ is decreasing with respect to $\beta$ in the case when $\alpha=1$ or $1/2$. Here, we change variables for the period function from $\beta$ to $x_0\in (1/\sqrt{3},1)$ and get:
\begin{proposition}
$$\frac{d}{dx_0}\big(P(x_0)\big) >0$$
\end{proposition}
Since $z$ always vanishes at the half period, we have
\begin{proposition}
For any initial value $x_0$, we have
$$\bar{z}+(\frac{1}{2}\frac{dP}{dx_0})z'=0$$
at the time $t=P(x_0)/2$. Also, since $z'<0$ at the half period, we get that 
$$\bar{z}(P(x_0)/2)>0, \quad \forall x_0.$$
\end{proposition}
From equation $(31)$ and the fact that $x'$ and $y'$ always vanish at the half-period, we have
\begin{proposition}
$$\bar{y}(P(x_0)/2)=\frac{1}{2\sqrt{x_0}}>0 \quad \textrm{ and }  \quad \bar{x}(P(x_0)/2)=-2x_0<0$$
\end{proposition}
We are ready to get some information about $\partial_0 N_+$, beginning with:
\begin{corollary} 
$\partial_0 N_+$ is the graph of a function in Cartesian coordinates.
\end{corollary}
\begin{proof}
This is equivalent to showing that 
$$\frac{d}{dx_0}a_{x_0}(P(x_0)/2)>0.$$
The chain rule gets us
$$\frac{d}{dx_0}a_{x_0}(P(x_0)/2)=\bar{a}(P(x_0)/2)+(\frac{1}{2}\frac{dP(x_0)}{dx_0})a'(P(x_0)/2)$$
$$=\bigg(\frac{2}{3x}\bigg(2\bar{z}+\frac{1}{2}b\bar{y}-a\bar{x}\bigg)+(x+az/2)\frac{dP}{dx_0}\bigg)\bigg\rvert_{t=P(x_0)/2}$$
By the previous propositions, we know that all the terms above are positive at $P(x_0)/2$, whence the desired result.
\end{proof}
As with showing that $b'>0$ in the interval $(0,\rho)$, things are more difficult with the function $b$. We need three lemmas first. The proof of the following may also suggest that $\alpha=1/2$ is a "special case"; nevertheless, the situation is different than for Sol. Richard Schwartz proves a similar limit in \cite{S} for Sol (in which case, the limit is $2$), but his method uses an additional symmetry of the flow lines that we cannot use here. 
\begin{lemma}For $\alpha=1/2$,
$$\lim_{x_0\to 1} b_{x_0}(P(x_0)/2)=4$$
\end{lemma}
More generally, we have the following conjecture
\begin{conjecture} Let
$$L(\alpha):=\lim_{x_0\to 1} b_{x_0,\alpha}(P(x_0,\alpha)/2), \textrm{ defined for all } \alpha \in (0,1].$$
We conjecture that $L(\alpha)$ is monotonically decreasing from $\alpha=0$ to $\alpha=1$ and that $\lim_{\alpha\to 0} L(\alpha)=\infty$. In fact, we also conjecture that 
$$L(\alpha)=\frac{2}{\alpha}.$$
\end{conjecture}
The next technical lemma is quite wearisome to prove. We direct the reader to Figure 7, which provides numerical evidence for its veracity.
\begin{lemma}
For $\alpha=1/2,$ we have
$$G(x_0):=\frac{dP}{dx_0}(x_0)-\pi\bigg(\frac{1}{2\sqrt{x_0}}+\frac{2x_0\sqrt{x_0}}{1-x_0^2}\bigg)<0,\quad \forall x_0 \in \bigg(\frac{1}{\sqrt{3}},1\bigg)$$ 
\end{lemma}
\begin{figure}[H]
\centering
\includegraphics[width=\textwidth]{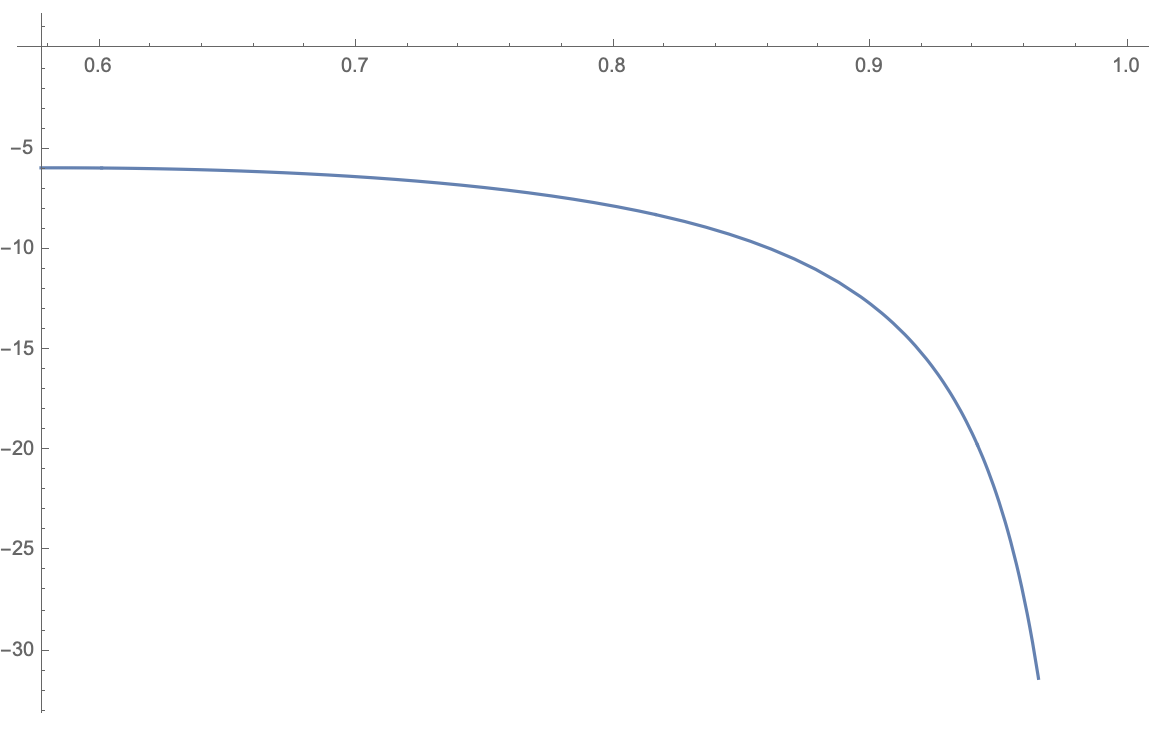}
\caption{The graph of the function $G(x_0)$.}
\end{figure}
The penultimate step towards the Monotonicity Theorem:
\begin{lemma}
$$\frac{d}{dx_0}b_{x_0}(P(x_0)/2)<0, \textrm{ whenever } b_{x_0}(P(x_0)/2)>\pi$$
\end{lemma}
\begin{proof}
By the chain rule we know
$$\frac{d}{dx_0}b_{x_0}(P(x_0)/2)=\bigg(-\frac{2}{3y}\bigg(2\bar{z}+\frac{1}{2}b\bar{y}-a\bar{x}\bigg)+(y-\frac{1}{4}bz)\frac{dP}{dx_0}\bigg)\bigg\rvert_{t=P(x_0)/2}.$$
The $bz$ term is zero at the half period, and we can simplify further. The above is less than zero if and only if
$$(2x^2+2y^2)\frac{dP}{dx_0}<b\bar{y}-2a\bar{x} \textrm{ at } t=P(x_0)/2.$$
However we know $x^2+y^2=1$ at the half period, and we may also employ the Reciprocity Lemma. This yields that our desired result is equivalent to showing
$$\frac{dP}{dx_0}(x_0)<\big(b(\bar{y}-\frac{y}{x}\bar{x})\big)\big\rvert_{t=P(x_0)/2}.$$
By hypothesis, $b_{x_0}(P(x_0)/2)>\pi$, and we can use equation $(31)$. This means it suffices to show
$$\frac{dP}{dx_0}(x_0)<\pi\bigg(\frac{1}{2\sqrt{x_0}}+\frac{2x_0\sqrt{x_0}}{1-x_0^2}\bigg),\quad \forall x_0 \in \bigg(\frac{1}{\sqrt{3}},1\bigg)$$
which is nothing but Lemma 28.
\end{proof}
Now we are ready to prove the Monotonicity Theorem.
\begin{theorem}[The Monotonicity Theorem]
For $\alpha=1/2$, $\partial_0 N_+$ is the graph of a non-increasing function (in Cartesian coordinates).
\end{theorem}
\begin{proof}
This is equivalent to showing 
$$\frac{d}{dx_0}b_{x_0}(P(x_0)/2)\leq0$$
or, as in the proof of the previous lemma,
$$\frac{dP}{dx_0}(x_0)\leq\big(b(\bar{y}-\frac{y}{x}\bar{x})\big)\big\rvert_{t=P(x_0)/2}.$$
We now show that $b_{x_0}(P(x_0)/2)\geq4$ always. If this is true, then the theorem follows by Lemma 29. Otherwise, assume that  $b_{x_0}(P(x_0)/2)<4$ for some choice of $x_0$. By Lemma 27, $b_{x_0}(P(x_0)/2)$ limits to $4$ as $x_0$ tends to $1$, so $b_{x_0}(P(x_0)/2)$ must eventually be greater than $\pi$. Let $x_0'$ be the last initial value where $b(P/2)\leq \pi$. By Lemma 29, $b_{x_0}(P(x_0)/2)$ will be strictly decreasing in $x_0$, which is a contradiction. Hence, $b>4$ always and the theorem is proven. 
\end{proof}
\section{Computer Code}
Here we present the Mathematica code that generates the figures we have presented. In addition, we present the Mathematica code written by Stephen Miller that allows numerical computation of the period function for arbitrary positive $\alpha$. We note that the whole of the program \cite{S2}, used to generate Figure 3, is available online on Richard Schwartz's website.
\subsection{Figures 1 and 2}
This gets us the implicit plots in the plane:
\begin{lstlisting}[language=Mathematica]
Manipulate[
 ContourPlot[(1/a) E^(2*a*z) + E^(-2 z) + 
    w^2 == (1+a)/(a*b^2), {w,-10,10}, {z,-10,10}, 
  PlotRange -> All, AspectRatio -> Automatic, PlotPoints -> 50], 
  {b,0.001,1}, {a,0.001,1}]
\end{lstlisting}
A 3-D rendering can be obtained with:
\begin{lstlisting}[language=Mathematica]
Manipulate[
 ContourPlot3D[(1/a) E^(2*a*z) + 
    E^(-2 z) + (x - Sqrt[a]*y)^2 == (1 + a)/(a*b^2), {x, -10, 
   10}, {y,-10,10}, {z,-10,10}, PlotRange -> All], {b,0.01, 
  1}, {a,0.05,1}]
\end{lstlisting}
\newpage
\subsection{Figure 4 and 5}
The following renders the symmetric flowline associated to any choice of $x_0 \in (\frac{1}{\sqrt{3}}, 1)$ for the $G_{1/2}$ group. Moreover, this provides numerical evidence that the unwieldy expression for the period function in $G_{1/2}$ we presented earlier is indeed correct. Figure 5 just adds the straight line from the origin to the endpoint of the flowline to show that the "Bounding Triangle Theorem" does not work in general.
\begin{lstlisting}[language=Mathematica]

(*This is the period function for the G_{1/2} group, 
as in Corollary 3.4 *)

P[B_]:= (*as in the statement of the corollary *)

(*This is the period function, 
after the change of variables from \beta to x0*)

L1[x0_] := P[((3*Sqrt[3]/2)*(x0-x0^3))^(1/3)]
              
 (*This numerically solves the fundamental system of ODE's *)
 
 Manipulate[
 	s1 = NDSolve[{x'[t] == -x[t]*z[t], y'[t] == (1/2)*y[t]*z[t], 
    z'[t] == -(1/2)*y[t]^2 + x[t]^2, a'[t] == 2*x[t] + a[t]*z[t], 
    b'[t] == 2 y[t] - (1/2)*b[t]*z[t], z[0] == 0, a[0] == 0, 
    b[0] == 0, x[0] == x0, y[0] == Sqrt[1 - x0^2]}, {x, y, z, a, 
    b}, {t, 0, Re[L1[x0]]/2}], {x0, 1/Sqrt[3], 1}]
    
 (*This plots the flowline*)    
 
    ParametricPlot[{a[t], b[t]} /. s1, {t, 0,Re[L1[x0]/2]}, 
 PlotRange -> All]
\end{lstlisting}
\newpage
\subsection{Figure 6}
We can use Stephen Miller's program to compute the period for arbitrary alpha, or we can "find" it by inspection, preferably looking at the function $z$.
\begin{lstlisting}[language=Mathematica]
(*This numerically solves our ODE's, d is alpha here*)

Manipulate[
 s = NDSolve[{x'[t] == -x[t]*z[t], y'[t] == d*y[t]*z[t], 
    z'[t] == -d*y[t]^2 + x[t]^2, a'[t] == 2*x[t] + a[t]*z[t], 
    b'[t] == 2 y[t] - d*b[t]*z[t], z[0] == 0, a[0] == 0, b[0] == 0, 
    x[0] == c1, y[0] == Sqrt[1 - c1^2]}, {x, y, z, a, b}, {t, 0, 
    20}], {c1, Sqrt[d/(d + 1)], 1}, {d, 0, 1}]
    
(*We can use this to find \rho by inspection*)
    Manipulate[z[t] /. s, {t, 0,20}]

(*This plots the derivative of b*)
Plot[b'[t] /. s, {t, 0, 15.99}, AspectRatio -> Automatic, 
 AxesOrigin -> {0, 0}]
\end{lstlisting}
\subsection{Figure 7}
This code provides the numerical evidence for Lemma 28.
\begin{lstlisting}[language=Mathematica]
L1[x0_]:= (*As before*)

(*First we evaluate the derivative*)
D[L1[x0],x0]

(*Then, we plot G(x0)*)
Plot[% - 3 (1/(2 Sqrt[x0]) + 2*x0*Sqrt[x0]/(1 - x0^2)),
{x0,1/Sqrt[3.],1},AxesOrigin -> {1/Sqrt[3.], 0}]
\end{lstlisting}
\newpage
\subsection{Computing the General Period Function}
Stephen Miller helped us with this code. It allows us to compute the period function for any choice of $\alpha$ and $\beta$. Using it, it appears that the monotonicity results we would like are indeed true for arbitrary $\alpha$, a promising sign for our Main Conjecture.
\begin{lstlisting}[language=Mathematica]
(*Defining the integrand*)
integrand[t_, A_, B_] = 2/
 Sqrt[1 - B^2/(A + 
     1) (A Exp[2 t] + Exp[-2 A t])]
     
    (*Numerically finding the endpoints of integration *)
     endpoints[A_, B_] := 
 Sort[Log[Flatten[
     y /. NSolve[
       1 - B^2/(A + 1) (A y + y^-A) == 0, 
       y, 20]]]/2]
       
     (*Numerical integration*)
       p[A_, B_] := 
 NIntegrate[integrand[t, A, B], 
  Join[{t}, endpoints[A, B]]]
  
  (*Generates the table presented earlier*)
Table[{a, p[a, .999], Pi*Sqrt[2/a]}, {a, 0.1, 1, 0.1}] 
\end{lstlisting}


\begin{thebibliography}{40}

  \bibitem{AQ}
Anderson, G.D., Vamanamurthy, M.K., Vuorinen, M. (1992).
\textit{Functional Inequalities for Hypergeometric Functions and Complete Elliptic Integrals}
SIAM J. Math. Anal. Vol 23, pp. 512-524.

\bibitem{ANG}
Angenent, S. (1991).
\textit{On the Formation of Singularities in the Curve Shortening Flow}.
J. Diff. Geo. 33. 601-633.

\bibitem{VA}
  Arnold, V. (1966).
  \textit{Sur la g\'eom\'etrie diff\'erentielle des groupes de Lie de dimension infinie et ses applications \`a l'hydrodynamique des fluides parfaits.} Ann. Inst. Fourier (Grenoble)
  
\bibitem{BSW}
  Beffa, G. Mari, Sanders, J.A., Wang, Jing Ping. (2002).
  \textit{Integrable Systems in Three-Dimensional Riemannian Geometry}.
  J. Nonlinear Sci. Vol. 12, pp. 143-167.
  
  \bibitem{LB}
  Bianchi, L. (1898).
  \textit{Sugli spazi a tre dimensioni che ammettono un gruppo continuo di movimenti}. 
  Memorie di Matematica e di Fisica della Societa Italiana delle Scienze, Serie Terza, Tomo XI, pp. 267-352 
  
    \bibitem{CD}
  Calogero, F., Degasperis, A. (1985). 
  \textit{A Modified Modified Korteweg-de Vries Equation}.
 Inverse Problems. Vol. 1, pp. 57-66.
  
  \bibitem{CFA}
  Clarkson, P.A., Fokas, A.S., Ablowitz, M.J. (1989).
  \textit{Hodograph Transformations of Linearizable Partial Differential Equations}. 
  SIAM J. Appl. Math. Vol. 49, No. 4, pp. 1188-1209.
  
  \bibitem{MS}
  Coiculescu, M.P., Schwartz, R.E. (2020).
  \textit{The Spheres of Sol}.
  arXiv:1911.04003v7
  
  \bibitem{COI}
  Coiculescu, M.P. (2020).
  \textit{An Interpolation from Sol to Hyperbolic Space}.
  Preprint.
  
  \bibitem{COI2}
  Coiculescu, M.P. (2021).
  \textit{Stationary solutions of the curvature preserving flow on space curves}.
  Arch. Math. (2021). https://doi.org/10.1007/s00013-020-01563-z
  
    \bibitem{CV}
  Cveticanin, L. (2004).
  \textit{Vibrations of the nonlinear oscillator with quadratic nonlinearity}.
  Physica A, 341, 123-135.
  
  \bibitem{DHW}
  Drugan, G., He, W., Warren, M.W. (2018).
  \textit{Legendrian curve shortening in $\mathbb{R}^3$}
  Commun. Anal. Geo. Vol. 26. No. 4. pp. 759-785.
  
    \bibitem{EM}
  Ellis, G.F.R., MacCallum M.A.H. (1969).
  \textit{A Class of Homogeneous Cosmological Models}. 
  Commun. math. Phys. 12, 108-141
  
  \bibitem{F}
  Folland, G.B. (1999).
  \textit{Real Analysis}.
  J. Wiley Interscience Texts.
 
 \bibitem{GH}
 Gage, M., Hamilton, R.S. (1986).
 \textit{The Heat Equation Shrinking Convex Plane Curves}.
 J. Diff. Geo. 23. 69-96.
  
    \bibitem{GPP}
  Gill, P.M., Pearce, C.E.M., Pe\v{c}ari\'{c}, J. (1997).
  \textit{Hadamard's Inequality for $r$-Convex Functions}
  Journal of Mathematical Analysis and Applications. 215, 461-470
  
  \bibitem{GR}
  Gradshteyn, I.S., Ryzhik, I.M. (2007).
  \textit{Table of Integrals, Series, and Products, Seventh Edition}
  Academic Press, Elsevier.
  
  \bibitem{G}
  Grayson, M.A. (1987).
  \textit{The Heat Equation Shrinks Embedded Curves to Round Points}.
  J. Diff. Geo. Vol. 26, No.2.
  
   \bibitem{G2}
  Grayson, M.A. (1983).
  {\it Geometry and Growth in Three Dimensions\/}, Ph.D. Thesis,
  Princeton University.
  
  \bibitem{G3}
  Grayson, M.A. (1989).
  \textit{The Shape of a Figure-Eight under the Curve Shortening Flow}.
  Invent. Math. 96, 177-180.
  
  \bibitem{HA}
  Halldorsson, H.P. (2012).
  \textit{Self-Similar Solutions to the Curve Shortening Flow}.
  Transactions of the American Mathematical Society. Vol. 364. No. 10.
  
  \bibitem{H}
  Hasimoto, H.  (1972).
  \textit{A Soliton on a Vortex Filament}
  J. Fluid Mechanics. Vol 51, No. 3, pp. 477-485.
  
  \bibitem{HEER}
  Heredero Hernandez, R., et al. (2020).
  \textit{Compacton Equations and Integrability: The Rosenau-Hyman Equations and Cooper-Shepard-Sodano Equations}.
  Discrete and Continuous Dynamical Systems. 
  
  \bibitem{I}
  Ivey, Thomas A. (2001).
  \textit{Integrable Geometric Evolution Equations for Curves}.
  Contemporary Mathematics. Vol 285.
  
     \bibitem{K}
   Kobayashi, S., Nomizu, K. (1963).
  \textit{Foundations of Differential Geometry Volume 1}.
  Wiley Classics Library
  
  \bibitem{N}
   Kobayashi, S., Nomizu, K. (1966).
  \textit{Foundations of Differential Geometry Volume 2}.
  Wiley Classics Library
  
  \bibitem{LW}
  Lou, Sen-yue, Wu, Qi-xian. (1999).
  \textit{Painleve Integrability of Two Sets of Nonlinear Evolution Equations with Nonlinear Dispersions}.
  Physics Letters A. Vol. 262, pp. 344-349.
  
  \bibitem{MSS}
  Mikhailov, A.V., Shabat, A.B., Sokolov, V.V. (1991).
  \textit{The Symmetry Approach to Classification of Integrable Equations}.
  A chapter in \textit{What is Integrability}, pp. 115-184.
  
    \bibitem{DLMF} NIST Digital Library of Mathematical Functions. http://dlmf.nist.gov/, Release 1.0.28 of 2020-09-15. F. W. J. Olver, A. B. Olde Daalhuis, D. W. Lozier, B. I. Schneider, R. F. Boisvert, C. W. Clark, B. R. Miller, B. V. Saunders, H. S. Cohl, and M. A. McClain, eds.
    
    \bibitem{SC}
    Sarmin, E.N., Chudov, L.A. (1962).
    \textit{On the Stability of the Numerical Integration of systems of Ordinary Differential Equations arising in the use of the Straight Line Method}.
    SSR Computational Mathematics and Mathematical Physics, 3 (6): 1537-1543 
    
  \bibitem{SR}
  Schief, W.K., Rogers, C. (1999).
  \textit{Binormal Motion of Curves of Constant Curvature and Torsion. Generation of Soliton Surfaces}.
  Proc. R. Soc. Lond. A. Vol 455. pp. 3163-3188.
  
  \bibitem{S}
  Schwartz, R.E. (2020).
  \textit{On Area Growth in Sol}.
  arXiv:2004.10622v2 
  \bibitem{S2}
  Schwartz, R.E. {\it Java Program for Sol\/}, download (in 2020) from \newline
http://www.math.brown.edu/$\sim$res/Java/SOL.tar
  
  \bibitem{ST}
  Trettel, S. (2019).
  \textit{Families of Geometries, Real Algebras, and Transitions}, Ph.D. Thesis,
  University of California Santa Barbara.
  
   \bibitem{T}
 Troyanov, M. (1998).
 \textit{L'horizon de Sol}. 
 Exposition. Math. 16, no. 5, 441-479.
  
  \bibitem{WW}
  Wilczynska, M.R., Webb, J.K., et al. (2020).
  \textit{Four direct measurements of the fine-structure
constant 13 billion years ago}.
 Sci. Adv. 6.
\end{thebibliography}
\end{document}